 \documentclass[onefignum,onetabnum]{siamart190516}

\usepackage{bm} 
\usepackage{dsfont}
\usepackage{amstext}
\usepackage{amsgen}
\usepackage{amsxtra}
\usepackage{mathrsfs}
\usepackage{amssymb}
\usepackage{lipsum}
\usepackage{amsfonts}
\usepackage{graphicx}
\usepackage{algorithmic}
\ifpdf
  \DeclareGraphicsExtensions{.eps,.pdf,.png,.jpg}
\else
  \DeclareGraphicsExtensions{.eps}
\fi

\newcommand{\N}{\mathbb{N}}
\newcommand{\R}{\mathbb{R}}
\newcommand{\Z}{\mathbb{Z}}
\newcommand{\T}{\mathbb{T}}
\newcommand{\fR}{f_R}
\newcommand{\fL}{f_L}
\newcommand{\Pe}{\text{Pe}}
\newcommand{\p}{{\bf p}}

\newcommand{\e}{{\bf e}}
\newcommand{\x}{{\bf x}}

\newcommand{\X}{{\bf X}}

\newcommand{\ud}{\mathrm{d}}

\renewcommand{\tt}{\tilde \theta}

\DeclareMathOperator{\esssup}{\mathrm{esssup}}
\newcommand{\z}{\mathbf{z}}
\newcommand{\y}{\mathbf{y}}

\newsiamremark{remark}{Remark}
\newsiamremark{hypothesis}{Hypothesis}
\crefname{hypothesis}{Hypothesis}{Hypotheses}
\newsiamthm{claim}{Claim}

\headers{Integro-differential model for active particles}{M. Bruna, M. Burger, A. Esposito and S. M. Schulz}

\title{Well-posedness of an integro-differential model for active Brownian particles\thanks{
\funding{The first author was partially funded by a Royal Society University Research Fellowship (URF/R1/180040) and a Humboldt Research Fellowship from the Alexander von Humboldt Foundation. The second and the third author gratefully acknowledge support by the German Science Foundation (DFG) through CRC TR 154  ``Mathematical Modelling, Simulation and Optimization Using the Example of Gas Networks". AE was supported by the Advanced Grant Nonlocal-CPD of the ERC under the EU’s Horizon 2020 research and innovation programme (grant agreement No.~883363). The fourth author was funded by the Royal Society Award (RGF/EA/181043).}}}

\author{Maria Bruna\thanks{Department of Applied Mathematics and Theoretical Physics, University of Cambridge, Cambridge CB3 0WA, UK 
  (\email{bruna@maths.cam.ac.uk}).}
\and Martin Burger\thanks{Department Mathematik, Friedrich-Alexander Universit\"at Erlangen-N\"urnberg, Cauerstr. 11, D 91058 Erlangen, Germany 
  (\email{martin.burger@fau.de}).}
\and Antonio Esposito\thanks{Mathematical Institute, University of Oxford, Oxford OX2 6GG, UK (\email{antonio.esposito@maths.ox.ac.uk})}
\and Simon Schulz\thanks{Department of Mathematics, University of Wisconsin-Madison, Van Vleck Hall, 480 Lincoln Dr, Madison, WI 53706, USA, (\email{smschulz2@wisc.edu}).}}

\usepackage{amsopn}

\ifpdf
\hypersetup{
  pdftitle={Integro-differential model for active particles},
  pdfauthor={M. Bruna, M. Burger, A. Esposito and S. M. Schulz}
}
\fi

\begin{document}

\maketitle

\begin{abstract}
  We propose a general strategy for solving nonlinear integro-differential evolution problems with periodic boundary conditions, where no direct maximum/minimum principle is available. This is motivated by the study of recent macroscopic models for active Brownian particles with repulsive interactions, consisting of advection-diffusion processes in the space of particle position and orientation. We focus on one of such models, namely a semilinear parabolic equation with a nonlinear active drift term, whereby the velocity depends on the particle orientation and angle-independent overall particle density (leading to a nonlocal term by integrating out the angular variable).
  The main idea of the existence analysis is to  exploit a-priori estimates from (approximate) entropy dissipation. 
  The global existence and uniqueness of weak solutions is shown using a two-step Galerkin approximation with appropriate cutoff in order to obtain nonnegativity, an upper bound on the overall density and preserve a-priori estimates. Our analysis naturally includes the case of finite systems, corresponding to the case of a finite number of directions. The Duhamel principle is then used to obtain additional regularity of the solution, {namely continuity in time-space.}  
Motivated by the class of initial data relevant for the application, which includes perfectly aligned particles (same orientation), we extend the well-posedness result to very weak solutions allowing distributional initial data with low regularity.  
  \end{abstract}



\section{Introduction}

In this paper, we present a well-posedness theory (existence, uniqueness, and regularity) for an integro-differential equation describing a system of interacting active Brownian particles.
Active Brownian particles are a model system for self-propelled particles, and have wide-ranging applications such as molecular motors, bacteria, animal swarms and pedestrian flows (\textit{cf.}~\textit{e.g.}~\cite{bruna2021active,cristiani2014multiscale,gompper20202020,helbing1995social}). Denote by $\X_i\in \Omega \subset \R^2$ and $\Theta_i \in \mathbb T := \R / (2 \pi \Z)$ the position and orientation of particle number $i$, respectively. The evolution of the system with $N$ indistinguishable particles is given by the following stochastic differential equations:
\begin{subequations}\label{sde_all}
\begin{align}
\label{sde_x}
	\ud \X_i &= \sqrt{2 D_T} \ud {\bf W}_i + v_0 \e(\Theta_i) \ud t - \sum_{j\ne i} \nabla u((\X_i-\X_j)/\varepsilon)\ud t,\\
	\label{sde_angle}
	\ud \Theta_i &= \sqrt{2 D_R} \ud W_i,
\end{align}
\end{subequations}
where $D_T$ and $D_R$ are the translational and rotational diffusion coefficients, respectively, ${\bf W}_i$ and $W_i$ are $N$ independent Brownian motions in $\R^2$ and $\R$ respectively, $v_0$ is the self-propulsion speed, $\e(\theta) = (\cos \theta, \sin \theta)$, and $u$ is the pairwise interaction potential, assumed to be radially symmetric, purely repulsive and short-ranged ($\varepsilon\ll 1$). In contrast to passive Brownian particles (that is $D_R = 0$ in \cref{sde_angle}), active Brownian particles perform a biased Brownian walk in the direction given  by their orientation.

Starting from \cref{sde_all} and considering the associated Fokker--Planck equation for the joint probability density in $(\Omega \times \T)^N$, one can derive the BBGKY hierarchy of coupled equations for the marginal probabilities. Depending on the assumptions on the interaction $u$, one can obtain different continuum models for the one-particle marginal density \cite{BruBurEspSch-modelling21}. One of such models is studied by Speck et al.~\cite{Speck:2015um}, where they obtain a closed equation by assuming that the system is homogeneous and neglecting the time-dependence of the pair correlation function. Following a rescaling of time and space, leading to a rescaled velocity or P\'eclet number $\Pe = v_0/\sqrt{D_R D_T} \ge 0$, the resulting equation is (see \cite{BruBurEspSch-modelling21} for details)
\begin{align}\label{model_intro}
\partial_t f + \Pe\nabla \cdot [  f (1- \rho) \e(\theta)]
= D_e \Delta f + \partial_{\theta}^2 f, 
\end{align}
where $D_e\in (0, 1]$ is an effective translational diffusion, the unknown $f = f(t, \x, \theta)$ is the phase space density and $\rho = \rho(t, \x)$ is the spatial density obtained through
\begin{align*}
	\rho(t,\x) = \int_0^{2\pi} f(t,\x,\theta) ~\ud\theta.
\end{align*}
Both densities are scaled such that they have mass $\phi\in [0, 1)$; this represents the effective occupied fraction of particles in $\Omega$.
By integrating  \cref{model_intro} with respect to $\theta$, we obtain the following equation for $\rho$
\begin{align}\label{modelrho_intro}
\partial_t \rho + \Pe\nabla \cdot [ (1-\rho) \p ]= D_e \Delta \rho,
\end{align}
where $\p(t,\x)$ is the polarisation, defined by $\p(t,\x)=\int_0^{2\pi}\e(\theta)f(t,\x,\theta)\ud\theta$. As we shall see later on, the equation for $\rho$, \cref{modelrho_intro}, plays an important role in the analysis of \cref{model_intro} in order to determine the sign of the nonlocality (in angle) $1-\rho$, which is the reason why usual maximum/minimum principle does not work in our case. It is also interesting to consider \cref{model_intro} in one spatial dimension, $\Omega \subset \R$. In this case, the set of possible orientations reduces to $\theta = 0, \pi$, or right- and left-moving particles, and the rotational diffusion reduces to discrete jumps between these two orientations. The one-dimensional version of \cref{model_intro} reads
\begin{align}\label{GT_model}
\begin{aligned}
\partial_t \fR + \Pe\partial_x [\fR (1-\rho)] &= \partial_{xx} \fR + \fL - \fR, \\
\partial_t \fL - \Pe\partial_x [\fL (1-\rho)] &= \partial_{xx} \fL + \fR - \fL.
\end{aligned}
\end{align}
where $f_R$ and $f_L$ are the densities of right- and left-moving particles, respectively, and the space density is $\rho=f_R+f_L$. This model is derived in \cite{BruBurEspSch-modelling21} from the one-dimensional version of \cref{sde_all}, and in \cite{kourbane2018exact} as the hydrodynamic limit of an active lattice gas. In fact, the two-dimensional model \cref{model_intro} can also be derived starting from an active lattice gas model, with a partial exclusion rule (only one particle is allowed by site, but particles can exchange sites in the diffusive step). This explains why models \cref{model_intro,GT_model} present linear diffusion. If instead one considers a complete exclusion rule (resulting in an asymmetric simple exclusion process), nonlinearities of cross-diffusion type (involving terms $\rho \nabla f, f\nabla \rho$) appear in the continuum model \cite{BruBurEspSch-modelling21}. \Cref{GT_model} coincides with the crowded version (with the additional $1-\rho$ term in the advection) of the Goldstein--Taylor model \cite{goldstein1951diffusion,taylor1922diffusion}, and for this reason we refer to it as the \emph{crowded Goldstein--Taylor model}. 
We shall see that well-posedness for the above system follows directly from that of \cref{model_intro}, as consequence of our strategy based on a two-step approximation.

To the best of our knowledge, most of the results in literature focus on time-periodic evolution equations or initial-value problems with Neumann or zero-flux boundary conditions, \textit{cf.}~\textit{e.g.}~\cite{ladyzhenskaya}, while periodicity in space has not been studied in detail. Moreover, the mobility $f(1-\rho)$ in the drift term does not yield a maximum principle due to the nonlocal dependence of $\rho$ on $f$. In particular, it is not immediately clear whether one can use the theory of Ladyzhenskaya and collaborators~\cite{ladyzhenskaya}, or that established by Amann \cite{Amann84_eq_semil_para}, based on fixed point theorems.
Moreover, we note that the equation does not present a usual Wasserstein gradient flow structure, even in the absence of angular diffusion, again because of the mobility $f(1-\rho)$ in the transport drift term, \textit{cf.}~for instance, \cite{AGS08}. In this regard, we also point out that this problem does not fall in the theory of generalised Wasserstein gradient flows for equations with nonlinear mobility studied for instance in \cite{DNS09,CarLisSavSle2010} and the references therein, as the mobility in \cref{model_intro} is nonlocal in addition to being, effectively, nonlinear.

For these reasons, we provide a valuable strategy to tackle space-periodic nonlinear evolution PDEs with a nonlocal in angle drift term and periodic boundary conditions by means of a double space-discretisation purely using periodicity. More precisely, in order to prove the existence and uniqueness of solutions in the sense of \cref{def:concept of solution}, we first consider a Galerkin approximation with respect to the angular variable $\theta$. We emphasise that, in the sequel, we make the choice to expand in terms of the sine and cosine basis instead of the complex exponentials; this is so that all coefficients are real-valued, which is required to make sense of the positive part in the parabolic system that we study; \textit{cf.}~\cref{eq:parabolic system a-b} in \cref{sec:well-posedness}. Existence of solutions to such a system is proven by means of a further Galerkin approximation in the space variable, again taking into account periodicity of the problem in the choice of the basis.

In this paper we provide the first rigorous result on a macroscopic model for Brownian active particles in phase space, thus starting with the possibly simplest model without cross-diffusion effects. However, it still carries two characteristic issues of these problems, namely the nonlinear drift and the nonlocal effects due to the overall density $\rho$. Note that here we work in the setting of a continuous angular variable, with discrete angles or directions as in the one-dimensional version of the model we directly see the relation to nonlinear parabolic systems. 
We stress that this strategy may be extended to the analysis of a larger class of space-periodic evolution equations, even some cross-diffusion systems, see \cite{BruBurEspSch-modelling21}. This will indeed be the object of further investigation.
Let us mention at this point the importance of the space-periodic setting also for pattern formation effects. A closed system with no-flux boundary conditions might lead to very different stationary flux behaviour and may not lead to the phase separation phenomena observed in many microscopic simulations and discussed in \cite{BruBurEspSch-modelling21}.

This article is organised as follows. 
In \cref{sec:def-main-results} we provide the general set of assumptions and the summary of the main results. 
\Cref{sec:well-posedness} contains the proof of existence and uniqueness of weak solutions of \cref{model_intro} and \cref{GT_model}. Here we provide a detailed explanation of the strategy used, which can be potentially adapted to a larger class of space-periodic nonlinear evolution problems. More precisely, well-posedness is proven via a two-step Galerkin discretisation. This involves solving a semilinear parabolic system outside what is covered by standard parabolic theory, not due to particular complications, but rather the periodic nature of the problem. In \cref{sec:regularity} we obtain higher regularity of solutions by applying the Duhamel principle. 
\Cref{sec:very-weak-sol} is devoted to enlarging the class of initial data, among them those presenting a Dirac delta in the angular variable. This is relevant to the application of active Brownian particles, as it corresponds to all particles being originally perfectly aligned (same orientation). 
Finally, the appendices collect some technical results used in the paper.

\section{Preliminaries and main results}
\label{sec:def-main-results}

In what follows, we refer to \cref{model_intro} and \cref{GT_model} as 2D and 1D problems, respectively, in reference to the dimension of the spatial domain $\Omega$. Note that, mathematically, the two-dimensional problem \cref{model_intro} lives in the three-dimensional space $\Upsilon := \Omega \times \T$. 
We choose the spatial domain $\Omega$ to be the $d$-dimensional torus $\T^d$, where $d\in\{1,2\}$ and $\T = \R / (2 \pi \Z)$. This imposes periodic boundary conditions as considered in \cite{BruBurEspSch-modelling21, Speck:2015um}. The $2\pi$-period is chosen by mathematical convenience, so that in the two-dimensional problem $\Upsilon \equiv \T^3$. In the sequel, we say that a function $f$ is \emph{$\T^M$-periodic} (for $M\in \{1,2,3\}$) if, for all $\boldsymbol{\xi}\in\mathbb{R}^M$ and all $j\in\{1,\dots,M\}$, there holds $f(\boldsymbol{\xi})=f(\boldsymbol{\xi}+2\pi \mathbf{e}_j)$, where $\{\mathbf{e}_j\}_{j=1}^M$ are the standard unit vectors of $\mathbb{R}^M$. 


\begin{definition}[Function spaces]
Let $M\in\{1,2,3\}$. We define the following Banach spaces 
\begin{equation*}
    \begin{aligned}
        & L^p_{per}(\T^M) := \lbrace f:\mathbb{R}^M \to \mathbb{R} \big| f \text{ is a.e.~} \T^M\text{-periodic, and } f \in L^p_{loc}(\mathbb{R}^M) \rbrace, \\ 
        & H^m_{per}(\T^M) := \lbrace f:\mathbb{R}^M \to \mathbb{R} \big| f \text{ is a.e.~} \T^M\text{-periodic, and } f \in H^m_{loc}(\mathbb{R}^M) \rbrace, 
    \end{aligned}
\end{equation*}
with $p \ge 1, m \in \N\cup \{0\}$ endowed with the norms $\Vert \cdot \Vert_{L^p((0,2\pi)^M)}$ and $\Vert \cdot \Vert_{H^m((0,2\pi)^M)}$, respectively. The topological dual of $H^m_{per}(\T^M)$ is denoted by $(H^m_{per})'(\T^M)$. Similarly, we define the spaces 
\begin{equation*}
    C^k_{per}(\bar{\T}^M) := \lbrace f:\mathbb{R}^M \to \mathbb{R} \big| f \text{ is } \T^M\text{-periodic, and } f \in C^k(\mathbb{R}^M) \rbrace, 
\end{equation*}
for $k \in \mathbb{N}\cup\{0\}\cup\{\infty\}$, equipped with their respective norm $\Vert \cdot \Vert_{C^k([0,2\pi]^M)}$. 
\end{definition}

For every $f,g:\Upsilon\to\R$, we use the notation $\left\langle\cdot,\cdot\right\rangle_\Upsilon$ to denote both the $L^2$ scalar product in $\Upsilon$ and duality between $(H^1)'$ and $H^1$ as a Gelfand triple with $L^2$, while for any $a,b:\Omega\to\R$ all such pairings (in $L^2$ and $H^1$) are denoted by $\left\langle a,b\right\rangle_{\Omega}$. Moreover, let us specify that, by slight abuse of notation, any curve $f:[0,T]\to B$ will map $t\in[0,T]\mapsto f(t,\cdot):=f(t)\in B$, for any Banach space $B$ we consider. 

\begin{definition}[Weak solution to 2D model]\label{def:concept of solution}
Let $f_0\in L^2_{per}(\Upsilon)$ be such that
\begin{equation}\label{eq:rho0 requirement}
    \rho_0(\mathbf{x}) := \int_0^{2\pi} f_0(\mathbf{x},\theta) \, \ud \theta \in [0,1] \qquad \text{for a.e.~}\mathbf{x}\in\Omega. 
\end{equation}
A curve $f$ is a \emph{weak solution} to \cref{model_intro} if $f\in C([0,T];L^2_{per}(\Upsilon)) \cap L^2([0,T];H^1_{per}(\Upsilon))$, 
$f'\in L^2([0,T];(H^1_{per})'(\Upsilon))$, and for a.e.~$t\in[0,T]$ and any $\varphi \in H^1_{per}(\Upsilon)$, there holds 
    \begin{equation}\label{eq:weak_form_def}
    \begin{split}
    \left\langle\partial_t f(t), \varphi\right\rangle_{\Upsilon}\!&=\!\left\langle \Pe(1\!-\!\rho(t))f(t) \e(\theta),\nabla\varphi\right\rangle_{\Upsilon}\!-\!\left\langle D_e\nabla f(t),\nabla\varphi\right\rangle_{\Upsilon}\!-\!\left\langle \partial_\theta f(t), \partial_\theta\varphi\right\rangle_{\Upsilon},\\
    f(0,\x,\theta)&=f_0(\x,\theta),
\end{split}
    \end{equation}
with periodic boundary conditions on $\Upsilon$, \mbox{where} $\rho(t,\mathbf{x})=\int_0^{2\pi}f(t,\x,\theta)\,\ud \theta$, and the initial data is achieved in the sense $f(0)=f_0$ in $L^2_{per}(\Upsilon)$. 
\end{definition}

Our first main result is the following.

\begin{theorem}[Existence and uniqueness for the 2D model]\label{thm:exist-uniq-phenom}
Let $T>0$ and $f_0\in L^2_{per}(\Upsilon)$ be nonnegative and such that $\rho_0(\x)=\int_0^{2\pi}f_0(\x,\theta)\, \ud{\theta}\in[0,1]$ for a.e.~$\x\in\Omega$. Then, there exists a unique weak solution to \cref{eq:weak_form_def} in the sense of \cref{def:concept of solution}. 
\end{theorem}

\begin{remark}\label{rem-actually defined in L2 at 0}
In the proof of \cref{thm:exist-uniq-phenom} we prove that $f\in L^\infty([0,T];L^2_{per}(\Upsilon))\cap L^2([0,T];H^1_{per}(\Upsilon))$. Then, $f \in C([0,T];L^2_{per}(\Omega))$ is deduced from \cite[Ch.~3 Sec.~1.4 Lem.~1.2]{bookTemam77}, since $f \in L^2([0,T];H^1_{per}(\Upsilon))$ and $f' \in L^2([0,T];(H^1_{per})'(\Upsilon))$.
\end{remark}

An analogous theorem is obtained for the 1D model \cref{GT_model}, see \cref{thm:exist-uniq-phenom1d} in \cref{sec:well-posed-Goldstein-taylor}. 
We then consider the regularity of weak solutions. More precisely, we provide the following regularity result up to the initial time.

\begin{theorem}[Regularity for the 2D model]\label{thm:regularity-3d}
Let $f_0 \in L^\infty(\Upsilon)$. Then the unique weak solution of \cref{eq:weak_form_def} provided by \cref{thm:exist-uniq-phenom} satisfies $f\in C((0,T]\times\bar{\Upsilon})$. 
Furthermore, if $f_0\in C(\bar\Upsilon)$, the unique weak solution to \cref{eq:weak_form_def} belongs to $C([0,T]\times\bar{\Upsilon})$.
\end{theorem}

The previous statement also holds true for the 1D model \cref{GT_model}, see \cref{sec-regularity-1d-goldstein-taylor}.

Finally, we extend the concept of solution in order to allow for initial data with a fixed orientation (see \cref{def:distributional solution}). Note that initial data of this form are of significantly lower regularity than considered in the original \cref{thm:exist-uniq-phenom}; indeed, they are distributional in the angle variable, and hence $\rho_0$ must be defined by duality. Our existence result for initial data of this kind entails defining \emph{very weak solutions}, and we refer the reader to \cref{sec:very-weak-sol} for further details; for clarity of presentation, we do not include the definition of very weak solutions here. The main result proved therein is as follows.

\begin{theorem}[Existence of very weak solutions for 2D model]\label{thm:very-weak-sec2}
Let $T>0$, $f_0$ be a nonnegative element of $L^2_{per}(\Omega;(H^1_{per})'(0,2\pi))$ and  $\rho_0(\x) := \langle f_0(\x,\cdot) , 1 \rangle \in [0,1]$ for a.e.~$\x\in\Omega$. Then there exists a very weak solution of \cref{model_intro} in the sense of \cref{def:distributional solution}, $f \in L^2([0,T];H^1_{per}(\Omega;(H^1_{per})'(0,2\pi))) \cap L^\infty([0,T];L^2(\Omega;(H^1_{per})'(0,2\pi))) \cap L^2((0,T)\times\Upsilon)$. 
\end{theorem}

Note that the result of \cref{thm:very-weak-sec2} implies an instantaneous smoothing phenomenon, since the initial data is distributional in the angle variable, while the solution $f$ belongs to $L^2((0,T)\times\Upsilon)$. 

\section{Well-posedness}\label{sec:well-posedness}

In this section we consider the analysis of the 1D and 2D models. We tackle the 2D model first, and show in \cref{sec:well-posed-Goldstein-taylor} that the result for the 1D model follows from that of the 2D model.

A crucial ingredient for the well-posedness of \cref{model_intro} and \cref{GT_model} is that $\rho\in[0,1]$ for all times. Equation \cref{modelrho_intro} for the density $\rho$ does not have an obvious maximum or minimum principle, since the drift term neither has a sign nor is zero. One can show that any solution to \cref{modelrho_intro} (provided it exists) satisfies $\rho\le1$, starting with an initial datum $\rho_0\in[0,1]$. However, at this early stage of the analysis, the nonnegativity of the solution is a delicate issue. For this reason we consider a variant of \cref{model_intro}, where we have $(1-(\rho)_+)_+$ in the transport term. It will turn out that we are indeed solving the original equation, \textit{i.e.} \cref{model_intro}. This final step is shown in Step 5 of the Proof of \cref{thm:exist-uniq-phenom}; \textit{cf.}~\cref{section:consistent}. The modified version of \cref{model_intro} we start with is 
\begin{equation}\label{eq:variant_model_proof}
\begin{aligned}
    &\partial_t f + \Pe\nabla \cdot (f(1- (\rho)_+)_+\e(\theta)) = D_e \Delta f + \partial_{\theta}^2 f, \\ 
&f(0,\mathbf{x},\theta) = f_0(\x,\theta), 
\end{aligned}
\end{equation}
in $\Upsilon = \T^3$, where $(a)_+=\max\{a,0\}$, for any $a\in\R$. Later we show that we recover the solution to the original equation \cref{model_intro} (see \cref{section:consistent}). We remind the reader that the nonlocality in the drift term means that there is no inherent maximum/minimum principle for \cref{eq:variant_model_proof}. Our strategy for solving the equation is to construct an approximating sequence $(f^n)_{n\in\mathbb{N}}$ whose limit solves \cref{eq:variant_model_proof}.

\subsection{Approximated system in angle} \label{sec:galerkin}

We introduce the following Galerkin approximation in $\theta$. Given the periodic boundary conditions and the term $\e(\theta)$ in \cref{eq:variant_model_proof}, it is convenient to consider the  basis $\{1, \cos(k(\cdot)), \sin (k (\cdot))\}_{k\in\mathbb{N}}$ of $L^2_{per}([0,2\pi))$. Since $f_0(\x,\cdot) \in L^2_{per}([0,2\pi))$ for a.e. $\x\in\Omega$, we can write
\begin{equation}\label{eq:f0 fourier exp}
    f_0(\x,\theta) = \frac{1}{2\pi}a_0(0, \x) + \frac{1}{\pi} \sum_{k \ge 1} \big( a_k(0, \x) \cos(k\theta)  + b_k(0, \mathbf{x})\sin(k\theta)\big), 
\end{equation}
where $a_{k}(0,\cdot), b_{k}(0,\cdot)$ are the Fourier coefficients of $f_0(\cdot, \theta)$.
Note that the requirement $0 \leq \rho_0(\x) \leq 1$ (\textit{cf.}~\cref{def:concept of solution}) implies that $0 \leq a_0(0, \x) \leq 1$ for a.e.~ $\x\in\Omega$.

For some fixed  $n\in\mathbb{N}$, we consider the finite-dimensional space 
\begin{equation} \label{finite_space}
X_n = \text{span}\{\cos(k(\cdot)), \sin (k (\cdot))\}_{k=0, \dots, n}.
\end{equation}
We try to find a solution $f^n$ to the weak-formulation \cref{eq:weak_form_def} adapted to \cref{eq:variant_model_proof} of the approximated system given by
\begin{equation} \label{weak_fn}
	\left\langle\partial_t f^n, \varphi\right\rangle_{\Upsilon}\!=\!\left\langle \Pe(1\!-\!(\rho^{n})_+)_+f^n(t) \e(\theta),\nabla\varphi\right\rangle_{\Upsilon}\!-\!\left\langle D_e\nabla f^n,\nabla\varphi\right\rangle_{\Upsilon}\!-\!\left\langle \partial_\theta f^n,
	\partial_\theta\varphi\right\rangle_{\Upsilon},
\end{equation}
where $\rho^n = \int f^n \ud \theta$, for  test functions $\varphi =\chi(\theta) \psi(\x)$ with $ \chi \in X_n$ and $\psi \in H^1_\text{per}(\Omega)$. We can express $f^n$ as
\begin{equation}\label{eq:approx fn}
        f^{n}(t,\x,\theta):= \frac{1}{2\pi}a_{0}^n(t,\mathbf{x}) + \frac{1}{\pi} \sum_{k = 1}^n \big( a_{k}^n(t,\mathbf{x}) \cos(k\theta)  + b_{k}^n(t,\mathbf{x})\sin(k\theta) \big), 
    \end{equation}
where the Fourier coefficients $a_{k}^n(t, \x), b_{k}^n(t, \x)$ need to be determined. We note that $\rho^n$ above corresponds to the constant coefficient, that is,
\begin{equation}\label{eq:approx rho n}
        \rho^n(t,\x)=\int_0^{2\pi} f^n(t,\x,\theta) \, \ud\theta \equiv a_{0}^n(t,\x).
\end{equation}
From \cref{weak_fn} we obtain a $(2n+1)$-dimensional system of semilinear parabolic equations for $a_{k}^n(t, \x), b_{k}^n(t, \x)$. In particular, testing \cref{weak_fn} for each $\chi(\theta)$ in the basis of $X_n$ gives the equation for the corresponding Fourier coefficient of $f^n$. More precisely, choosing $\chi = 1$ results in
\begin{subequations}
  	\label{eq:parabolic system a-b}
\begin{equation}\label{eq:a_0-equation}
	\partial_t a_{0}^n+ \Pe\nabla\cdot\left[(1- (a_{0}^n)_+)_+ {\bf J}_0^n \right] = D_e\Delta a_{0}^n, \quad {\bf J}_0^n = (a_{1}^n,b_{1}^n).
\end{equation}
Repeating for $\chi= \cos (k \theta), \sin(k \theta)$ for $k = 1, \dots, n$ leads to
\begin{align}
  	\partial_t a_{k}^n + \frac{\Pe}{2}\nabla\cdot\left[(1- (a_{0}^n)_+)_+ {\bf J}_{k}^n \right] & = D_e\Delta a_{k}^n - k^2 a_{k}^n,\\
  	\partial_t b_{k}^n + \frac{\Pe}{2}\nabla\cdot\left[(1- (a_{0}^n)_+)_+ {\bf Q}_{k}^n  \right] & = D_e\Delta b_{k}^n - k^2b_{k}^n,
  	\end{align}
in $\T^2$, where the velocities ${\bf J}_k^n$ and ${\bf Q}_{k}^n$ are given by 
\begin{align}
\begin{aligned}
  {\bf J}_{1}^n &= (2a_{0}^n+a_{2}^n,b_{2}^n), &  {\bf Q}_{1}^n & = (b_{2}^n,2a_{0}^n - a_{2}^n),\\
  {\bf J}_{l}^n & = (a_{l+1}^n+a_{l-1}^n,b_{l+1}^n-b_{l-1}^n), & {\bf Q}_{l}^n & = (b_{l+1}^n+b_{l-1}^n,a_{l-1}^n - a_{l+1}^n),\\
  {\bf J}_{n}^n & = (a_{n-1}^n,-b_{n-1}^n), &  {\bf Q}_{n}^n & = (b_{n-1}^n,a_{n-1}^n),
  \end{aligned}
\end{align}
\end{subequations}
for $l = 2, \dots, n-1$, with periodic boundary conditions on $\Omega$ and initial conditions $a_{k}^n(0,\x) , b_{k}^n(0,\x) $ given by the corresponding Fourier coefficients of $f_0$ in \cref{eq:f0 fourier exp}.

Provided that \cref{eq:parabolic system a-b} can be solved, we obtain an explicit description of each $f^n$ in terms of its Galerkin coefficients $a_{k}^n, b_{k}^n$. In order to prove \cref{thm:exist-uniq-phenom}, we need to show the convergence (in a sense to be made precise) of the sequence $(f^n)_{n\in\mathbb{N}}$ towards some limit $f$, and we must show that $f$ solves \cref{eq:variant_model_proof} in the weak sense prescribed by \cref{def:concept of solution}. As previously mentioned, we will see that we are indeed solving \cref{model_intro}.

\subsection{Existence of solutions to the semilinear parabolic system with spatially-periodic boundary conditions}
\label{subsec:existence of parabolic system further galerkin}

In order to solve the semilinear parabolic system \cref{eq:parabolic system a-b}, we perform a further Galerkin approximation in space, with the aim of rewriting it as a system of ODEs. While the well-posedness for systems similar to \cref{eq:parabolic system a-b} is considered by Ladyzhenskaya et al.~\cite{ladyzhenskaya} and Amann \cite{Amann1985_semilin_parab}, it is not obvious how these results can be applied to space-periodic problems or to systems without a maximum principle  (resp.~$L^\infty$-bounds). This motivates our study. We define 
$$
{\bf c}^n (t,\x) = (a_0^n, a_1^n, \dots, a_n^n, b_1^n, \dots, b_n^n).
$$
Since throughout this section we keep $n\in\mathbb{N}$ fixed, in order to avoid confusion we drop the superscript $n$.

\begin{theorem}[Solution to \cref{eq:parabolic system a-b}]\label{thm:well-posed-parab}
Let ${\bf c}(0,\x) \in (L^2_{per}(\Omega))^{2n+1}$ be the initial data. For any $T>0$ there exist curves $\bf c$ weak solutions to system \cref{eq:parabolic system a-b} such that
\begin{subequations}\label{eq:regularity-solution-system}
\begin{align}
    {\bf c} &\in (L^\infty([0,T];L^2_{per}(\Omega))^{2n+1}\cap L^2([0,T];H^1_{per}(\Omega)))^{2n+1},\\
    {\bf c}' &\in (L^2([0,T];(H^1_{per})'(\Omega)))^{2n+1}.
\end{align}
\end{subequations}
With $C_{0,n}=\sum_{k=0}^n(\|a_{k}(0)\|_{L^2(\Omega)}^2+\|b_{k}(0)\|_{L^2(\Omega)}^2)$, there holds 
\begin{subequations}\label{eq:uniform-estimates-sol-syst}
\begin{align}
    \sum_{k=0}^n\left(\|a_{k}(t)\|_{L^2(\Omega)}^2+\|b_{k}(t)\|_{L^2(\Omega)}^2\right)\le C_{0,n}\exp\left( 2CT\frac{\Pe^2}{D_e}\right),
\end{align}
\begin{equation}
    \begin{split}
        \sum_{k=0}^n\int_0^T\!\!\!\! k^2\big(\|a_{k}(t)\|_{L^2(\Omega)}^2\!+\!\|b_{k}(t)\|_{L^2(\Omega)}^2\big) &\!+\!D_e \big(\|\nabla a_{k}(t)\|_{L^2(\Omega)}^2\!+\!\|\nabla b_{k}(t)\|_{L^2(\Omega)}^2\big)\ud{t}\\
        &\leq C(C_{0,n},\Pe,D_e,T), 
    \end{split}
\end{equation}
\begin{equation}
    \|a'_0\|_{L^2([0,T];(H^1_{per})'(\Omega))}\le\bar{C}(\Pe,\phi,T,C_{0,n}).
\end{equation}
\end{subequations}
\end{theorem}
\begin{proof} \textit{Step 1 (approximation of the system): } We look for an approximate solution ${\bf c}^m$ of \cref{eq:parabolic system a-b} in $X_m^2$ \cref{finite_space}
$$
X_m^2 = \text{span}\{ \cos(px),\sin(px), \cos(qy), \sin(qy) \}_{p, q = 0, \dots, m}.
$$
From the $(2n+1)$-dimensional system of PDEs \cref{eq:parabolic system a-b}, this further approximation leads to a $(2n+1)(2m+1)^2$-dimensional system of ordinary differential equations (ODEs) for the Fourier coefficients of the elements of ${\bf c}^m$. Defining the vector of unknowns by $\Lambda^{n,m}(t)$ and its initial data by $\Lambda^{n,m}_0$ (which are given explicitly in \cref{section:big ode}), we obtain an ODE of the following form:
\begin{equation}\label{eq:ODE system Lambda}
    \left\lbrace\begin{aligned}
    &\frac{d\Lambda^{n,m}}{dt}(t) = F(t,\Lambda^{n,m}(t)) \qquad \forall t \in [0,T], \\ 
    &\Lambda^{n,m}(0) = \Lambda^{n,m}_0, 
    \end{aligned}\right. 
\end{equation}
where $F$ is locally Lipschitz. Therefore the Cauchy--Lipschitz Theorem implies that, for each $m\in\mathbb{N}$, there exists a solution to \cref{eq:ODE system Lambda} in $C^1([0,T])$.

\textit{Step 2 (uniform estimates for the approximating solution): } Testing the weak formulations for $\{a_k ^{n,m},b_k ^{n,m}\}$ to perform the classical $L^2$ parabolic estimate, we obtain the following estimates for a.e. $t\in[0,T]$ :
\begin{equation}\label{eq:a_0-estimate}
     \begin{split}
     \frac{1}{2}\frac{d}{dt}\|a_0 ^{n,m} \|_{L^2(\Omega)}^2& = \left\langle\partial_ta_0 ^{n,m} ,a_0 ^{n,m} \right\rangle_{\Omega}\\
     &\le -\frac{D_e}{2}\|\nabla a_0 ^{n,m} \|_{L^2(\Omega)}^2 + \frac{\Pe^2}{2D_e}\left(\|a_1 ^{n,m} \|_{L^2(\Omega)}^2+\|b_1 ^{n,m} \|_{L^2(\Omega)}^2\right),
     \end{split}
 \end{equation}
\begin{equation}\label{eq:a_1-b_1-estimate}
    \begin{split}
    \frac{1}{2}&\frac{d}{dt} \left(\|a_1 ^{n,m} \|_{L^2(\Omega)}^2+\|b_1 ^{n,m} \|_{L^2(\Omega)}^2\right)\\
    &= - D_e\left(\|\nabla a_1 ^{n,m} \|_{L^2(\Omega)}^2 + \|\nabla b_1 ^{n,m} \|_{L^2(\Omega)}^2\right)  - \left(\|a_1 ^{n,m} \|_{L^2(\Omega)}^2 + \|b_1 ^{n,m} \|_{L^2(\Omega)}^2\right)\\
    &\quad + \frac{\Pe}{2}\left\langle( 1 \!-\! (a_0 ^{n,m})_+)_+ {\bf J}_1 ^{n,m}, \nabla a_1 ^{n,m} \right\rangle_\Omega \!+\! \frac{\Pe}{2}\left\langle( 1\! -\! (a_0 ^{n,m})_+)_+ {\bf Q}_1 ^{n,m}, \nabla b_1 ^{n,m} \right\rangle_\Omega\\
  & \leq  \!- \frac{3D_e}{4}\left(\|\nabla a_1 ^{n,m} \|_{L^2(\Omega)}^2\! +\! \|\nabla b_1 ^{n,m} \|_{L^2(\Omega)}^2\right)\!  -\! \left( \|a_1 ^{n,m} \|_{L^2(\Omega)}^2\! +\! \|b_1 ^{n,m} \|_{L^2(\Omega)}^2 \right)\\
    &\quad +\frac{4\Pe^2}{D_e}\|a_0 ^{n,m} \|_{L^2}^2 +\frac{\Pe^2}{D_e}\left(\|a_2 ^{n,m} \|_{L^2(\Omega)}^2+\|b_2 ^{n,m} \|_{L^2(\Omega)}^2\right), 
    \end{split}
\end{equation}
while, for $k\in\{2,\dots,n-1\}$, 
\begin{equation}\label{eq:a_k-b_k-estimate}
    \begin{split}
    &\frac{1}{2} \frac{d}{dt}  \left( \|a_{k} ^{n,m} \|_{L^2(\Omega)}^2 + \|b_{k} ^{n,m} \|_{L^2(\Omega)}^2 \right) \\
    & \le\! -\frac{3D_e}{4} \left( \|\nabla a_k ^{n,m} \|_{L^2(\Omega)}^2 + \|\nabla b_k ^{n,m} \|_{L^2(\Omega)}^2 \right) \! -\! k^2 \left( \|a_k ^{n,m} \|_{L^2(\Omega)}^2 + \|b_k ^{n,m} \|_{L^2(\Omega)}^2 \right)\\
    &\quad+ \frac{\Pe^2}{D_e} \left( \|a_{k+1} ^{n,m} \|_{L^2(\Omega)}^2 + \|a_{k-1} ^{n,m} \|_{L^2(\Omega)}^2  +  \|b_{k+1} ^{n,m} \|_{L^2(\Omega)}^2 + \|b_{k-1} ^{n,m} \|_{L^2(\Omega)}^2 \right),
    \end{split}
\end{equation}
and
\begin{equation}\label{eq:a_n-b_n-estimate}
    \begin{split}
    \frac{1}{2}\frac{d}{dt} \left( \|a_{n} ^{n,m} \|_{L^2(\Omega)}^2 + \|b_{n} ^{n,m} \|_{L^2(\Omega) }^2\right)
    & \le -\frac{3D_e}{4}\left(\|\nabla a_n ^{n,m} \|_{L^2(\Omega)}^2 + \|\nabla b_n ^{n,m} \|_{L^2(\Omega)}^2\right)\\
    &\quad- n^2\left(\|a_n ^{n,m} \|_{L^2(\Omega)}^2+\|b_n ^{n,m} \|_{L^2(\Omega)}^2\right)\\
    &\quad+ \frac{\Pe^2}{2D_e} \left( \|a_{n-1} ^{n,m} \|_{L^2(\Omega)}^2 + \|b_{n-1} ^{n,m} \|_{L^2(\Omega)}^2 \right).
    \end{split}
\end{equation}
By summing up Eqs.~\cref{eq:a_0-estimate}--\cref{eq:a_n-b_n-estimate}
, we obtain
\begin{equation}\label{eq:sum_up_eqs}
    \begin{split}
       &\frac{1}{2}\!\frac{d}{dt}\!\sum_{k=0}^n\!\left(\!\|a_k ^{n,m}\|_{L^2(\Omega)}^2\!+\!\|b_k ^{n,m}\|_{L^2(\Omega)}^2\!\right)\!+\!\sum_{k=0}^n\!k^2\!\left(\!\|a_k ^{n,m}\|_{L^2(\Omega)}^2\!+\!\|b_k ^{n,m}\|_{L^2(\Omega)}^2\!\right)\\
       &\!+\!D_e\!\sum_{k=0}^n\!\left(\!\|\nabla a_k ^{n,m}\|_{L^2(\Omega)}^2\!+\!\|\nabla b_k ^{n,m}\|_{L^2(\Omega)}^2\!\right)\\
       &\qquad\qquad\le C\frac{\Pe^2}{D_e}\sum_{k=0}^n\left(\|a_k ^{n,m}\|_{L^2(\Omega)}^2+\|b_k ^{n,m}\|_{L^2(\Omega)}^2\right),
    \end{split}
\end{equation}
for some positive constant $C$ independent of $n,m$ and $\Pe$. Since on the left-hand side all the quantities but the time-derivative are positive, we also have
\begin{equation*}
    \frac{d}{dt}\!\sum_{k=0}^n\!\left(\!\|a_k ^{n,m}\|_{L^2(\Omega)}^2\!+\!\|b_k ^{n,m}\|_{L^2(\Omega)}^2\!\right)\!\le \!2C\frac{\Pe^2}{D_e}\!\sum_{k=0}^n\!\left(\!\|a_k ^{n,m}\|_{L^2(\Omega)}^2\!+\!\|b_k ^{n,m}\|_{L^2(\Omega)}^2\!\right),
\end{equation*}
whence, by Gr\"{o}nwall's inequality, using also \cref{eq-eg a0km initial data for ode system} and the Plancherel Theorem, 
\begin{equation}\label{eq:LinfL2-coeff}
    \sum_{k=0}^n\left(\|a_k ^{n,m}\|_{L^2(\Omega)}^2+\|b_k ^{n,m}\|_{L^2(\Omega)}^2\right)\le C_{0,n}\exp\left( 2CT\frac{\Pe^2}{D_e}\right),
\end{equation}
where, $C_{0,n}=\sum_{k=0}^n\|a_{k}(0)\|_{L^2(\Omega)}^2+\|b_{k}(0)\|_{L^2(\Omega)}^2$, which is bounded by $\Vert f_0 \Vert_{L^2(\Upsilon)}^2$. From \cref{eq:LinfL2-coeff} we deduce that $\{a_k ^{n,m},b_k ^{n,m}\}_{k=1}^n$ are bounded in $L^\infty([0,T];L_{per}^2(\Omega))$ independently of $m$. As direct consequence, by integrating \cref{eq:sum_up_eqs} in time, 
\begin{equation}\label{eq:L2H1-coeff}
    \begin{split}
        \sum_{k=0}^nk^2\int_0^T\big(&\|a_k ^{n,m}\|_{L^2(\Omega)}^2+\|b_k ^{n,m}\|_{L^2(\Omega)}^2\big)\ud{t}\\
        &+D_e\sum_{k=0}^n\int_0^T\left(\|\nabla a_k ^{n,m}\|_{L^2(\Omega)}^2+\|\nabla b_k ^{n,m}\|_{L^2(\Omega)}^2\right)\ud{t}\\
        &\le C\frac{\Pe^2}{D_e}\sum_{k=0}^n\int_0^T\left(\|a_k ^{n,m}\|_{L^2(\Omega)}^2+\|b_k ^{n,m}\|_{L^2(\Omega)}^2\right)\ud{t} + C_{0,n} \\ 
        &\le CT\frac{\Pe^2}{D_e}C_{0,n}\exp\left(2CT \frac{\Pe^2}{D_e}\right)+C_{0,n} =: C(C_{0,n},\Pe,D_e,T), 
    \end{split}
\end{equation}
independent of $m$, and $\{a_k ^{n,m},b_k ^{n,m}\}_{k=1}^n$ are uniformly bounded in $L^\infty([0,T];L^2_{per}(\Omega))\cap L^2([0,T];H^1_{per}(\Omega))$. 
Next, we show $\|{a_{k}^m}'\|_{L^2([0,T];(H^1_{per})'(\Omega))}$, $\|{b_{k}^m}'\|_{L^2([0,T];(H^1_{per})'(\Omega))}$, for $0\le k\le n$, are uniformly bounded in $m$. 

Let us consider a test function $\psi\in H^1_{per}(\Omega)$ such that $\|\psi\|_{H^1(\Omega)}\le1$. Since $a_{0}^m$ is a weak solution to \cref{eq:a_0-equation} for $k=0$, we have for a.e. $t\in[0,T]$ 
\begin{equation}\label{eq:a_0-time-derivative-estimate}
    \begin{split}
        |\left\langle\partial_t a_{0}^{m} , \psi \right\rangle_{\Omega}|&\le D_e| \left\langle \nabla a_{0}^{m} ,\nabla\psi\right\rangle_{\Omega}|+|\left\langle \Pe(1- (a_{0}^{m})_+)_+(a_{1}^{m} ,b_{1}^{m} ),\nabla\psi\right\rangle_{\Omega}|\\
        &\le D_e\|\nabla a_{0}^{m} \|_{L^2(\Omega)}\|\nabla\psi\|_{L^2(\Omega)}+ \Pe\|a_{1}^{m} \|_{L^2(\Omega)}\|\nabla\psi\|_{L^2(\Omega)}\\
        &\quad+ \Pe\|b_{1}^{m} \|_{L^2(\Omega)}\|\nabla\psi\|_{L^2(\Omega)}\\
        &\le\max\left\{ \Pe,D_e \right\}\left(\|a_{1}^{m} \|_{L^2(\Omega)}+\|b_{1}^{m} \|_{L^2(\Omega)}+\|\nabla a_{0}^{m} \|_{L^2(\Omega)}\right).
    \end{split}
\end{equation}
By taking the supremum over all $\psi\in H^1_{per}(\Omega)$ such that $\|\psi\|_{H^1(\Omega)}\le1$, squaring and integrating in time, we obtain 
\begin{equation}\label{eq.time deriv a0 bounded}
    \begin{split}
    \int_0^T\|\partial_ta_{0}^{m}& \|_{(H^1_{per})'(\Omega)}^2\ud{t} 
    \le\bar{C}(\Pe,\phi,T,C_0),
    \end{split}
\end{equation}
since $a_{1}^{m},b_{1}^{m}$ are uniformly bounded in $L^\infty([0,T];L^2_{per}(\Omega))\cap L^2([0,T];H^1_{per}(\Omega))$, and the boundedness of $\Vert \nabla a_{0}^{m}\Vert_{L^2([0,T];L^2_{per}(\Omega))}$ follows from \cref{eq:L2H1-coeff}. 
In turn, this gives that the sequence $\|{a_{0}^m}'\|_{L^2([0,T];(H^1_{per})'(\Omega))}$ is uniformly bounded in $m$.
With a similar computation we can prove a uniform bound in $m$ for $\|{a_{k}^m}'\|_{L^2([0,T];(H^1_{per})'(\Omega))}$, $\|{b_{k}^m}'\|_{L^2([0,T];(H^1_{per})'(\Omega))}$, where $1\le k\le n$.

  \textit{Step 3 (convergence of the approximating solutions): } For $0 \leq k \leq n$, the sequence $\{a_k ^{m},b_k ^{m}\}_m$ is uniformly bounded in $m$ in
$L^\infty([0,T];L^2_{per}(\Omega))\cap L^2([0,T];H^1_{per}(\Omega))$. The Banach--Alaoglu Theorem yields that there exist subsequences $\{a_{k}^{m_l},b_{k}^{m_l}\}_l$ and curves $a_k, b_k \in L^\infty([0,T];L^2_{per}(\Omega))\cap L^2([0,T];H^1_{per}(\Omega))$, for $0\le k \le n$, such
that
\begin{subequations}\label{eq:weak conv-ab}
\begin{align}
 &a_{k}^{m_l}\overset{\ast}{\rightharpoonup} a_k, \quad b_{k}^{m_l}\overset{\ast}{\rightharpoonup} b_k \quad \mbox{ in } L^\infty([0,T];L^2_{per}(\Omega)) \text{ as } m_l \to \infty,\\
    & a_{k}^{m_l}\rightharpoonup a_k, \quad b_{k}^{m_l}\rightharpoonup b_k \quad \mbox{ in }L^2([0,T];H^1_{per}(\Omega)) \text{ as } m_l \to \infty.
\end{align}
\end{subequations}
The limits coincide since $H^1\subset L^2\equiv (L^2)'\subset (H^1)'$ and $\varphi\in L^2([0,T];L^2_{per}(\Omega))$ is a common test function. Since $\|{a_{0}^m}'\|$, $\|{a_{k}^m}'\|$, $\|{b_{k}^m}'\|$ are bounded independently of $m$, for $1\le k\le n$, in $L^2([0,T];(H^1_{per})'(\Omega))$, the Banach--Alaoglu Theorem implies 
\begin{align}\label{eq:weak conv ab dt}
    {a_{k}^{m_l}}'\overset{\ast}{\rightharpoonup} g_{a_k} \quad {b_{k}^{m_l}}'\overset{\ast}{\rightharpoonup} g_{b_k} \quad \mbox{ in } L^2([0,T];(H^1_{per})'(\Omega)),
\end{align}
for some $g_{a_k},g_{b_k} \in L^2([0,T];(H^1_{per})'(\Omega))$, $0\le k\le n$. By testing against a smooth compactly supported test function defined on $(0,T)$ and using the weak convergence $a_{k}^{m_l}\rightharpoonup a_k$, $b_{k}^{m_l}\rightharpoonup b_k$ in $L^2([0,T];H^1_{per}(\Omega))$, we get $g_{a_k}=a'_{k}$, $g_{b_k}=b'_{k}$. Applying the Aubin--Lions Lemma to the sequence $\{a_{0}^{m}\}_{m\in\mathbb{N}}$, due to the uniform bound of $\|a'_{0,m}\|_{L^2([0,T];(H^1_{per})'(\Omega))}$ from \cref{eq.time deriv a0 bounded}, $a_{0}^{m}$ strongly converges in $L^2([0,T];L^2_{per}(\Omega))$.

The strong convergence of $a_{0}^{m}$, \cref{eq:weak conv-ab}, and \cref{eq:weak conv ab dt} allow us to pass to the limit in the weak form of system \cref{eq:parabolic system a-b}, thus to obtain existence of weak solutions. We stress that strong convergence for $a_{0}^{m}$ is needed in order to allow convergence through the positive part in the drift term. Indeed, the positive part function is continuous and it grows at most linearly, thus $(1-(a_{0}^{m})_+)_+ \to (1-(a_0)_+)_+$ strongly in $L^2([0,T];L^2_{per}(\Omega))$. As consequence of the lower semicontinuity of the norm we obtain the regularity \cref{eq:regularity-solution-system} and the uniform estimates \cref{eq:uniform-estimates-sol-syst}.
\end{proof}

\subsection{Consistency of the approximating scheme}\label{section:consistent}
We now proceed with the proof of \cref{thm:exist-uniq-phenom}. We split the proof into two parts: existence and uniqueness. The proof of existence consists of five distinct steps. 
\begin{proof}[Proof of \cref{thm:exist-uniq-phenom} - existence]
First of all, let us remind the reader we consider the variant \cref{eq:variant_model_proof} of the original problem, with periodic boundary conditions on $\Upsilon$.
We will see that we recover the solution to the original problem \cref{model_intro} in Step 5. 
 
  \textit{Step 1 (approximating solution): } As previously outlined, for each $n\in\N$, we consider the approximation of \cref{eq:variant_model_proof} given by \cref{eq:approx fn} where the functions $\{a_{k}^n, b_{k}^n\}_{k=0}^n$ are a solution the semilinear parabolic equations \cref{eq:parabolic system a-b}. Note that \cref{thm:well-posed-parab} provides existence for the coefficients $\{a_{k}^n, b_{k}^n\}_{k=0}^n$ satisfying \cref{eq:regularity-solution-system} and \cref{eq:uniform-estimates-sol-syst}. In the next step we will use that, for each $n\in\N$, there holds
\begin{subequations}\label{eq:n-uniform-estimates-sol-syst}
\begin{align}\label{eq:n-l2-bound}
    \sum_{k=0}^n\left(\|a_{k}^n(t)\|_{L^2(\Omega)}^2+\|b_{k}^n(t)\|_{L^2(\Omega)}^2\right)\le C_0\exp\left( 2CT\frac{\Pe^2}{D_e}\right),
\end{align}
\begin{equation}\label{eq:nh1-bound}
    \begin{split}
        &\sum_{k=0}^nk^2\int_0^T\big(\|a_{k}^n(t)\|_{L^2(\Omega)}^2\!+\!\|b_{k}^n(t)\|_{L^2(\Omega)}^2\big)\ud{t}\\
        &+D_e\sum_{k=0}^n\int_0^T\!\left(\|\nabla a_{k}^n(t)\|_{L^2(\Omega)}^2+\|\nabla b_{k}^n(t)\|_{L^2(\Omega)}^2\right)\ud{t} \leq C(C_0,\Pe,D_e,T),
    \end{split}
\end{equation}
\begin{equation}\label{eq:n-time-deriv-bound}
    \|{a_{0}^n}'\|_{L^2([0,T];(H^1)'(\Omega))}\le\bar{C}(\Pe,D_e,T,C_0),
\end{equation}
\end{subequations}
being $C_0=\sum_{k=0}^\infty(\|a_{k}(0,\cdot)\|_{L^2(\Omega)}^2+\|b_{k}(0, \cdot) \|_{L^2(\Omega)}^2) = \Vert f_0 \Vert_{L^2(\Omega)}^2$ and $C$ a constant independent of $n$ (see Eqs.~\cref{eq:LinfL2-coeff}-\cref{eq:L2H1-coeff} in the proof of \cref{thm:well-posed-parab}).

\textit{Step 2 (uniform estimates for the approximating solution): } The uniform bound in $L^\infty([0,T];L^2_{per}(\Upsilon))$ for the approximation $f^n$ of \cref{eq:approx fn} follows directly from \cref{eq:n-uniform-estimates-sol-syst}. Indeed, or any $t\in[0,T]$, there holds
    \begin{equation*}
    \begin{split}
    \|f^n(t)\|_{L^2(\Upsilon)}^2\!\le\!\frac{1}{\pi}\|a_{0}^n(t)\|_{L^2(\Omega)}^2\!+\!\frac{4}{\pi}\sum_{k=1}^n\left(\|a_{k}^n(t)\|_{L^2(\Omega)}^2\!+\!\|b_{k}^n(t)\|_{L^2(\Omega)}^2\right)\!\le\!\frac{C_0}{\pi} e^{2CT\frac{\Pe^2}{D_e}}.
\end{split}
\end{equation*}
Note that in the penultimate inequality we used $|\sin(k\theta)|,|\cos(k\theta)|\le1$ for any $k\in\mathbb{N}$, Fubini Theorem and the orthogonality conditions for the trigonometric functions. 
Similarly, we obtain a uniform bound for $f^n$ in $L^2([0,T];H^1_{per}(\Upsilon))$. Indeed, $\int_0^T\|\nabla_\xi f^n(t)\|_{L^2(\Upsilon)}^2 \, \ud{t}\!=\!\int_0^T\int_{\Upsilon}|\nabla_{\x}f^n(t,\x,\theta)|^2 \, \ud{\x} \, \ud{\theta} \, \ud{t}\!+\!\int_0^T\int_{\Upsilon}|\partial_\theta f^n(t,\x,\theta)|^2 \, \ud{\x} \, \ud{\theta} \, \ud{t}$, from which it follows that, with $\boldsymbol{\xi}=(\x,\theta)$, 
    \begin{equation*}
    \begin{split}
    \int_0^T\|\nabla_{\boldsymbol{\xi}} f^n (t)\|_{L^2(\Upsilon)}^2 \, \ud{t}&\leq \frac{1}{\pi}\int_0^T\|\nabla a_{0}^n(t)\|_{L^2(\Omega)}^2\ud{t}\\
    &\quad+\frac{4}{\pi}\sum_{k=1}^n\int_0^T\left(\|\nabla a_{k}^n(t)\|_{L^2(\Omega)}^2+\|\nabla b_{k}^n(t)\|_{L^2(\Omega)}^2\right)\ud{t}\\
    &\quad+\frac{2}{\pi}\sum_{k=1}^nk^2\int_0^T\left(\|a_{k}^n(t)\|_{L^2(\Omega)}^2+\|b_{k}^n(t)\|_{L^2(\Omega)}^2\right)\ud{t}\\
    &\overset{\cref{eq:nh1-bound}}{\le}C(C_0,\Pe,D_e,T),
\end{split}
\end{equation*}
where this final constant is independent of $n$. 

Finally, we prove that $\|(f^{n})'\|_{L^2([0,T];(H^1_{per})'(\Upsilon))}$ is uniformly bounded. This will be a crucial observation in order to pass to the limit in $n$ in the equation, by applying the Aubin--Lions Lemma. Let $\varphi\in H^1_{per}(\Upsilon)$ be a test function such that $\|\varphi\|_{H^1(\Upsilon)}\le1$. Since $f^n$ is a weak solution to \cref{eq:variant_model_proof}, we obtain for a.e. $t\in[0,T]$
\begin{equation}\label{eq:time-derivative-estimate}
    \begin{split}
        |\left\langle\partial_t f^n(t), \varphi \right\rangle_{\Upsilon}|&\le D_e|\left\langle \nabla  f^n(t),\nabla \varphi\right\rangle_{\Upsilon}|+|\left\langle \partial_\theta f^n(t), \partial_\theta\varphi\right\rangle_{\Upsilon}|\\
        &\quad+|\left\langle \Pe(1- (a_{0}^n)_+)_+f^n(t) \e(\theta),\nabla \varphi\right\rangle_{\Upsilon}|\\
        &\le D_e\|\nabla f^n(t)\|_{L^2(\Upsilon)}\|\nabla \varphi\|_{L^2(\Upsilon)}+\|\partial_\theta f^n(t)\|_{L^2(\Upsilon)}\|\partial_\theta\varphi\|_{L^2(\Upsilon)}\\
        &\quad+  \Pe\|f^n(t)\|_{L^2(\Upsilon)}\|\nabla \varphi\|_{L^2(\Upsilon)}\\
        &\le\max\left\{ \Pe,D_e\right\}\left(\|f^n(t)\|_{L^2(\Upsilon)}+\|\nabla f^n(t)\|_{L^2(\Upsilon)}+\|\partial_\theta f^n(t)\|_{L^2(\Upsilon)}\right).
    \end{split}
\end{equation}
Taking the supremum over all $\varphi\in H^1_{per}(\Upsilon)$ with $\|\varphi\|_{H^1(\Upsilon)}\le1$, $\|\partial_tf^n(t)\|_{(H^1_{per})'(\Upsilon)}\le C(\Pe,D_e)\left(\|f^n(t)\|_{L^2(\Upsilon)}+\|\nabla f^n(t)\|_{L^2(\Upsilon)}+\|\partial_\theta f^n(t)\|_{L^2(\Upsilon)}\right)$, whence
\begin{equation}
    \begin{split}
    \int_0^T\|\partial_tf^n(t)\|_{(H^1_{per})'(\Upsilon)}^2\ud{t}&\le \tilde{C}(\Pe,D_e)\int_0^T\|f^n(t)\|_{H^1(\Upsilon)}^2\ud{t} \le\bar{C}(\Pe,\phi,T,C_0),
    \end{split}
\end{equation}
due to the uniform bound on $\|f^n\|_{L^2([0,T];H^1_{per}(\Upsilon))}$. Therefore, we have the desired bound for $(f^{n})'$ in $L^2([0,T];(H^1_{per})'(\Upsilon))$.

  \textit{Step 3 (convergence of the approximation): } Since the sequence $\{f^n\}_{n\in\mathbb{N}}$ of \cref{eq:approx fn} is uniformly bounded in $L^\infty([0,T];L^2_{per}(\Upsilon))\cap L^2([0,T];H^1_{per}(\Upsilon))$ we infer from the Banach--Alaoglu Theorem that there exist a subsequence $\{f^{n_k}\}_{k\in\mathbb{N}}$ and a curve $f\in L^\infty([0,T];L^2_{per}(\Upsilon))\cap L^2([0,T];H^1_{per}(\Upsilon))$ such
that 
\begin{equation}\label{eq:weak conv}
    \begin{split}
    &f^{n_k}\overset{\ast}{\rightharpoonup} f \quad \mbox{ in } L^\infty([0,T];L^2_{per}(\Upsilon)), \quad f^{n_k}\rightharpoonup f \quad \mbox{ in }L^2([0,T];H^1_{per}(\Upsilon)).
    \end{split}
\end{equation}
The limits coincide since $H^1\subset L^2\equiv (L^2)'\subset (H^1)'$ and any $\varphi\in L^2([0,T];L^2_{per}(\Upsilon))$ is a common test function. The uniform boundedness of $\|(f^n)'\|_{L^2([0,T];(H^1_{per})'(\Upsilon))}$ and the Banach--Alaoglu Theorem also yields that 
\begin{align}\label{eq:weak conv dt}
    (f^{n_k})'\overset{\ast}{\rightharpoonup} f' \quad \mbox{ in } L^2([0,T];(H^1_{per})'(\Upsilon)), 
\end{align}
where we verify that this latter limit is indeed $f'$ by testing against a smooth compactly supported test function defined on $(0,T)$ and using the weak convergence $f^{n_k}\rightharpoonup f$ in $L^2([0,T];H^1_{per}(\Upsilon))$. Henceforth, we do not relabel further subsequences. 

Additionally, the aforementioned boundedness of $\|(f^n)'\|_{L^2([0,T];(H^1_{per})'(\Upsilon))}$  yields, taking a further subsequence if necessary, by an application of the Aubin--Lions Lemma that $f^{n}\to f$ strongly in $L^2([0,T];L^2_{per}(\Upsilon))$ as $n\to\infty$. We take a further subsequence such that, for a.e.~$t\in[0,T]$, $f^{n}(t)\to f(t)$ strongly in $L^2_{per}(\Upsilon)$ as $n\to\infty$. Moreover, by further applying the theorem of Banach-Alaoglu and the Aubin--Lions Lemma to the sequence of curves $\{a_{0}^n\}_{n\in\mathbb{N}}$, due to the uniform bounds from \cref{eq:n-uniform-estimates-sol-syst}, we also know that $a_{0}^n$ strongly converges in $L^2([0,T];L^2_{per}(\Omega))$. In other words, up to a subsequence, $\{\rho^n\}_{n\in\mathbb{N}}$ converges strongly in $L^2([0,T];L^2_{per}(\Omega))$. Moreover, defining $\rho := \int_0^{2\pi}f \, \ud \theta$, where $f$ is the aforementioned strong limit in $L^2([0,T];L^2_{per}(\Omega))$, we have, for the convergent subsequence in question, 
\begin{equation*}
    \begin{aligned}
        \Vert \rho - \rho^n \Vert_{L^2([0,T];L^2(\Omega))}^2 &= \int_0^T \int_\Omega \bigg|\int_0^{2\pi} \big( f(t,\x,\theta) - f^n(t,\x,\theta) \big) \, \ud \theta \bigg|^2 \, \ud x \, \ud t \\ 
        &\leq 2\pi \Vert f - f^n \Vert^2_{L^2([0,T];L^2(\Upsilon))} \to 0 \qquad \text{as } n \to \infty, 
    \end{aligned}
\end{equation*}
where we applied Jensen's inequality, thus $\rho^n \to \rho$ strongly in $L^2([0,T];L^2_{per}(\Omega))$. 

This allows us to pass to the limit in the weak form of \cref{eq:variant_model_proof}, thus to obtain that $f$ is a weak solution to the revised \cref{eq:variant_model_proof}. We stress that strong convergence for $a_{0}^n$, \textit{i.e.}, $\{\rho^n\}_{n\in\mathbb{N}}$, is needed in order to allow convergence through the positive part in the drift term. Indeed, the positive part function is continuous and it grows at most linearly, thus $(1-(\rho^n)_+)_+ \to (1-(\rho)_+)_+$ strongly in $L^2([0,T];L^2_{per}(\Omega))$.
More precisely, for a function $\psi\in C^1([0,T])$ such that $\psi(0)=\psi(T)=0$, there holds (using, \textit{e.g.}, \cite[App.~E Th.~8]{evans}) $\int_0^T \langle \partial_t f^n(t),\varphi \rangle_\Upsilon \psi(t) \,\ud t = - \int_0^T \langle f^n(t) , \varphi \rangle_\Upsilon \psi'(t) \,\ud t$, and hence 
\begin{align*}
    \int_0^T\!\langle f^n(t), \!\varphi\rangle_{\Upsilon} \psi'(t)\ud{t}\!=\!&-\!\int_0^T\!\left\langle \Pe(1\!-\! (a_{0}^n(t))_+\!)_+\!f^n(t) \e(\theta),\!\nabla\varphi\right\rangle_{\Upsilon}\!\psi(t)\ud{t}\\
    &+\!D_e\! \int_0^T\!\left\langle \nabla f^n(t),\!\nabla\varphi\right\rangle_{\Upsilon}\!\psi(t)\ud{t}\!+\!\int_0^T\!\left\langle \partial_\theta f^n(t),\! \partial_\theta\varphi\right\rangle_{\Upsilon}\!\psi(t)\ud{t}.
\end{align*}
Using \cref{eq:weak conv}-\cref{eq:weak conv dt}, the aforementioned strong convergence $(1-(\rho^n)_+)_+ \to (1-(\rho)_+)_+$ in $L^2(0,T;L^2_{per}(\Omega))$, we pass to the limit $n\to+\infty$ in the above and obtain, after another integration by parts in $t$, 
\begin{align*}
    \int_0^T\!\left\langle \partial_t f(t), \varphi\right\rangle_{\Upsilon}\psi(t)\ud{t}&\!=\!\int_0^T\left\langle \Pe(1- (a_{0}(t))_+)_+f(t) \e(\theta),\nabla\varphi\right\rangle_{\Upsilon}\psi(t)\ud{t}\\
    &\quad-D_e \int_0^T\left\langle \nabla f(t),\nabla\varphi\right\rangle_{\Upsilon}\psi(t)\ud{t}-\int_0^T\left\langle \partial_\theta f(t), \partial_\theta\varphi\right\rangle_{\Upsilon}\psi(t)\ud{t}.
\end{align*}
Hence, since $\psi$ was arbitrary, $f \in L^2([0,T];H^1_{per}(\Upsilon))$ with $f' \in L^2([0,T];(H^1_{per})'(\Upsilon))$ solves \cref{eq:variant_model_proof} weakly (\textit{i.e.}~duality with $H^1_{per}(\Upsilon)$ for a.e.~$t\in[0,T]$). 

\textit{Step 4 (initial data): } Recall that $f_0$ has Fourier expansion \cref{eq:f0 fourier exp} and $f^n(0,\cdot,\cdot) \to f_0$ strongly in $L^2_{per}(\Upsilon)$. Moreover, by extracting a subsequence after applying the Aubin--Lions Lemma under \cref{eq:weak conv dt}, $f^n(t) \to f(t)$ strongly in $L^2_{per}(\Upsilon)$ for a.e.~$t\in[0,T]$. By \cref{rem-actually defined in L2 at 0}, $f^n,f \in C([0,T];L^2_{per}(\Upsilon))$. Hence, $t=0$ is a Lebesgue point, and $f^n(0) \to f(0)$ strongly in $L^2_{per}(\Upsilon)$. By uniqueness of limits, $f(0) = f_0$ in $L^2_{per}(\Upsilon)$.

  \textit{Step 5 (solution to the original equation and nonnegativity): } As direct consequence of the previous step, by integrating the equation \cref{eq:variant_model_proof} in the angle variable, we infer that $\rho(t,\x)=\int_0^{2\pi}f(\x,t,\theta)\ud{\theta}$ is a curve belonging to $L^\infty([0,T];L^2_{per}(\Omega))\cap L^2([0,T];H^1_{per}(\Omega))$, and is a weak solution of 
\begin{equation}\label{eq:pos part twice}
\partial_t\rho+\Pe\nabla\cdot((1-(\rho)_+)_+\p)= D_e \Delta\rho.
\end{equation}
Moreover, by the weak-* lower semicontinuity of the norm, $\rho'\in L^2([0,T];(H^1_{per})'(\Omega))$ as it is the weak-* limit curve of the sequence $\{{a_{0}^n}'\}_{n\in\N}$. 
Begin by noting that $(1-(\rho)_+)_- = -(1-\rho\mathds{1}_{\rho>0})\mathds{1}_{\rho > 1}$, which implies that, in the sense of distributions, $\partial_t (1-(\rho)_+)_- = \partial_t \rho \mathds{1}_{\rho > 1}$ and $\nabla (1-(\rho)_+)_- = \nabla \rho \mathds{1}_{\rho > 1}$. Testing \cref{eq:pos part twice} with $(1-(\rho)_+)_-$, 
\begin{equation*}
    \frac{1}{2}\frac{d}{dt}\int_\Omega (1-(\rho)_+)_-^2 \, \ud \x = \Pe \int_\Omega \nabla (1-(\rho)_+)_- \cdot (1-(\rho)_+)_+ \p \, \ud \x - D_e\int_\Omega |\nabla \rho|^2 \mathds{1}_{\rho > 1} \, \ud \x. 
\end{equation*}
Observe that the first term on the right-hand side is null, since the supports of the two terms in the integrand are disjoint, while the final term is nonpositive. We thereby deduce $\frac{1}{2}\int_\Omega (1-(\rho(t))_+)_-^2 \, \ud \x \leq 0$, since we initially have $(1-\rho_0(\x))_-=0$ for a.e.~$x\in\Omega$. As a result, $\rho(t,\x) \leq 1$ for a.e.~$(t,\x)\in (0,T)\times\Omega$, as required.

Returning to \cref{eq:variant_model_proof} and testing with $(f)_-$, we get, after integrating by parts, 
\begin{equation*}
    \begin{aligned}
        \frac{1}{2}\!\frac{d}{dt}\!\int_\Upsilon \!(f)_-^2 \ud \boldsymbol{\xi}\! &=\! \Pe\! \int_\Upsilon \!\nabla (\!f\!)_- \!\cdot\! (\!f\!)_- (1\!-\!(\rho)_+\!)_+ \e(\theta) \ud \boldsymbol{\xi} \!-\! D_e \int_\Upsilon |\nabla (\!f\!)_-|^2 \, \ud \boldsymbol{\xi} \!-\! \int_\Upsilon\! |\partial_\theta (\!f\!)_-|^2 \ud \boldsymbol{\xi} \\ 
        &\leq C(\Pe,D_e) \int_\Upsilon \frac{1}{2}(f)_-^2 \, \ud \boldsymbol{\xi} - \frac{1}{2}D_e \int_\Upsilon |\nabla (f)_-|^2 \, \ud \boldsymbol{\xi} - \int_\Upsilon |\partial_\theta (f)_-|^2 \, \ud \boldsymbol{\xi}, 
        \end{aligned}
\end{equation*}
where $\boldsymbol{\xi}=(\boldsymbol{\x},\theta)$. Dropping the two nonpositive terms, Gr\"{o}nwall's Lemma implies 
\begin{equation*}
    \int_\Upsilon (f(t,\x,\theta))_-^2 \, \ud \boldsymbol{\xi} \leq \exp(Ct) \int_\Upsilon (f_0(\x,\theta))_-^2 \, \ud \boldsymbol{\xi} = 0 \qquad \text{a.e.}~t \in (0,T), 
\end{equation*}
since $f_0$ is nonnegative a.e.~on $\Upsilon$. Hence, $f(t,\x,\theta) \geq 0$ for a.e.~$(t,\x,\theta)\in \Upsilon\times(0,T)$ and thus the space density $\rho(t,\x)\in[0,1]$ for a.e. $(t,\x)\in \Upsilon\times(0,T)$. Hence the limiting curve $f$ is a weak solution in the sense of \cref{def:concept of solution} to the original equation \cref{model_intro}. We emphasise that $f \in C([0,T];L^2_{per}(\Upsilon))$ is deduced from $f \in L^2([0,T];H^1_{per}(\Upsilon))$ and $f' \in L^2([0,T];(H^1_{per})'(\Upsilon))$; \textit{cf.}~\cref{rem-actually defined in L2 at 0}. 
\end{proof}

\subsection{Uniqueness of solutions}\label{section:uniqueness}

\begin{lemma}
Given $f_0 \in L^2_{per}(\Upsilon)$ nonnegative such that its corresponding $\rho_0$ satisfies $\rho_0(\x) = \int_0^{2\pi} f_0(\x,\theta) \, \ud \theta \in [0,1]$ for a.e.~$\x\in\Omega$, there exists at most one periodic weak solution, in the sense of \cref{def:concept of solution}, of the problem \cref{eq:weak_form_def}. 
\end{lemma}

\begin{proof}
Let us consider two weak solutions to \cref{model_intro}, $f_1, f_2$, in the sense of \cref{def:concept of solution}, and set $\bar{f}:=f_1-f_2$ to be their difference. We refer to $\rho_1, \rho_2$ as the corresponding space densities, and to $\bar{\rho}:=\rho_1-\rho_2$ as their difference. By means of a direct computation, $\bar{f}$ is a weak solution to
\begin{align*}
\partial_t \bar{f} + \Pe\nabla \cdot ((1- \rho_1)\bar{f} \e(\theta)-\bar{\rho}f_2\e(\theta))=D_e \Delta \bar{f} + \partial_{\theta}^2 \bar{f},
\end{align*}
and $\bar{\rho}$ solves $\partial_t \bar{\rho} + \Pe\nabla \cdot ( (1-\rho_1) \bar{\p}-\bar{\rho}\p_2)= D_e \Delta \bar{\rho}$, where we remind the reader $\bar{\p}(t,\x)=\int_0^{2\pi}\bar{f}(t,\x,\theta)\ud{\theta}$ and $\p_2(t,\x)=\int_0^{2\pi}f_2(t,\x,\theta)\ud{\theta}$, for any $(t,\x,\theta)\in[0,T] \times \Upsilon$. Note that, by directly integrating the previous equation for $\bar{\rho}$, $\int_\Omega \bar{\rho}(\x) \ud \x = 0$, and thus $\int_\Omega \int_0^{2\pi} \bar{f}(t,\x,\theta) \, \ud \theta \, \ud \x = 0$. Motivated by this latter ``mean-zero condition'', we consider the space 
$$
H^1_z(\Upsilon):=\left\{f\in H^1(\Upsilon): \int_\Upsilon f(\xi)\ud\xi=0\right\},
$$
and its dual $(H_z^1)'(\Upsilon)$, where subscript $z$ refers to the mean-zero property in question. Let $L:=(-\Delta)^{-1}:(H^1_z)'(\Upsilon)\to (H^1_z)(\Upsilon)$ be the inverse of the solution operator for the Poisson equation with periodic boundary condition, that is $Lg=u$, where $\int_\Upsilon\nabla u\cdot\nabla\varphi=\left\langle g,\varphi\right\rangle$ for all $\varphi\in H^1_z(\Upsilon)$.
In the following we shall use that $L$ maps continuously $L^2_z(\Upsilon)$ into $H^2_z(\Upsilon)$. Moreover, we note the self-adjointness property 
\[
\left\langle \varphi, L\varphi \right\rangle=\left\langle(-\Delta)\circ(-\Delta)^{-1} \varphi, L\varphi \right\rangle=\|\nabla L\varphi\|_{L^2(\Upsilon)}^2 = \langle L\varphi , \varphi \rangle,
\]
for every $\varphi \in H^1_z(\Upsilon)$, where the duality product is understood in the sense of $L^2(\Upsilon)$. Similarly, using the commutativity of the partial derivatives, for any $\varphi \in H^1_z(\Upsilon)$, 
\begin{equation*}
    \langle \partial_\theta \varphi , \partial_\theta L \varphi \rangle = \langle (-\Delta)\partial_\theta L \varphi ,  \partial_\theta L \varphi \rangle = \Vert \nabla \partial_\theta L \varphi \Vert^2_{L^2(\Upsilon)}. 
\end{equation*}
Let $\lambda>0$. For the next lines of computation, we use \cref{def:concept of solution}, Young's inequality, the bound $\rho_1\le1$, and the self-adjointness property. We drop the time dependence for ease of presentation. We obtain 
\begin{equation}\label{eq:uniq-gronw-1}
    \begin{split}
    \frac{1}{2}\!\frac{\ud}{\ud{t}}\!\left(\!\|\nabla L\bar{f}\|_{L^2(\Upsilon)}^2\!+\!\lambda\|\bar{\rho}\|_{L^2(\Omega)}^2\!\right)&\!=\left\langle\partial_t\bar{f},L\bar{f}\right\rangle+\lambda\left\langle\partial_t\bar{\rho},\bar{\rho}\right\rangle\\
        &\!=\!-D_e\!\|\bar{f}\|_{L^2(\Upsilon)}^2\!-\!\|\nabla\partial_\theta L\bar{f}\|_{L^2(\Upsilon)}^2\!-\!\lambda \! D_e\!\|\nabla\bar{\rho}\|_{L^2}^2\\
        &\quad\!+\!\Pe\left\langle(1-\rho_1)\bar{f}\e(\theta),\nabla L \bar{f}\right\rangle\!-\!\Pe \lambda\!\left\langle\bar{\rho}\p_2,\!\nabla \bar{\rho}\right\rangle\\
        &\quad+\!\Pe \lambda\left\langle(1\!-\!\rho_1)\bar{\p},\nabla \bar{\rho}\right\rangle\!-\!\Pe\!\left\langle\bar{\rho}f_2\e(\theta),\!\nabla L \bar{f}\right\rangle\\
        &\!\le\!-D_e\left(1-\frac{1}{2\varepsilon}-\frac{\lambda}{\sigma}\pi\right)\|\bar{f}\|_{L^2(\Upsilon)}^2\\
        &\quad-\!\lambda\!\left(\!D_e\!-\!\frac{\Pe^2\sigma}{D_e}\right)\!\|\nabla\bar{\rho}\|_{L^2(\Omega)}^2 \!+\!\frac{\lambda D_e}{2\sigma}\|\bar{\rho}\|_{L^2(\Omega)}^2\\
        &\quad+\frac{\Pe^2\varepsilon}{2D_e}\|\nabla L\bar{f}\|_{L^2(\Upsilon)}^2\!-\!\Pe\!\left\langle\bar{\rho}f_2\e(\theta),\!\nabla L \bar{f}\right\rangle, 
    \end{split}
\end{equation}
where we also used $\|\bar{\p}\|_{L^2(\Omega)}^2\le2\pi\|\bar{f}\|_{L^2(\Upsilon)}^2$, and noticed that for any $(t,\x)\in\Omega\times[0,T]$, $|\p_2(t,\x)|\le\int_0^{2\pi}f_2(t,\x,\theta)\ud{\theta}=\rho_2(t,\x)\le1 $, which implies $\|\bar{\rho}\p_2\|_{L^2(\Omega)}^2\le\|\bar{\rho}\|_{L^2(\Omega)}^2$. In order to apply the Gr\"{o}nwall inequality and conclude the argument we must estimate the final term in \cref{eq:uniq-gronw-1}. Since the operator $L$ maps continuously $L^2_z(\Upsilon)$ into $H^2_z(\Upsilon)$, we have $L\bar{f}\in H^2_z(\Upsilon)$, thus $\nabla L\bar{f}\in H^1_z(\Upsilon)$. By means of the Sobolev embedding $H^1(\Upsilon) \hookrightarrow L^{p_*}(\Upsilon)$ for $p_*=6$, and that $|\Upsilon|<+\infty$, we infer that $\nabla L\bar{f}\in L^{p'}(\Upsilon)$ for any $p'\le p_*=6$, and $f_2\in L^q(\Upsilon)$ for any $q\le p_*=6$. In particular, for $\frac{6}{5}\le p\le\frac{3}{2}$ and $q=\frac{2p}{2-p}\le 6$
, using the Cauchy--Young inequality to get from the second line to the third, and the Gagliardo--Nirenberg inequality to get the final line, we see that the term $\Pe\left|\left\langle\bar{\rho}f_2\e(\theta),\nabla L \bar{f}\right\rangle\right|$ is bounded by 
\begin{equation}\label{eq:uniq-sob-embedding}
\begin{split}
    \Pe\|\bar{\rho}f_2\|_{L^p(\Upsilon)}\|\nabla L\bar{f}\|_{L^{p'}(\Upsilon)} &\le C(\Pe,|\Upsilon|,p',p_*)\|\bar{f}\|_{L^2(\Upsilon)}\|\bar{\rho}f_2\|_{L^p(\Upsilon)}\\
    &\le \frac{\lambda D_e}{2}\|\bar{f}\|_{L^2(\Upsilon)}^2+\frac{\tilde{c}}{\lambda D_e}\|f_2\|_{H^1(\Upsilon)}^2\|\bar{\rho}\|_{L^2(\Omega)}^2,
\end{split}
\end{equation}
where the constants $\lambda,\tilde{c}$ depend only on 
$\Upsilon,p,q$ via the embeddings. Using \cref{eq:uniq-sob-embedding} in \cref{eq:uniq-gronw-1}, we get 
\begin{equation*}
    \begin{split}
     & \frac{1}{2}\frac{\ud}{\ud{t}}\left(\|\nabla L\bar{f}\|_{L^2(\Upsilon)}^2+\lambda\|\bar{\rho}\|_{L^2(\Omega)}^2\right) \\ & \qquad \le-D_e\left(1-\frac{\lambda}{2}-\frac{1}{2\varepsilon}-\frac{\lambda}{\sigma}\pi\right)\|\bar{f}\|_{L^2(\Upsilon)}^2 
      -\lambda\left(D_e-\frac{\Pe^2\sigma}{D_e}\right)\|\nabla\bar{\rho}\|_{L^2(\Omega)}^2\\
        &\qquad \quad+\frac{\Pe^2\varepsilon}{2D_e}\|\nabla L\bar{f}\|_{L^2(\Upsilon)}^2+\frac{\lambda D_e}{2\sigma}\|\bar{\rho}\|_{L^2(\Omega)}^2 
        +\frac{\tilde{c}}{\lambda D_e}\|f_2\|_{H^1(\Upsilon)}^2\|\bar{\rho}\|_{L^2(\Omega)}^2\\
        &\qquad \le \frac{M(f_2)}{\lambda^2\sigma}\left(\|\nabla L\bar{f}\|_{L^2(\Upsilon)}^2+\lambda\|\bar{\rho}\|_{L^2(\Omega)}^2\right),
    \end{split}
\end{equation*}
where we set $\sigma=\frac{D_e^2}{\Pe^2}$,$\lambda=({1+\frac{\Pe^2}{D_e^2}\pi})^{-1}$, $\varepsilon=\frac{1}{\lambda}$, and $M(f_2):=\max\{\frac{\Pe^2\lambda\sigma}{2D_e},\frac{\lambda^2D_e}{2}+\frac{\tilde{c}\sigma}{D_e}\|f_2\|_{H^1(\Upsilon)}^2\}$. Finally, since $f_2\in L^2([0,T];H^1_{per}(\Upsilon))$, by means of Gr\"{o}nwall inequality we obtain uniqueness, as for any $t\in[0,T]$
\[
\|\nabla L\bar{f}(t)\|_{L^2(\Upsilon)}^2+\lambda\|\bar{\rho}(t)\|_{L^2(\Omega)}^2\le\underbrace{\left(\|\nabla L\bar{f}(0)\|_{L^2(\Upsilon)}^2+\lambda\|\bar{\rho}(0)\|_{L^2(\Omega)}^2\right)}_{=0}e^{2\int_0^t\frac{M(f_2(s))}{\lambda^2\sigma}\ud{s}}.
\]
\end{proof}

\subsection{Existence and uniqueness of solutions to the 1D Model}\label{sec:well-posed-Goldstein-taylor}

The analysis carried out previously for \cref{model_intro} can be applied to obtain well-posedness for the 1D model \cref{GT_model}. Therefore, we will not provide proofs but only point out the main differences.

Let $\fL(x,t)$ and $\fR(x,t)$ be the densities of the left- and right-moving particles, respectively. The (rescaled) one dimensional version of \cref{model_intro} is
\begin{equation}\label{1d-phen}
\begin{cases}
\partial_t f_R + \Pe\partial_x (f_R (1-\rho)) = \partial_{xx} f_R + f_L - f_R \\
\partial_t f_L - \Pe\partial_x (f_L (1-\rho)) = \partial_{xx} f_L + f_R - f_L.
\end{cases}
\end{equation}
The \emph{space density} and the \emph{polarisation} are, respectively,
\[
\rho:=f_R+f_L, \qquad p:= f_R-f_L.
\]
\begin{definition}[Weak solution to 1D model]\label{def:1d-phen-sol}
Let $f_{R,0},f_{L,0}\in L^2_{per}([0,2\pi])$ be nonnegative functions such that
\[
\rho_0(x)=f_{R,0}(x)+f_{L,0}(x) \in [0,1],\qquad \mbox{for a.e } x\in[0,2\pi].
\]
A \emph{weak solution} to system \cref{1d-phen} is a pair of curves $(f_R,f_L)$ such that $f_R,f_L\in L^2([0,T];H^1_{per}([0,2\pi]))$ and $f'_R, f'_L\in L^2([0,T];(H^1_{per})'([0,2\pi]))$, and for a.e. $t\in[0,T]$ and any $\varphi,\phi \in H^1_{per}([0,2\pi])$ it holds
    \begin{equation*}\label{eq:weak-1d-phen}
    \begin{split}
    &\langle\partial_t f_R(t),\varphi\rangle_\Omega\! =\! \Pe\langle f_R(t) (1\!-\!\rho(t)),\partial_x\varphi\rangle_\Omega\!-\!\langle\partial_{x} f_R(t),\partial_x\varphi\rangle_\Omega\!+\langle f_L(t)\! -\! f_R(t),\varphi\rangle_\Omega \\
    &\langle\partial_t f_L(t),\phi\rangle_\Omega \!=\!- \Pe\langle f_L(t) (1\!-\!\rho(t)) ,\partial_x\phi\rangle_\Omega \!-\!\langle\partial_{x} f_L(t),\partial_x\phi\rangle_\Omega\!+ \langle f_R(t)\! -\! f_L(t),\phi\rangle_\Omega,\\
    &(f_R,f_L)(x,0)=(f_{R,0},f_{L,0}),
    \end{split}
    \end{equation*}
    with periodic boundary conditions on $\Omega$, where $\rho(t,x)=f_R(t,x)+f_L(t,x)$, and the initial data is achieved in the sense $f_R(0)=f_{R,0}$ and $f_L(0)=f_{L,0}$ in $L^2_{per}([0,2\pi])$. 
\end{definition}


Notice that adding/subracting the equations in system \cref{1d-phen} yields 
\begin{equation}\label{eq:system-rho-p}
\begin{cases}
    \partial_t \rho + \Pe\partial_x ( p (1-\rho)) =  \partial_{xx} \rho\\
    \partial_t p + \Pe\partial_x (\rho (1-\rho)) =  \partial_{xx} p - 2 p.
\end{cases}
\end{equation}

As already mentioned, existence and uniqueness of solutions to \cref{1d-phen} can be proved with a strategy similar to that used for \cref{model_intro}. The main difference is the absence of the angular variable, which simplifies the discretisation of the system; indeed the angular variable is already discretised, as there are only two possible directions (left and right). It therefore suffices to write a Galerkin approximation only in space, as in \cref{subsec:existence of parabolic system further galerkin}. Moreover, the solution $(\fR, \fL)$ directly inherits the bound of the density $\rho$ (unlike in the higher dimensional case). This simplifies the uniqueness proof, since we need not employ the operator $(-\Delta)^{-1}$. The result is stated below. 

\begin{theorem}[Well-posedness for 1D model]\label{thm:exist-uniq-phenom1d}
Let $(f_{R,0},f_{L,0})\in (L^2_{per}([0,2\pi]))^2$ be a pair of nonnegative functions such that
\[
\rho_0(x)=f_{R,0}(x)+f_{L,0}(x)\in[0,1] \quad \mbox{for a.e. } x\in [0,2\pi].
\]
For any $T>0$, there exists a unique weak solution to \cref{1d-phen} in the sense of \cref{def:1d-phen-sol}. Moreover, $f_R(t,x),f_L(t,x)\in[0,1]$ for a.e. $t\in[0,T]$ and $x\in [0,2\pi]$.
\end{theorem}

\begin{remark}\label{rem:continuity in L2 only 1d}
As per \cref{rem-actually defined in L2 at 0}, we deduce $f_R,f_L \in C([0,T];L^2_{per}([0,2\pi]))$. 
\end{remark}

\section{Regularity}\label{sec:regularity}

Having established well-posedness in \cref{sec:well-posedness}, we turn to the issue of regularity. More precisely, we provide here the proof of \cref{thm:regularity-3d}, which is performed in \cref{lem:regularity 2D,lem:regularity-2d-final}. For ease of presentation we start with the one-dimensional model.

\subsection{Regularity for the 1D model}\label{sec-regularity-1d-goldstein-taylor}

Our strategy is to apply the Duhamel principle so as to treat \cref{1d-phen} as a forced one-dimensional heat equation. Inspired by \cite[Chapter 9]{cannon}, we make the following definition. 

\begin{definition}\label{def:periodic heat kernel}
Define $\Phi : (0,\infty)\times \mathbb{R} \to \mathbb{R}$ by the explicit formula 
\begin{equation}\label{eq:periodic heat kernel def}
    \Phi(t,x) := \frac{1}{2\pi} + \frac{1}{\pi}\sum_{n=1}^\infty e^{-n^2 t} \cos nx \qquad \forall (t,x) \in (0,\infty)\times \mathbb{R}. 
\end{equation}
\end{definition}

Note that the above may be rewritten in terms of the complex exponential basis as $\Phi(t,x) = \frac{1}{2\pi}+\sum_{n\in\mathbb{Z}\setminus \{0\}} \frac{e^{-n^2 t}}{2\pi} e^{inx}$ for all $(t,x) \in (0,\infty)\times\mathbb{R}$. In what follows it will be useful to rewrite $\Phi$ in terms of the Jacobi theta-3 function (denoted by $\theta_3$), and in terms of the Jacobi theta function (denoted by $\vartheta$, also sometimes by $\vartheta_{00}$), 
\begin{equation}\label{eq:in terms of jacobi}
    \Phi(t,x) = \frac{1}{2\pi}\theta_3\left(\frac{x}{2},e^{-t}\right) = \frac{1}{2\pi}\vartheta\left(\frac{x}{2\pi},\frac{it}{\pi}\right) \qquad \forall (t,x) \in (0,\infty) \times \mathbb{R}. 
\end{equation}

\begin{lemma}\label{lem:properties of periodic kernel}
The function $\Phi$ of \cref{def:periodic heat kernel} is a smooth $2\pi$-periodic function that satisfies $\partial_t \Phi - \partial_{xx}\Phi = 0 $ as a pointwise equality in $(0,\infty)\times \mathbb{R}$, is nonnegative in $(0,\infty)\times\mathbb{R}$, and $\Vert \Phi(t,\cdot) \Vert_{L^1([0,2\pi])} = 1$ for every $t \in (0,\infty)$. Additionally, 
\begin{equation}\label{eq:L2L2 for kernel}
    \int_0^t \Vert \Phi(\tau,\cdot) \Vert^2_{L^2([0,2\pi])} \, \ud \tau = t + \sum_{n=1}^\infty \frac{1}{2n^2}(1-e^{-2n^2 t}) \qquad \forall t > 0, 
\end{equation}
and $\int_0^t \int_0^{2\pi} \Phi(t-s,x-y)^2 \, \ud y \, \ud s = \int_0^t \Vert \Phi(\tau,\cdot) \Vert^2_{L^2([0,2\pi])} \, \ud \tau$ for all $(t,x) \in (0,T]\times[0,2\pi]$. Finally, for any $2\pi$-periodic $C^2$ function $\psi$, $\lim_{t\to0^+}\int_0^{2\pi}\Phi(t,x)\psi(x) \, \ud x = \psi(0)$. 
\end{lemma}

The proof of \cref{lem:properties of periodic kernel} is in \cref{sec:complementary-sec-regularity}. We use the function $\Phi$ as a periodic heat kernel, in the sense made precise by the following result. The initial data $\psi$ is chosen in $L^1_{loc}(\mathbb{R})$ to justify the expansion in terms of the sine/cosine basis. 

\begin{lemma}\label{lem:generate periodic solutions}
Let $\psi \in L^1_{per}([0,2\pi])$. Then, $\Psi(t,x) := \int_0^{2\pi} \Phi(t,x-y)\psi(y) \, \ud y$ for all $(t,x) \in (0,\infty)\times\mathbb{R}$ is a smooth $2\pi$-periodic function defined on the full-space $(0,\infty)\times\mathbb{R}$, which solves the Cauchy problem 
\begin{equation}
    \left\lbrace\begin{aligned}
        & \partial_t \Psi - \partial_{xx} \Psi = 0 \qquad &&\text{in } (0,\infty)\times\mathbb{R}, \\ 
        & \Psi|_{t=0} = \psi \qquad && \text{on } \mathbb{R}, 
    \end{aligned}\right.
\end{equation}
where the initial data is achieved in $L^1_{loc}$, \textit{i.e.}, $\lim_{t \to 0^+} \Vert \Psi(t,\cdot) - \psi \Vert_{L^1([0,2\pi])} = 0$. 
\end{lemma}

The proof of this result is in \cref{sec:complementary-sec-regularity}. An alternative way of generating periodic solutions of the heat equation is to convolve the classical one-dimensional heat kernel with $2\pi$-periodic test functions. However, when performing integration by parts in the Duhamel formula (\textit{cf.}~\cref{eq:fr duhamel} in the lemma below) the boundary terms do not vanish unless both the kernel and the ``test function'' are periodic. Hence, the kernel of \cref{def:periodic heat kernel} is better suited to our analysis.

\begin{lemma}[Regularity away from initial time]\label{lem:regularity 1D}
Suppose that the periodic initial data for problem \cref{1d-phen} is such that $f_{R,0},f_{L,0} \in L^\infty(\Omega)$. Then, the unique weak solutions $f_R$ and $f_L$ of \cref{1d-phen} belong to $C((0,T]\times \bar{\Omega})$. 
\end{lemma}
\begin{proof}
Recall from \cref{thm:exist-uniq-phenom1d} that we have $f_R,f_L , \rho \in L^\infty([0,T]\times\Omega)$. 

The key observation is that the equations \cref{1d-phen} can be interpreted as heat equations with a source on the right-hand side. Since we already derived the fundamental solution of the periodic heat equation in  \cref{lem:properties of periodic kernel,lem:generate periodic solutions}, and we have already shown uniqueness in the relevant space of solutions to \cref{1d-phen}, we are in a position to apply the Duhamel principle and obtain implicit integral relations for $f_R$ and $f_L$. In the first step of the proof, we justify that these Duhamel formulas are well-defined. In a second step, we prove the continuity of these functions away from the initial time.

  \textit{Step 1 (Duhamel formula): } We justify that the formula for $f_R$ 
\begin{equation}\label{eq:fr duhamel}
    \begin{aligned}
    f_R(t,x) \!=\! & \int_0^{2\pi} \Phi(t,x-y) f_{R,0}(y) \, \ud y \\
    - & \!\int_0^t \!\int_0^{2\pi} \! \Phi(t\!-\!s,\!x\!-\!y) \Big( \Pe \partial_y \{  f_R(s,y)[1-\rho(s,y) ]\} + f_R(s,y) \!-\! f_L(s,y)  \Big) \ud y  \ud s ,
    \end{aligned}
\end{equation}
is well-defined, for a.e.~$(t,x) \in [0,T]\times \bar{\Omega}$. To begin with, whenever $t>0$, the function $\Phi(t,\cdot)$ is smooth and is therefore square-integrable on the compact interval $[0,2\pi]$. 
This shows that the integrand in the first integral in \cref{eq:fr duhamel} is integrable. Meanwhile, the equality \cref{eq:L2L2 for kernel} shows that $\Phi \in L^2([0,T];L^2([0,2\pi]))$, whence, using the boundedness of $\rho$, 
\begin{equation*}
    \begin{aligned}
    &\int_0^t\int_0^{2\pi} \Phi(t-s,x-y) \big |\partial_y \{ f_R(s,y)[1-\rho(s,y)] \} \big | \, \ud y \, \ud s \\ 
    &\leq \!2\!\int_0^t\! \bigg[\! \int_0^{2\pi}\! \Phi(t-s,x-y)^2 \, \ud y \!\bigg]^{1/2}\! \bigg[\! \int_0^{2\pi}\! |\partial_y f_R(s,y)|^2\! +\! f_R(s,y)^2 |\partial_y \rho(s,y)|^2 \, \ud y \bigg]^{1/2} \, \ud s \\ 
    &\leq \!\int_0^t\! \int_0^{2\pi}\! \Phi(t\!-\!s,\!x\!-\!y)^2 \, \ud y \, \ud s \!+\! \Vert f_R \Vert^2_{L^2(0,T;H^1(\Omega))} \!+\! \int_0^t \!\int_0^{2\pi}\! f_R(s,y)^2 |\partial_y \rho(s,y)|^2 \, \ud y \, \ud s, 
    \end{aligned}
\end{equation*}
where we have used the Young inequality. As before, we use a change of variable and the constraint $x \in [0,2\pi]$ to bound the first term on the right-hand side of the above by $2 \Vert \Phi \Vert^2_{L^2(0,T;L^2([0,2\pi]))}$. Meanwhile, 
\begin{equation*}
    \int_0^t \int_0^{2\pi} f_R(s,y)^2 |\partial_y \rho(s,y)|^2 \, \ud y \, \ud s \leq \Vert f_R \Vert^2_{L^\infty((0,T)\times\Omega)} \Vert \rho \Vert^2_{L^2(0,T;H^1(\Omega))}, 
\end{equation*}
using that $f_R \in L^\infty((0,T)\times\Omega)$ from \cref{thm:exist-uniq-phenom1d}. Similarly, using \cref{lem:properties of periodic kernel}
\begin{equation*}
    \int_0^t\! \int_0^{2\pi}\! \Phi(t\!-\!s,\!x\!-\!y) |f_R(s,y)\! -\! f_L(s,y)| \, \ud y \, \ud s \leq 2\Vert \Phi \Vert_{L^2(0,T;L^2(\Omega))} \Vert f_R\! -\! f_L \Vert_{L^2(0,T;L^2(\Omega))}. 
\end{equation*}
Therefore, the integrand in the second integral of \cref{eq:fr duhamel} is also integrable. It is an exercise to verify that \cref{eq:fr duhamel}, and the equivalent formula for $f_L$ satisfy \cref{1d-phen} and are periodic on $\Omega$. These formulas therefore coincide with the unique solutions of \cref{1d-phen}.

  \textit{Step 2 (continuity away from initial time): } Fix $(t,x) \in (0,T]\times\bar{\Omega}$ and let $(t_n,x_n) \in (0,T]\times\Omega$ be converging to $(t,x)$. Using the integrability from Step 1, the continuity of $\Phi$ for $t>0$, and the Dominated Convergence Theorem on \cref{eq:fr duhamel} and the analogous formula for $f_L$, we directly obtain $|f_R(t_n,x_n) - f_R(t,x)| \to 0$ and $|f_L(t_n,x_n) - f_L(t,x)| \to 0$ as $n\to\infty$. It follows that $f_R,f_L \in C((0,T]\times \bar{\Omega})$. In turn, we also obtain $\rho \in C((0,T]\times \bar{\Omega})$. 
\end{proof}

\begin{lemma}\label{lem:reg at initial and final time}
Suppose that $f_{R,0},f_{L,0} \in C(\bar{\Omega})$ and are periodic. Then, the unique weak solutions $f_R$ and $f_L$ of \cref{1d-phen} belong to $C([0,T]\times\bar{\Omega})$. 
\end{lemma}
\begin{proof}
We only write the proof for $f_R$, as the proof for $f_L$ is identical. In view of the uniform continuity of $f_R$ on the compact $\bar{\Omega}$, given $\varepsilon>0$ there exists $\delta=\delta(\varepsilon)$ for which if $x,y \in [0,2\pi]$ are such that $|x-y|<\delta$, then  $|f_{R,0}(x)-f_{R,0}(y)| < \varepsilon/4$. Now, given $t \in (0,\infty)$ and $x^* \in (0,2\pi)$, there holds from \cref{eq:fr duhamel}
\begin{equation*}
    \begin{aligned}
    |f_R(t,x) - f_{R,0}(x^*)| \leq& \bigg| \int_0^{2\pi} \Phi(t,x-y)  f_{R,0}(y) \, \ud y - f_{R,0}(x^*) \bigg| \\ 
    &+ \Pe \int_0^t \int_0^{2\pi}  \Phi(t-s,x-y) |\partial_y \big( f_R(s,y) (1-\rho(s,y) ) \big) | \, \ud y \, \ud s \\ 
    &+ k \int_0^t \int_0^{2\pi} \Phi(t-s,x-y) \big| f_R(s,y) - f_L(s,y) \big| \, \ud y \, \ud s \\ 
    =:& I_1(t,x) + I_2(t,x) + I_3(t,x), 
    \end{aligned}
\end{equation*}
where we used the nonnegativity of $\Phi$ and $\int_0^{2\pi}\Phi(t,x) \, \ud x = 1$ for every $t>0$ from \cref{lem:properties of periodic kernel}. Using H\"{o}lder's inequality, the $L^2$-boundedness of $\partial_x f_R,f_R,f_L$, and \cref{lem:properties of periodic kernel}, we obtain $\lim_{t\to0^+}\Vert I_2(t,\cdot)\Vert_{L^\infty([0,2\pi])} = \lim_{t\to0^+}\Vert I_3(t,\cdot)\Vert_{L^\infty([0,2\pi])} = 0$.

It remains to control $I_1$. Recall that since $f_{R,0} \in C(\bar{\Omega})$ is periodic, there exists a sequence $(\psi_m)_{m\in\mathbb{N}}$ of periodic elements of $C^2(\bar{\Omega})$ such that $\Vert f_{R,0} - \psi_m \Vert_{L^\infty(\bar{\Omega})} \to 0$ as $n \to \infty$. Then, since $\psi_m \in C^2(\bar{\Omega})$ for each $m\in\mathbb{N}$, we know that it has Fourier coefficients $a_{m,n},b_{m,n}$ that decay as $n^{-2}$; \textit{cf.}~\cref{eq:exp psi fourier for reg}-\cref{eq:decay fourier coeffs an bn}. Then, using the convolution result for Fourier series, 
\begin{equation*}
    \begin{aligned}
        \bigg| \psi_m(x) - \int_0^{2\pi} \Phi(t,x-y)\psi_m(y) \, \ud y \bigg| &= \frac{1}{\pi}\bigg| \sum_{n=1}^\infty (1-e^{-n^2 t})\big(a_{m,n} \cos nx + b_{m,n} \sin nx \big) \bigg| \\ 
        &\leq C_m \sum_{n=1}^\infty \frac{1}{n^2}(1-e^{-n^2 t}) \to 0 \qquad \text{as } t \to 0^+, 
    \end{aligned}
\end{equation*}
where we applied the Dominated Convergence Theorem applied to the atomic measure in the final line. Note that 
\begin{equation*}
    \begin{aligned}
        I_1(t&,x) \leq |f_{R,0}(x^*) - f_{R,0}(x)| + |f_{R,0}(x) - \psi_m(x)| \\ 
        &+ \bigg| \psi_m(x) - \int_0^{2\pi} \Phi(t,x-y) \psi_m(y) \, \ud y \bigg| + \bigg| \int_0^{2\pi} \Phi(t,x-y) \big(\psi_m(y) - f_{R,0}(y)\big) \, \ud y \bigg|, 
    \end{aligned}
\end{equation*}
and elementary manipulations (\textit{cf.}~Step 3 of Proof of \cref{lem:properties of periodic kernel}) relying also on the periodicity of $\psi_m$ yield $\int_0^{2\pi}\Phi(t,x-y)\psi_m(y) \, \ud y = \int_0^{2\pi}\Phi(t,x-y)\psi_m(x+y) \, \ud y$. Thus, given $\varepsilon>0$, by first picking $x$ such that $|x-x^*|<\delta$ and then choosing $m$ sufficiently large such that $\Vert f_{R,0} - \psi_m \Vert_{L^\infty(\bar{\Omega})} < \varepsilon/4$, and finally taking $t$ sufficiently close to zero for this particular $m$ (\textit{cf.}~convergence to initial data of \cref{lem:properties of periodic kernel}), it follows that 
\begin{equation*}
    I_1(t,x) < \frac{3}{4}\varepsilon
        + \Vert \psi_m - f_{R,0}\Vert_{L^\infty(\bar{\Omega})} \int_0^{2\pi} \Phi(t,x-y)  \, \ud y < \varepsilon, 
\end{equation*}
where we used the unit integral property of $\Phi$. Hence, $\lim_{(t,x) \to (0^+,x^*)}I_1(t,x) = 0$, and we conclude $\lim_{(t,x) \to (0^+,x^*)} |f_{R,0}(x^*) - f_R(t,x)| = 0$, which verifies the required continuity at the initial time. The same argument shows the result for $f_L$. 
\end{proof}

\begin{remark}\label{rem:higher reg maybe possible}
We showed that the weak solutions of the 1D model, which (\textit{cf.}~\cref{rem:continuity in L2 only 1d}) in principle only belong to $C([0,T];L^2_{per}([0,2\pi]))\cap L^2([0,T];H^1_{per}([0,2\pi])$, are in fact continuous under mild assumptions on the initial data. Although the authors suspect that it is possible to show higher regularity of the solutions, the method of proof required for such a result would be rather different to the current approach. For this reason, we leave this for further investigation. 
\end{remark}

\subsection{Regularity for the 2D model}

In accordance with the expression for the one-dimensional periodic kernel in terms of complex exponentials (\textit{cf.}~\cref{sec-regularity-1d-goldstein-taylor}), we find that the periodic heat kernel in three spatial dimensions (\textit{i.e.}~when $\Upsilon = [0,2\pi]^3$) is given by the explicit formula 
\begin{equation}\label{3d kernel}
    \tilde{\Phi}(t,\x) = \frac{1}{(2\pi)^3} + \sum_{\mathbf{n}\in\mathbb{Z}^3\setminus \{0\}} \frac{e^{-|\mathbf{n}|^2 t}}{(2\pi)^3}e^{i \mathbf{n} \cdot \x} \qquad \forall (t,\x) \in (0,\infty)\times \mathbb{R}^3. 
\end{equation}
Note that the above may be rewritten in terms of the Jacobi theta-3 function as $\tilde{\Phi}(t,\x) = 
         (2\pi)^{-3} \theta_3( {x_1}/{2} , e^{-t} ) \theta_3( {x_2}/{2} , e^{-t} ) \theta_3( {x_3}/{2} , e^{-t} )$ for every $(t,\x) \in (0,\infty) \times\mathbb{R}^3$. As per the proof of \cref{lem:properties of periodic kernel}, we obtain that $\tilde{\Phi}$ is nonnegative in $(0,\infty)\times\mathbb{R}^3$ with 
\begin{equation}
    \Vert \tilde{\Phi}(t,\cdot) \Vert_{L^1(\Upsilon)} = 1 \quad \text{and} \quad \int_0^t \Vert \tilde{\Phi}(\tau,\cdot) \Vert^2_{L^2(\Upsilon)} \, \ud \tau = t + \sum_{\mathbf{n}\in\mathbb{Z}^3}\frac{1}{2|\mathbf{n}|^2}(1-e^{-|\mathbf{n}|^2 t}) \qquad \forall t > 0, 
\end{equation}
and that $\tilde{\Phi}$ is a smooth $\Upsilon$-periodic solution of the heat equation in $(0,\infty)\times\mathbb{R}^3$. 
In turn, an identical proof to that of \cref{lem:generate periodic solutions} yields the following result. 

\begin{lemma}\label{lem:generate periodic solutions 2D}
Let $\psi \in L^1_{per}(\Upsilon)$. Then, $\tilde{\Psi}(t,\boldsymbol{\xi}) := \int_\Upsilon \tilde{\Phi}(t,\boldsymbol{\xi}-\mathbf{z})\psi(\mathbf{z}) \, \ud \mathbf{z}$ for all $(t,\boldsymbol{\xi}) \in (0,\infty)\times\mathbb{R}^3$ is a smooth $\Upsilon$-periodic function defined on the full-space $(0,\infty)\times\mathbb{R}^3$, and solves the Cauchy problem 
\begin{equation}
    \left\lbrace\begin{aligned}
        & \partial_t \tilde{\Psi} - \Delta_{\boldsymbol{\xi}} \tilde{\Psi} = 0 \qquad &&\text{in } (0,\infty)\times\mathbb{R}^3, \\ 
        & \Psi|_{t=0} = \psi \qquad && \text{on } \mathbb{R}^3, 
    \end{aligned}\right.
\end{equation}
where the initial data is achieved in $L^1_{loc}$, \textit{i.e.}, $\lim_{t\to 0^+}\Vert \tilde{\Psi}(t,\cdot) - \psi \Vert_{L^1(\Upsilon)} = 0$. 
\end{lemma}

In turn, we have the following regularity result. 

\begin{lemma}\label{lem:regularity 2D}
Suppose the initial data for problem \cref{model_intro} is such that $f_0 \in L^\infty_{per}(\Upsilon)$. Then unique weak solution $f$ of \cref{model_intro} supplied by \cref{thm:exist-uniq-phenom} belongs to $C((0,T]\times\bar{\Upsilon})$. 
\end{lemma}
\begin{proof}
There are two steps to the proof: the first is to show the boundedness of $f$ (which requires building an appropriate test function); the second is to argue via a Duhamel formula (\textit{cf.}~Proof of \cref{lem:regularity 1D}).

  \textit{Step 1 (boundedness of $f$): } Recall from \cref{thm:exist-uniq-phenom} that the weak solution of \cref{model_intro}, $f$, a priori only belongs to the space $L^\infty([0,T];L^2_{per}(\Upsilon))\cap L^2([0,T];H^1_{per}(\Upsilon))$. We show that $f \in L^\infty((0,T)\times\Upsilon)$.

For any $k>0$, one verifies that $(f-k)_+ \in L^\infty([0,T];L^2_{per}(\Omega))\cap L^2([0,T];H^1_{per}(\Upsilon))$ and $\partial_t (f-k)_+ \in L^2([0,T];(H^1_{per})'(\Upsilon))$. Indeed, there holds $(f-k)_+ = \mathds{1}_{\{f>k\}} (f-k)$, from which it follows that, in the sense of distributions, $\nabla_{t,\x,\theta}(f-k)_+ = \mathds{1}_{\{f>k\}} \nabla_{t,\x,\theta} f$ a.e.~in $(0,T)\times\Upsilon$, where $\nabla_{t,\x,\theta}$ is the vector of derivatives $(\partial_t , \nabla , \partial_\theta)$, and we directly deduce the boundedness of the relevant norms. Thus, $(f(t)-k)_+$ is an admissible test function for the weak formulation of \cref{def:concept of solution}. Note additionally that, by \cite[Ch.~2 Lem.~1.2]{bookTemam77}, there holds in the sense of distributions, for a.e.~$t\in(0,T)$, $\frac{d}{dt} \frac{1}{2}\int_\Upsilon (f(t)-k)_+^2 \, \ud \x = \langle\partial_t f(t), (f(t)-k)_+ \rangle_{\Upsilon}$. Using $(f(t)-k)_+$ as the test function in the weak formulation of \cref{def:concept of solution}, integrating in $\x$, using the boundedness of $\rho$, H\"{o}lder's inequality, and the Young inequality, we get \begin{equation*}
    \begin{aligned}
        \frac{d}{dt}\left(\frac{1}{2}\Vert (f(t)-k)_+ \Vert_{L^2(\Upsilon)}^2 \right) + \int_\Upsilon \bigg( \frac{D_e}{2} |\nabla (f(t)-k)_+|^2 + & |\partial_\theta (f(t)-k)_+|^2 \bigg) \, \ud \x \, \ud \theta \\ 
        &\leq \frac{ \Pe^2}{2D_e} \Vert (f(t)-k)_+ \Vert_{L^2(\Upsilon)}^2, 
    \end{aligned}
\end{equation*}
for a.e.~$t\in(0,T)$. It follows from the Gr\"{o}nwall Lemma that there holds 
\begin{equation*}
    \begin{aligned}
        \frac{1}{2}\Vert (f(t)-k)_+ \Vert_{L^2(\Upsilon)}^2 \leq \frac{1}{2}\Vert (f_0-k)_+ \Vert_{L^2(\Upsilon)}^2 \exp\left({\frac{ \Pe^2}{D_e}t}\right) \qquad \text{a.e.~}t\in(0,T). 
    \end{aligned}
\end{equation*}
By selecting $k = \Vert f_0 \Vert_{L^\infty(\Upsilon)}$, the right-hand side of the above is null. It follows that $f(t) \leq \Vert f_0 \Vert_{L^\infty(\Upsilon)}$ for a.e.~$t\in(0,T)$, \textit{i.e.}, $\Vert f \Vert_{L^\infty((0,T)\times\Upsilon)} \leq \Vert f_0 \Vert_{L^\infty(\Upsilon)}$.

  \textit{Step 2 (Duhamel formula for $f$): } Without loss of generality for this final step we only consider the case where we have normalised the \cref{model_intro} by the positive constant $D_e$, so that $\partial_t - D_e\Delta - \partial^2_\theta$ may be rewritten as $\partial_t - \Delta_{\boldsymbol{\xi}}$ with $\boldsymbol{\xi}$ the usual 3-dimensional vector. 

As per the proof of \cref{lem:regularity 1D}, in view of the uniqueness provided by \cref{thm:exist-uniq-phenom}, it follows from \cref{model_intro} and the Duhamel principle that $f$ satisfies the following implicit equation, where the kernel $\tilde{\Phi}$ was defined in \cref{3d kernel} (\textit{cf.}~\cref{lem:generate periodic solutions 2D}), 
\begin{equation}\label{eq:f duhamel}
    f(t,\boldsymbol{\xi})\!=\!\int_\Upsilon \tilde{\Phi}(t,\boldsymbol{\xi}-\mathbf{z}) f_0(\mathbf{z}) \, \ud \mathbf{z}+ \Pe \int_0^t \int_\Upsilon  \tilde{\Phi}(t-s,\boldsymbol{\xi}-\mathbf{z}) \nabla_{\mathbf{z}}\cdot  \big( f(s,\mathbf{z}) (1-\rho(s,\mathbf{z})  )\mathbf{e}(\theta) \big) \ud \mathbf{z} \ud s
\end{equation}
for a.e.~$(t,\boldsymbol{\xi}) \in [0,T]\times \bar{\Upsilon}$. Note that the above is well-defined: using the boundedness of $f$ from Step 1, we proceed to show the integrability of the integrand in \cref{eq:f duhamel} as per Step 1 of the proof of \cref{lem:regularity 1D}. We now deduce that $f \in C((0,T]\times\bar{\Omega})$ using the continuity of $\tilde{\Phi}$ for $t>0$ and the Dominated Convergence Theorem, exactly as we did in Step 2 of the proof of \cref{lem:regularity 1D}. 
\end{proof}

Finally, by following the strategy of proof of \cref{lem:reg at initial and final time} to the letter, we obtain the following result concerning continuity up to the initial time. 
\begin{lemma}\label{lem:regularity-2d-final}
Suppose that the periodic initial data $f_0$ for problem \cref{model_intro} belongs to $C(\bar{\Upsilon})$. Then, the unique weak solution $f$ of \cref{model_intro} belongs to $C([0,T]\times\bar{\Upsilon})$. 
\end{lemma}

We note, as per \cref{rem:higher reg maybe possible}, that it is likely possible to improve \cref{lem:regularity-2d-final}. However, the method for proving such a result is different from the current approach, which uses the Duhamel principle, and so we leave this for further investigation.

\section{Extension to very weak solutions for the 2D model}\label{sec:very-weak-sol}

The assumption on the initial data in $L^2_{per}(\Upsilon)$ rules out initial configurations with a fixed orientation, that is, $f_0(x,\theta) = h^0(x)\delta_{\theta = \theta^*}$ for some $\theta^* \in [0,2\pi)$ and $h^0 \in L^2_{per}(\Omega)$, where the latter is the usual Dirac delta in the angular variable only. Since such initial data belongs to $L^2(\Omega;H^s(0,2\pi))$ for $s<-1/2$, we turn our attention to initial data $f_0\in L^2_{per}(\Omega;(H^1_{per})'(0,2\pi))$. For the convenience of the reader, we clarify that, for any normed vector space $V$, 
\begin{equation}\label{eq:def periodic distribution actually}
        f \in L^2_{per}(\Omega;V) \iff \left\lbrace \begin{aligned}& f : \mathbb{R}^2 \to V, \\ 
        &\int_\Omega \Vert f(\x,\cdot) \Vert_{V}^2 \, \ud \x < +\infty, \text{ and} \\  &f(\x-\mathbf{w},\cdot)=f(\x,\cdot) \quad \text{a.e. } \x \in \Omega, \quad \forall \mathbf{w} \in (2\pi\mathbb{Z})^2, 
    \end{aligned}\right. 
\end{equation}
where the final equality on the right-hand side holds in $V$.

In view of the previous discussion, we define the following Hilbert spaces, to give meaning to a weaker formulation of the equation (\textit{cf.}~\cref{def:distributional solution}), 
\begin{equation}\label{eq:banach spaces for distributional solution}
    \begin{aligned}
        &X := L^2([0,T];H^1_{per}(\Upsilon)) \cap L^2([0,T];L^2_{per}(\Omega; H^2_{per}(0,2\pi))), \\ 
        &Y := H^1_{per}(\Upsilon) \cap L^2_{per}(\Omega;H^2_{per}(0,2\pi)). 
    \end{aligned}
\end{equation}
Note that $\Vert u \Vert_{X'}^2 = \int_0^T \Vert u(t) \Vert^2_{Y'} \, \ud t$, and $\Vert u(t) \Vert_{Y'} = \sup_{\overset{\varphi \in C^\infty_{per}(\bar{\Upsilon})}{\Vert \varphi \Vert_{Y} \leq 1}}|\langle u(t) , \varphi \rangle_\Upsilon|$. 

The initial density must now be defined in duality, \textit{i.e.}, 
\begin{equation}\label{eq:space-density-duality}
    \rho_0(\x):=\langle f_0(\x,\cdot), 1 \rangle_{(H^1_{per})'\times H^1_{per}} \qquad \mbox{a.e. } x\in\Omega,
\end{equation}
where $\langle \cdot , \cdot \rangle_{(H^1_{per})'\times H^1_{per}}$ is the duality product on $(0,2\pi)$. Note that for $f_0(\x,\theta) = h^0(\x)\delta_{\theta = \theta^*}$, the initial space density is $\rho_0(\x)=h^0(\x)$.

We now extend the concept of solution established in \cref{def:concept of solution} to admit a larger class of functions (\textit{i.e.}~smaller class of test functions). With slight abuse of terminology, we name these solutions \emph{very weak solutions}, even though only the regularity in the angular variable is affected. 

\begin{definition}\label{def:distributional solution}
A \emph{very weak solution} of Eq.~\cref{model_intro} is a curve $f$ belonging to $L^2([0,T];H^1_{per}(\Omega;(H^1_{per})'(0,2\pi))) \cap L^\infty([0,T];L^2_{per}(\Omega;(H^1_{per})'(0,2\pi))) \cap L^2((0,T)\times\Upsilon)$, with $\partial_t f \in X'$ such that, for every test function $\varphi \in Y$, there holds for a.e.~$t\in[0,T]$ 
    \begin{equation}\label{eq:weak_form_def_distributional}
    \begin{split}
    &\left\langle\partial_t f(t), \varphi\right\rangle_{\Upsilon}\!=\!\left\langle \Pe(1- \rho(t))f(t) \e(\theta),\nabla\varphi\right\rangle_{\Upsilon}\!-\!\left\langle D_e\nabla f(t),\nabla\varphi\right\rangle_{\Upsilon}+ \left\langle f(t), \partial_\theta^2 \varphi\right\rangle_{\Upsilon},\\
    &f(0,\x,\theta)=f_0(\x,\theta),\qquad \text{on } \Upsilon\times\{0\},
\end{split}
    \end{equation}
with periodic boundary conditions on $\Upsilon$, \mbox{where} $\rho(t,\mathbf{x})=\int_0^{2\pi}f(t,\x,\theta)\,\ud \theta$, and where the initial data is achieved in the sense $f(0) = f_0$ in $Y'$. 
\end{definition}
Since $Y \subset L^2_{per}(\Omega;H^2_{per}(0,2\pi))$, we have the inclusions $L^2_{per}(\Omega;(H^1_{per})'(0,2\pi)))\subset L^2_{per}(\Omega;(H_{per}^2)'(0,2\pi))\subset Y'$. It follows that $f \in L^2([0,T];L^2_{per}(\Omega;(H^1_{per})'(0,2\pi)))$ and $\partial_t f \in X'$ imply $f \in H^1([0,T];Y')$. Hence, by \cite[Sec.~5.9.2 Th.~2]{evans}, $f \in C([0,T];Y')$. Therefore, $f(0) \in Y'$ is well-defined in the previous definition.

We now give the proof of \cref{thm:very-weak-sec2}; the main theorem of this section. Since a very weak solution is an element of $L^2((0,T)\times\Upsilon)$, this result implies that the solutions never concentrate in angle to a Dirac mass, even when the initial data does. This illustrates an instantaneous smoothing property of the equation.

\begin{proof}[Proof of \cref{thm:very-weak-sec2}]
The idea of the proof is to construct a sequence of regularised initial data that approaches $f_0$ in the required norm, and to show that the corresponding regularised solutions converge to a very weak solution.

  \textit{Step 1 (regularised sequence of initial data and solutions): } Due to \cref{lem:approximation-initial-data} in \cref{sec:appendix-approximation}, 
there exists $(f_0^m)_{m\in\mathbb{N}}$ nonnegative elements of $C^\infty_{per}(\bar{\Upsilon})$ such that 
\begin{equation}\label{eq:regularised initial data}
    \Vert f_0 - f_0^m \Vert_{L^2(\Omega;(H^{1}_{per})'(0,2\pi))} \to 0 \qquad \text{as } m \to \infty, 
\end{equation}
with $\rho_0^m(\x) := \int_0^{2\pi} f_0^m(\x,\theta) \, \ud \theta \in [0,1]$ for a.e.~$\x \in \Upsilon$. Note that \cref{eq:regularised initial data} of course implies that there exists a positive constant $M$ independent of $m$ such that 
\begin{equation}\label{eq:your big M}
    \sup_m \Vert f_0^m \Vert_{L^2(\Omega;(H^{1}_{per})'(0,2\pi))} \leq  M <+\infty. 
\end{equation}

Since each $f_0^m$ is an admissible initial datum for \cref{thm:exist-uniq-phenom}, we deduce that there exists a unique $f_m \in L^\infty([0,T];L^2_{per}(\Upsilon))\cap L^2([0,T];H^1_{per}(\Upsilon))$ with $\partial_t f_m \in L^2([0,T];(H^1_{per})'(\Upsilon))$ solving \cref{model_intro} in the weak sense given by by \cref{def:concept of solution}. Moreover, we have the uniform bound $0 \leq \rho_m(t,\x) \leq 1$ for a.e.~$(t,x)\in(0,T)\times\Omega$. Note however that we do not have a uniform estimate on $\Vert f_m \Vert_{L^2((0,T);H^1(\Upsilon))}$ or even on $\Vert f_m \Vert_{L^2((0,T)\times\Upsilon)}$, since these latter estimates depend on $\Vert f_0^m \Vert_{L^2(\Upsilon)}$, which explodes in the limit as $m \to \infty$. We therefore require a different estimate on $(f_m)_{m\in\mathbb{N}}$.

  \textit{Step 2 (an auxiliary Dirichlet problem): } In contrast to the inverse Laplacian introduced in \cref{section:uniqueness}, we introduce $L_\theta := (-\partial^2_\theta)^{-1}:(H^1_{per})'(0,2\pi) \to H^1_{per}(0,2\pi)$, which is understood to be the unique solution of the Dirichlet problem: 
\begin{equation}\label{eq:our favorite dirichlet}
    \left\lbrace\begin{aligned}
        & (-\partial_\theta^2) L_\theta f = f \qquad \text{in } (0,2\pi), \\ 
        & L_\theta f(0)=L_\theta f(2\pi)=0. 
    \end{aligned}\right.
\end{equation}
Let us note that for any function $f\in (H^{1}_{per})'(0,2\pi)$, there holds
\begin{equation*}
        \|f\|_{(H^1_{per})'(0,2\pi)}=\sup_{{\overset{\varphi \in H^1_{per}(0,2\pi)}{\Vert \varphi \Vert_{H^1(0,2\pi)}\le1}}}\langle f,\varphi\rangle_{(H^1_{per})'\times H^1_{per}}\le\|\partial_\theta L_\theta f \|_{L^2_{per}(0,2\pi)}. 
\end{equation*}
Indeed, using the property $(-\partial_\theta^2)L_\theta f = f$, we have, where the duality product is between $H^1_{per}$ and its dual, 
\begin{equation*}
    \begin{split}
    |\langle f,\varphi\rangle|&=|\langle (-\partial_\theta^2)\circ(-\partial_\theta^2)^{-1} f,\varphi\rangle| =|\langle \partial_\theta L_\theta f,\partial_\theta\varphi\rangle_{L^2\times L^2}|
    \le\|\partial_\theta L_\theta f\|_{L^2}\|\partial_\theta \varphi\|_{L^2}.
    \end{split}
\end{equation*}
Note that the integration by parts in the above involve no additional boundary term due to the periodicity of $f$ and $L_\theta f$. By choosing the specific test function $\varphi=L_\theta f \in H^1_{per}(0,2\pi)$ in the previous estimate, we obtain the maximum value in the duality product defining the $(H^{1}_{per})'$ norm, hence 
\begin{equation}\label{eq:linking Hminus one of f with L2 of L theta}
\|f\|_{(H^{1}_{per})'(0,2\pi)}=\|\partial_\theta L_\theta f\|_{L^2(0,2\pi)}.
\end{equation}
Since $L_\theta f \in H^1_0(0,2\pi)$, the Poincar\'{e} inequality in one dimension yields 
\begin{equation}\label{eq:poincare in theta}
    \Vert L_\theta f \Vert_{L^2(0,2\pi)}^2 \leq C_P \Vert \partial_\theta L_\theta f \Vert^2_{L^2(0,2\pi)} = C_P \Vert f \Vert^2_{(H^1_{per})'(0,2\pi)}\leq C_P \Vert f \Vert^2_{L^2_{per}(0,2\pi)}, 
\end{equation}
where the constant $C_P$ depends only on the domain $(0,2\pi)$, and \cref{eq:linking Hminus one of f with L2 of L theta} was used to obtain the equality. In turn, given a function $f$ such that $\nabla f \in L^2_{per}(0,2\pi) \subset (H^1_{per})'(0,2\pi)$, there holds the estimate 
\begin{equation}\label{eq:to show its admissible ii}
    \Vert L_\theta \nabla f \Vert_{L^2(0,2\pi)}^2 \leq C_P \Vert \nabla f \Vert^2_{(H^1_{per})'(0,2\pi)} \leq C_P \Vert \nabla f \Vert^2_{L^2 (0,2\pi)}, 
\end{equation}
and, directly from \cref{eq:linking Hminus one of f with L2 of L theta}, 
\begin{equation}\label{eq:to justify finiteness of grad theta nabla L f}
    \Vert \partial_\theta L_\theta \nabla f \Vert_{L^2(0,2\pi)}^2 = \Vert \nabla f \Vert^2_{(H^1_{per})'(0,2\pi)} \leq \Vert \nabla f \Vert^2_{L^2 (0,2\pi)}. 
\end{equation}

  \textit{Step 3 (energy estimate on first derivative independent of $m$): } We now apply the operator $L_\theta$ to the sequence $(f_m)_{m\in\mathbb{N}}$. Note that, for a.e.~$t\in(0,T)$, 
\begin{equation*}
    \Vert L_\theta f_m(t) \Vert_{H^1(\Upsilon)}^2 = \Vert L_\theta f_m(t) \Vert_{L^2(\Upsilon)}^2 + \Vert \nabla L_\theta f_m(t) \Vert^2_{L^2(\Upsilon)} + \Vert \partial_\theta L_\theta f_m(t) \Vert^2_{L^2(\Upsilon)}. 
\end{equation*}
It then follows from \cref{eq:poincare in theta} and \cref{eq:to show its admissible ii} that 
\begin{equation*}
    \begin{aligned}
        \Vert L_\theta f_m(t) \Vert_{H^1(\Upsilon)}^2 \leq &\int_\Omega \Vert L_\theta f_m(t,\x,\cdot) \Vert_{L^2(0,2\pi)}^2 \, \ud \x + \int_\Omega \Vert \nabla L_\theta f_m(t,\x,\cdot) \Vert^2_{L^2(0,2\pi)} \, \ud \x \\ 
        &+ \int_\Omega \Vert \partial_\theta L_\theta f_m(t,\x,\cdot) \Vert^2_{L^2(0,2\pi)} \, \ud \x \\ 
        \leq & 2C_P \int_\Omega \big( \Vert f_m(t,\x,\cdot) \Vert^2_{L^2(0,2\pi)} + \Vert \nabla f_m(t,\x,\cdot) \Vert^2_{L^2(0,2\pi)} \big) \, \ud \x. 
    \end{aligned}
\end{equation*}
Thus, $\Vert L_\theta f_m(t) \Vert_{H^1(\Upsilon)}^2 \leq 2C_P \Vert f_m \Vert_{H^1(\Upsilon)}^2 < +\infty$. Therefore, $L_\theta f_m(t)$ is an admissible test function for the weak formulation of \cref{def:concept of solution}. This latter observation enables us to test the equation with $L_\theta f_m$ and to obtain an estimate independent of $m$. We get, for a.e.~$(t,\x) \in (0,T)\times\Omega$, 
\begin{equation*}
    \begin{split}
        \frac{1}{2}\frac{d}{dt}\|f_m(t,\x,\cdot)\|_{(H^{1}_{per})'(0,2\pi)}^2=&\frac{1}{2}\frac{d}{dt}\|\partial_\theta L_\theta f_m(t,\x,\cdot) \|_{L^2(0,2\pi)}^2 \\ 
        =& \langle \partial_t f_m(t,\x,\cdot), L_\theta f_m(t,\x,\cdot)\rangle_{(H^1)' \times H^1}, 
        \end{split}
        \end{equation*}
        where we used that $\partial_t f \in L^2([0,T];(H^1_{per})'(\Upsilon))$ and $L_\theta f_m \in L^2([0,T];H^1_{per}(\Upsilon))$. Hence, by integrating the previous line with respect to the space variable and using the weak formulation, also noting that $\langle\nabla f_m,\nabla L_\theta f_m\rangle_\Upsilon = \|\nabla \partial_\theta L_\theta f_m\|_{L^2(\Upsilon)}^2$ and $\langle \partial_\theta f_m,\partial_\theta L_\theta f_m\rangle_\Upsilon = \|f_m\|_{L^2(\Upsilon)}^2$, 
        \begin{equation}\label{eq:pre gronwall distributional}
        \begin{split}
        \frac{1}{2}\frac{d}{dt}\|f_m(t)  \|_{L^2(\Omega;(H^{1}_{per})'(0,2\pi))}^2 = \Pe\langle (1\!&-\!\rho_m)f_m \e(\theta), \!\nabla L_\theta f_m\rangle_\Upsilon \! \\ 
        &-\!\|\nabla \partial_\theta L_\theta f_m\|_{L^2(\Upsilon)}^2 -\!\|f_m\|_{L^2(\Upsilon)}^2, 
    \end{split}
\end{equation}
where we used the equality $(-\partial^2_\theta)L_\theta f_m = f_m$, integrated by parts in $\theta$ the last two duality products, and used the Dirichlet problem \cref{eq:our favorite dirichlet}. Next, observe that $f_m \e(\theta) = (-\partial_\theta^2)L_\theta f_m \e(\theta) = -\partial_\theta ( \partial_\theta L_\theta f_m \e(\theta) ) - \partial_\theta L_\theta f_m \e^\perp(\theta)$, where $\e^\perp(\theta) = (-\sin \theta, \cos \theta)$. Using the fact that $\rho_m$ is independent of the angle variable, it follows from an integration by parts that 
\begin{equation*}
    \begin{split}
        \langle (1-\rho_m)f_m \e(\theta) , \nabla L_\theta f_m \rangle_\Upsilon = \langle (1-\rho_m) &\partial_\theta L_\theta f_m \e(\theta) , \nabla \partial_\theta L_\theta f_m \rangle_\Upsilon\\ &- \langle (1-\rho_m) \partial_\theta L_\theta f_m \e^\perp(\theta) , \nabla L_\theta f_m \rangle_\Upsilon,
    \end{split}
\end{equation*}
whence the H\"{o}lder inequality and the bound for $\rho_m$ give $|\langle (1-\rho_m)f_m \e(\theta), \nabla L_\theta f_m \rangle_\Upsilon|         \leq \Vert \partial_\theta L_\theta f_m\Vert_{L^2(\Upsilon)} \Vert \nabla \partial_\theta L_\theta f_m \Vert_{L^2(\Upsilon)} + \Vert \partial_\theta L_\theta f_m \Vert_{L^2(\Upsilon)} \Vert \nabla L_\theta f_m \Vert_{L^2(\Upsilon)}$. Thus, by applying the Poincar\'{e} inequality \cref{eq:poincare in theta} to bound $\nabla L_\theta f_m$ by $\nabla \partial_\theta L_\theta f_m$ and Young's inequality, 
\begin{equation*}
    \begin{aligned}
        |\langle (1\!-\!\rho_m)f_m \e(\theta) ,\! \nabla L_\theta f_m \rangle_\Upsilon|
        \leq   \Pe(1\!+\!C_P) \Vert \partial_\theta L_\theta f_m\Vert_{L^2(\Upsilon)}^2 \!+\! \frac{1}{2 \Pe} \Vert \partial_\theta \nabla L_\theta f_m \Vert_{L^2(\Upsilon)}^2. 
    \end{aligned}
\end{equation*}
Combining the above with \cref{eq:pre gronwall distributional}, using \cref{eq:to justify finiteness of grad theta nabla L f} to rewrite $\|\nabla \partial_\theta L_\theta f_m\|_{L^2(\Upsilon)}^2$ and \cref{eq:linking Hminus one of f with L2 of L theta} to rewrite $\Vert \partial_\theta L_\theta f_m\Vert^2_{L^2(\Upsilon)}$, and integrating in time, we get, for a.e.~$t\in(0,T)$, 
\begin{equation*}
    \begin{aligned}
        \|f_m(t)\|_{L^2(\Omega;(H^{1}_{per})'(0,2\pi))}^2 + \int_0^t \|\nabla f_m(\tau)\|_{L^2(\Omega;(H^1_{per})'(0,2\pi))}^2 \, \ud \tau + 2 \int_0^t \|f_m(\tau)\|_{L^2(\Upsilon)}^2 \, \ud \tau& \\ 
        \quad \le \|f_0^m \|_{L^2(\Omega;(H^{1}_{per})'(0,2\pi))}^2 + \Pe^2(1+C_P) \int_0^t \|f_m(\tau)\|_{L^2(\Omega;(H^{1}_{per})'(0,2\pi))}^2 \, \ud \tau&. 
    \end{aligned}
\end{equation*}
Gr\"{o}nwall's Lemma yields $\|f_m(t)\|_{L^2(\Omega;(H^{1}_{per})'(0,2\pi))}^2 \le M^2 \exp(\Pe^2(1+C_P)T)$ a.e.~$t \in (0,T)$, and we recall from \cref{eq:your big M} that the right-hand side is independent of $m$. Hence, we obtain, for some positive constant $M_1$ independent of $m$, 
\begin{equation}\label{eq:the uniform bound for the distributional solution}
    \begin{aligned}
        \esssup_{t\in[0,T]}\!\|f_m(t)\|_{L^2(\Omega;(H^{1}_{per})'(0,2\pi))}^2\! &+\! \int_0^T \!\|\nabla f_m(t)\|_{L^2(\Omega;(H^1_{per})'(0,2\pi))}^2 \, \ud t \\ 
        &+ \!\int_0^T \!\|f_m(t)\|_{L^2(\Upsilon)}^2 \, \ud t \leq M_1. 
    \end{aligned}
\end{equation}
Thus, $\|f_m \|_{L^\infty([0,T];L^2(\Omega;(H^{1}_{per})'(0,2\pi)))}$,$\| f_m \|_{L^2([0,T];H^1(\Omega;(H^1_{per})'(0,2\pi)))}$, $\|f_m \|_{L^2((0,T)\times\Upsilon)}$ are bounded independently of $m$.

  \textit{Step 4 (corresponding estimate on the time derivatives): } We now return to \cref{model_intro} in order to obtain a similar uniform estimate on the sequence of time derivatives $(\partial_t f_m)_{m\in\mathbb{N}}$. Fix $\varphi$ to be smooth and periodic on $\bar{\Upsilon}$. Such a function is an admissible test function for the weak formulation of the equation, and we get 
\begin{equation*}
    \begin{aligned}
        |\langle \partial_t f_m(t) , \varphi \rangle_\Upsilon| \leq & |\Pe\langle (1-\rho_m) f_m \e(\theta) , \nabla \varphi \rangle_\Upsilon| + |\langle \nabla f_m , \nabla \varphi \rangle_\Upsilon| + |\langle \partial_\theta f_m , \partial_\theta \varphi \rangle_\Upsilon| \\ 
        \leq &  \Pe \Vert f_m \Vert_{L^2(\Upsilon)}\Vert \nabla \varphi \Vert_{L^2(\Upsilon)} + \Vert f_m \Vert_{L^2(\Upsilon)} \Vert \partial_\theta^2 \varphi \Vert_{L^2(\Upsilon)} \\ 
        &+ \Vert \nabla f_m \Vert_{L^2(\Omega;(H^1_{per})'(0,2\pi))} \Vert \nabla \varphi \Vert_{L^2(\Omega;H^1(0,2\pi))} \\ 
        \leq & \big((1+ \Pe) \Vert f_m(t) \Vert_{L^2(\Upsilon)} + \Vert \nabla f_m(t) \Vert_{L^2(\Omega;(H^1_{per})'(0,2\pi))} \big) \big( \Vert \varphi \Vert_{H^1(\Upsilon)} \\ 
        &+ \Vert \varphi \Vert_{L^2(\Omega;H^2(0,2\pi))} \big), 
    \end{aligned}
\end{equation*}
for a.e.~$t\in(0,T)$. We therefore obtain, using \cref{eq:banach spaces for distributional solution} and \cref{eq:banach spaces for distributional solution}, $\Vert \partial_t f_m(t) \Vert_{Y'} \leq ((1+ \Pe) \Vert f_m(t) \Vert_{L^2(\Upsilon)} + \Vert \nabla f_m(t) \Vert_{L^2(\Omega;(H^1_{per})'(0,2\pi))} )$, whence, using \cref{eq:the uniform bound for the distributional solution} to bound the right-hand side of the above, we get, with $M_2$ a positive constant independent of $m$, 
\begin{equation}\label{eq:time deriv distribution pre estimate}
    \Vert \partial_t f_m \Vert_{X'} \leq M_2. 
\end{equation}

  \textit{Step 5 (estimates on $\rho$ independent of $m$): } We infer uniform estimates on the sequence $(\rho_m)_{m\in\mathbb{N}}$ from those on $(f_m)_{m\in\mathbb{N}}$. Firstly, $\rho_m(t,\x) = \int_0^{2\pi} f_m(t,\x,\theta) \, \ud \theta \qquad$ for a.e.~$(t,\x)\in(0,T)\times\Omega$, from which it is easy to see, since the constant function $1/2\pi$ belongs to $H^1_{per}(0,2\pi)$ and has norm equal to $1$ therein,
\begin{equation}\label{eq:towards unif est for rho m i}
    |\rho_m(t,\x)| \leq 2\pi \sup_{\Vert \varphi \Vert_{H^1_{per}(0,2\pi)} \leq 1} \bigg|\int_0^{2\pi} f_m(t,\x,\theta) \varphi(\theta) \, \ud \theta \bigg| = 2\pi \Vert f_m(t,\x,\cdot) \Vert_{(H^1_{per})'(0,2\pi)}, 
\end{equation}
for a.e.~$(t,\x)\in(0,T)\times\Omega$. The same estimate on the derivative yields $|\nabla \rho_m(t,\x)| \leq 2\pi \Vert \nabla f_m(t,\x,\cdot) \Vert_{(H^1_{per})'(0,2\pi)}$ a.e.~in $(0,T)\times\Omega$. It follows directly from the comment under the uniform estimate \cref{eq:the uniform bound for the distributional solution} that there holds 
\begin{equation}\label{eq:unif estimates on rho m i}
    \Vert \rho_m \Vert_{L^\infty([0,T];L^2_{per}(\Omega))} + \Vert \rho_m \Vert_{L^2([0,T];H^1_{per}(\Omega))} \leq M_3, 
\end{equation}
for some $M_3$ independent of $m$. Finally, for the time derivatives, by returning to the weak formulation of \cref{modelrho_intro} and arguing as we did to obtain \cref{eq:time deriv distribution pre estimate}, we obtain 
\begin{equation*}
    |\langle \partial_t \rho_m(t),\varphi \rangle_\Omega| \leq  \Pe |\Omega|^{1/2} \Vert \nabla \varphi \Vert_{L^2(\Omega)} + \Vert \nabla \rho_m(t) \Vert_{L^2(\Omega)}\Vert \nabla \varphi \Vert_{L^2(\Omega)}, 
\end{equation*}
whence $\Vert \partial_t \rho_m(t) \Vert_{(H^1_{per})'(\Omega)} \leq  \Pe |\Omega|^{1/2} + \Vert \nabla \rho_m(t) \Vert_{L^2(\Omega)}$ for a.e.~$t \in (0,T)$. As a result, we deduce from \cref{eq:unif estimates on rho m i} the uniform estimate 
\begin{equation}\label{eq:unif estimates on rho m ii}
    \Vert \partial_t \rho_m \Vert_{L^2([0,T];(H^1_{per})'(\Omega))} \leq M_4, 
\end{equation}
for some positive constant $M_4$ also independent of $m$.

  \textit{Step 6 (compactness and passage to the limit in $m$): } Due to the uniform estimate $\sup_m \Vert f_m \Vert_{L^2((0,T)\times\Upsilon)} < +\infty$ deduced under \cref{eq:the uniform bound for the distributional solution}, we have from the Banach--Alaoglu Theorem that there exists $f \in L^2((0,T)\times\Upsilon)$ such that 
\begin{equation*}
    f_m \rightharpoonup f \qquad \text{weakly in } L^2((0,T)\times\Upsilon), 
\end{equation*}
up to a subsequence, which, with slight abuse of notation, we still label as $(f_m)_{m\in\mathbb{N}}$. In turn, there holds $\rho_m \rightharpoonup \rho$ weakly in $L^2((0,T)\times\Omega)$, where $\rho(t,\x) = \int_0^{2\pi} f(t,\x,\theta) \, \ud \theta$ for a.e.~$(t,\x)\in(0,T)\times\Omega$. Meanwhile, given the estimates \cref{eq:unif estimates on rho m i} and \cref{eq:unif estimates on rho m ii}, the Aubin--Lions Lemma implies the strong convergence 
\begin{equation*}
    \rho_m \to \rho \qquad \text{strongly in } L^2((0,T)\times\Omega). 
\end{equation*}
The term $(1-\rho_m)f_m$ is therefore the product of a strongly converging sequence and a weakly converging sequence in $L^2((0,T)\times\Omega)$. We are therefore able to pass to the limit weakly in $L^2((0,T)\times\Omega)$ in the nonlinear drift term. 

Additionally, taking further subsequences if necessary, the remaining uniform estimates from \cref{eq:the uniform bound for the distributional solution} and \cref{eq:time deriv distribution pre estimate}, as well as the Banach--Alaoglu Theorem, imply the convergences $\partial_t f_m \overset{*}{\rightharpoonup} \partial_t f$ in $X'$, $\nabla f_m \overset{*}{\rightharpoonup} \nabla f$  in $L^2([0,T];L^2(\Omega;(H^1_{per})'(0,2\pi)))$, which proves the convergence of the relevant duality products to the limit equation. For the last duality product, we use an integration by parts to write $\langle \partial_\theta f_m , \partial_\theta \varphi \rangle_\Upsilon = - \langle f_m , \partial^2_\theta \varphi \rangle_\Upsilon \to \langle f , \partial^2_\theta \varphi \rangle_\Upsilon$ as $m \to \infty$, where we used the weak convergence $f_m \rightharpoonup f$ in $L^2((0,T)\times\Upsilon)$. Since the test function $\varphi$ is chosen to belong to the space $X$, we obtain the limit equation as $m\to\infty$, as required. 

\textit{Step 7 (initial data): } Recalling also the discussion under \cref{def:distributional solution}, the estimates \cref{eq:the uniform bound for the distributional solution} and \cref{eq:time deriv distribution pre estimate} imply that $f_m,f \in H^1([0,T];Y')$. By \cite[Sec.~5.9.2 Th.~2]{evans}, $f_m,f \in C([0,T];Y')$. Recall from \cref{eq:regularised initial data} $\Vert f_0 - f_0^m \Vert_{L^2(\Omega;(H^{1}_{per})'(0,2\pi))} \to 0$, from which $\Vert f_0-f_0^m \Vert_{Y'} \to 0$; since $L^2(\Omega;(H^1_{per})'(0,2\pi)) \hookrightarrow Y'$ continuously. From \cref{eq:the uniform bound for the distributional solution}, $(f_m)_{m\in\mathbb{N}}$ is bounded in  ${L^2([0,T];L^2_{per}(\Upsilon))}$. Using the Rellich--Kondrachov Compactness Embedding and properties of the adjoint map, $L^2_{per}(\Upsilon)$ embeds compactly into $Y'$; since $Y\hookrightarrow H^1_{per}(\Upsilon)$ continuously, and $H^1_{per}(\Upsilon)$ is compactly embedded in $L^2_{per}(\Upsilon)$. Meanwhile, the sequence $(\partial_t f_m)_{m\in\mathbb{N}}$ is uniformly bounded in $X'=L^2([0,T];Y')$. Thus the Aubin--Lions Lemma implies $f_m \to f$ strongly in $X'$, and hence, up to taking a further subsequence, $f_m(t) \to f(t)$ strongly in $Y'$ a.e.~$t\in[0,T]$. Uniqueness of limits in $Y'$ at the Lebesgue point $t=0$ implies $f(0)=f_0$ in $Y'$. 
\end{proof}

\appendix

\section{The system of ODEs for the Galerkin approximation in space}\label{section:big ode}
Throughout this section, $m \in \mathbb{N}$ is fixed. For $k\in\{0,\dots,n\}$, we define  \begin{equation}\label{eq:approx-a-b second}
    \begin{aligned}
        a_k ^{n,m}(t,x,y):= &\sum_{q=0}^m\sum_{p=0}^m\alpha_{1,p,q}^{k,m}(t)\cos px \cos qy + \sum_{q=1}^m\sum_{p=0}^m\alpha_{2,p,q}^{k,m}(t)\cos px \sin qy \\ 
        &+ \sum_{q=0}^m\sum_{p=1}^m\alpha_{3,p,q}^{k,m}(t)\sin px \cos qy + \sum_{q=1}^m\sum_{p=1}^m\alpha_{4,p,q}^{k,m}(t)\sin px \sin qy, \\
        b_k ^{n,m}(t,x,y):= &\sum_{q=0}^m\sum_{p=0}^m\beta_{1,p,q}^{k,m}(t)\cos px \cos qy + \sum_{q=1}^m\sum_{p=0}^m\beta_{2,p,q}^{k,m}(t)\cos px \sin qy \\ 
        &+ \sum_{q=0}^m\sum_{p=1}^m\beta_{3,p,q}^{k,m}(t)\sin px \cos qy + \sum_{q=1}^m\sum_{p=1}^m\beta_{4,p,q}^{k,m}(t)\sin px \sin qy, 
    \end{aligned}
\end{equation}
where $\{\alpha_{j,p,q}^{k,m},\beta_{j,p,q}^{k,m}\}$ are chosen to be the solutions of a specific ODE system. 

Recall that the products of sines and cosines form an orthogonal basis of $L^2_{per}(\Omega)$, which motivates the approximations \cref{eq:approx-a-b second}. Our aim is to recover the coefficients $\{a_{k},b_{k}\}_{k=0}^n$ as the limits as $m \to \infty$ of the coefficients defined in \cref{eq:approx-a-b second}. To do so, we impose that $\{a_k ^{n,m},b_k ^{n,m}\}$ solve \cref{eq:parabolic system a-b} weakly (\textit{i.e.}~in duality with $H^1_{per}(\Omega)$) and with initial data $\{a_k ^{n,m}(0),b_k ^{n,m}(0)\}$ computed by testing $\{a_{k}^0,b_{k}^0\}$ with products of sines and cosines up to frequencies $m$ inclusive and adding these contributions, \textit{i.e.}, 
\begin{equation}\label{eq-eg a0km initial data for ode system}
    \begin{aligned}
        a^0_{k,m}(\x) = &\sum_{q=0}^m\sum_{p=0}^m C_{1,p,q} \cos px \cos qy +\sum_{q=1}^m\sum_{p=0}^m C_{2,p,q} \cos px \sin q y  \\ 
        &+ \sum_{q=0}^m\sum_{p=1}^m C_{3,p,q} \sin px \cos qy + \sum_{q=1}^m\sum_{p=1}^m C_{4,p,q} \sin px \sin qy, 
    \end{aligned}
\end{equation}
where $C_{1,p,q} = \int_\Omega \cos px \cos qy a^0_k(\x) \, \ud \x$, $C_{2,p,q} = \int_\Omega \cos px \sin qy a^0_k(\x) \, \ud \x$, $C_{3,p,q}=\int_\Omega \sin px \cos qy a^0_k(\x) \, \ud \x$, and $C_{4,p,q}=\int_\Omega \sin px \sin qy a^0_k(\x) \, \ud \x$ (up to multiples of $2\pi$). The expression for $b^0_{k,m}$ is similar. By construction, $(a^0_{k,m},b^0_{k,m}) \to (a^0_kb^0_k)$ strongly in $L^2(\Omega)$ as $m\to\infty$.

By substituting appropriate test functions into the aforementioned weak formulations for $a_k ^{n,m}$ and $b_k ^{n,m}$, we obtain a system of ODEs with $\alpha^{k,m}_{j,p,q},\beta^{k,m}_{j,p,q}$ as the unkowns.

Testing against $\varphi = 1$ in the equation for $a_{0}^{m}$ in \cref{eq:parabolic system a-b} we find  $4\pi^2\frac{d \alpha^{0,m}_{1,0,0}}{dt}(t) = 0$, whence $\alpha^{0,m}_{1,0,0}(t) = \alpha^{0,m}_{1,0,0}(0)$ for all $t \in [0,T]$. Moreover, testing against $\varphi = 1$ in the equations for $a_{1}^{m}$ and $b_{1}^{m}$ in \cref{eq:parabolic system a-b}, we obtain $4\pi^2\frac{d \alpha^{k,m}_{1,0,0}}{dt}(t) = -4\pi^2 k^2\alpha^{k,m}_{1,0,0}(t)$ , $4\pi^2 \frac{d \beta^{k,m}_{1,0,0}}{dt}(t) = - 4\pi^2 k^2\beta^{k,m}_{1,0,0}(t)$ for $k \in \{1,\dots,n\}$. All coefficients $\{\alpha^{k,m}_{1,0,0},\beta^{k,m}_{1,0,0}\}_{k=0}^n$ are therefore determined. 

Then, testing the equations for $a_k ^{n,m},b_k ^{n,m}$ in \cref{eq:parabolic system a-b} with $\varphi = \cos p x$ for $p \in \{1,\dots,m\}$, we get 
\begin{equation*}
    2\pi^2\frac{d  \alpha^{0,m}_{1,p,0}}{dt}(t) -\Pe\big\langle (1-(a_{0}^{m}(t))_+)_+a_{1}^{m}(t),(-p \sin p x)\big\rangle = -D_e p^2 2\pi^2  \alpha^{0,m}_{1,p,0}(t), 
\end{equation*}
\begin{equation*}
\begin{aligned}
    2\pi^2\frac{d \alpha^{1,m}_{1,p,0}}{dt}(t) -\frac{\Pe}{2}\big\langle (1-(a_{0}^{m}(t))_+)_+(2a_{0}^{m}(t)+&a_{2,m}(t)),(-p \sin p x)\big\rangle = \\ &-D_e p^2 2\pi^2 \alpha^{1,m}_{1,p,0}(t)- 2\pi^2 \alpha^{1,m}_{1,p,0}(t),\\
    2\pi^2\frac{d \beta^ {1,m}_{1,p,0}}{dt}(t) -\frac{\Pe}{2}\big\langle (1-(a_{0}^{m}(t))_+)_+b_{2,m}(t),(-&p \sin p x)\big\rangle = -D_e p^2 2\pi^2 \beta^{1,m}_{1,p,0}(t)\\ & - 2\pi^2 \beta^{1,m}_{1,p,0}(t),
\end{aligned}
\end{equation*}
while, for $1 < k < n$, 
\begin{equation*}
\begin{aligned}
    2\pi^2\frac{d \alpha^{k,m}_{1,p,0}}{dt}(t) -\frac{\Pe}{2}\big\langle (1-(a_{0}^{m}(t))_+)_+(a_{k-1,m}(t)+&a_{k+1,m}(t)),(-p \sin p x)\big\rangle =\\ -D_e p^2 2\pi^2 \alpha^{k,m}_{1,p,0}(t)
    &- 2\pi^2 k^2 \alpha^{k,m}_{1,p,0}(t),\\
    2\pi^2\frac{d \beta^{k,m}_{1,p,0}}{dt}(t) -\frac{\Pe}{2}\big\langle (1-(a_{0}^{m}(t))_+)_+(b_{k-1,m}(t)+&b_{k+1,m}(t)),(-p \sin p x)\big\rangle =\\ -D_e p^2 2\pi^2 \beta^{k,m}_{1,p,0}(t)
    &- 2\pi^2 k^2 \beta^{k,m}_{1,p,0}(t),
\end{aligned}
\end{equation*}
and 
\begin{equation*}
\begin{aligned}
    2\pi^2\frac{d \alpha^{n,m}_{1,p,0}}{dt}(t) -\frac{\Pe}{2}\big\langle (1-(a_{0}^{m}(t))_+)_+&a_{n-1,m}(t),(-p \sin p x)\big\rangle = -D_e p^2 2\pi^2 \alpha^{n,m}_{1,p,0}(t)\\
    &-2\pi^2 n^2 \alpha^{n,m}_{1,p,0}(t),\\
    2\pi^2\frac{d \beta^{n,m}_{1,p,0}}{dt}(t) -\frac{\Pe}{2}\big\langle (1-(a_{0}^{m}(t))_+)_+&b_{n-1,m}(t),(-p \sin p x)\big\rangle = -D_e p^2 2\pi^2 \beta^{n,m}_{1,p,0}(t)\\
    &- 2\pi^2 n^2 \beta^{n,m}_{1,p,0}(t). 
\end{aligned}
\end{equation*}
The ODEs for the coefficients $\{\alpha_{3,p,0}^{k,m},\beta_{3,p,0}^{k,m}\}_{k=0}^n$ are obtained similarly by testing the equations with $\varphi = \sin p x$ for $p \in \{1,\dots,m\}$ and have a similar structure to those for $\{\alpha_{1,p,0}^{k,m},\alpha_{1,p,0}^{k,m}\}_{k=0}^n$. The same is true for the ODEs for the coefficients $\{\alpha_{1,0,q}^{k,m},\beta_{1,0,q}^{k,m}\}_{k=0}^n$ (obtained by testing \cref{eq:parabolic system a-b} with $\varphi = \cos q y$ for $q \in \{1,\dots,m\}$) and those for the coefficients $\{\alpha_{2,0,q}^{k,m},\beta_{2,0,q}^{k,m}\}_{k=0}^n$ (obtained by testing \cref{eq:parabolic system a-b} $\varphi = \sin q y$ for $q \in \{1,\dots,m\}$).

Similarly, the ODEs for the coefficients $\{\alpha_{1,p,q}^{k,m},\beta_{1,p,q}^{k,m}\}_{k=0}^n$ are obtained by testing \cref{eq:parabolic system a-b} with $\varphi = \cos p x \cos q y$ for $p,q \in \{1,\dots,m\}$. In detail, we obtain 
\begin{align*}
    \frac{d  \alpha^{0,m}_{1,p,q}}{dt} \!=\!\frac{\Pe}{\pi^2}\big\langle (1\!-\!(a_{0}^{m}(t))_+)_+&(a_{1}^{m}(t),b_{1}^{m}(t)),\!(\!-p \sin p x \cos q y,\!-q \sin qy \cos px)\big\rangle_{\Omega}\\ 
    &-D_e (p^2+q^2)  \alpha^{0,m}_{1,p,q}(t), 
\end{align*}
\begin{equation*}
    \begin{aligned}
    &\pi^2\frac{d \alpha^{1,m}_{1,p,q}}{dt}(t) +D_e (p^2+q^2) \pi^2 \alpha^{1,m}_{1,p,q}(t)  + \pi^2 \alpha^{1,m}_{1,p,q}(t)=\\
    &\frac{\Pe}{2}\big\langle (1\!-\!(a_{0}^{m}(t))_+)_+(2a_{0}^{m}(t)\!+\!a_{2,m}(t),b_{2,m}(t)),(\!-p \sin p x \cos q y,\!-q \sin qy \cos p x)\big\rangle_{\Omega}, \\ 
    &\pi^2\frac{d \beta^{1,m}_{1,p,q}}{dt}(t) 
    +D_e (p^2+q^2) \pi^2 \beta^{1,m}_{1,p,q}(t)  + \pi^2 \beta^{1,m}_{1,p,q}(t)=\\
    &\frac{\Pe}{2}\big\langle (1\!-\!(a_{0}^{m}(t))_+)_+(b_{2,m}(t),2a_{0}^{m}(t)\!-\!a_{2,m}(t)),(\!-p \sin p x \cos q y,\!-q \sin qy \cos p x)\big\rangle_{\Omega}, 
\end{aligned}
\end{equation*}
while, for $1 < k < n$, 
\begin{equation*}
    \begin{aligned}
    &\frac{2\pi^2}{\Pe}\frac{d \alpha^{k,m}_{1,p,q}}{dt}(t) \!+\! \frac{2 D_e}{\Pe} (p^2\!+\!q^2) \pi^2 \alpha^{k,m}_{1,p,q}(t)  + \frac{2 \pi^2}{\Pe} k^2 \alpha^{k,m}_{1,p,q}(t)=\\
    &\big\langle \!(\!1\!-\!(a_{0}^{m})_+)_+(a_{k-1,m}\!+\!a_{k+1,m},b_{k+1,m}\!-\!b_{k-1,m}),\!(\!-p \sin p x \cos q y,\!-q \sin qy \cos p x)\big\rangle_{\Omega},\\ 
    &\frac{2\pi^2}{\Pe}\frac{d \beta^{k,m}_{1,p,q}}{dt}(t) 
    \!+\! \frac{2 D_e}{\Pe} (p^2\!+\!q^2) \pi^2 \beta^{k,m}_{1,p,q}(t)  + \frac{2 \pi^2}{\Pe} k^2 \beta^{k,m}_{1,p,q}(t)=\\
    &\big\langle \!(\! 1\!-\!(a_{0}^{m})_+)_+(b_{k-1,m}\!+\!b_{k+1,m},a_{k-1,m}\!-\!a_{k+1,m}),\!(\!-p \sin p x \cos q y,\!-q\sin qy \cos p x)\big\rangle_{\Omega}, 
\end{aligned}
\end{equation*}
and 
\begin{equation*}
    \begin{aligned}
    &\pi^2\frac{d \alpha^{n,m}_{1,p,q}}{dt}(t)
    +D_e (p^2+q^2) \pi^2 \alpha^{n,m}_{1,p,q}(t)  + \pi^2 n^2 \alpha^{n,m}_{1,p,q}(t)=\\
    &\frac{\Pe}{2}\big\langle (1-(a_{0}^{m}(t))_+)_+(a_{n-1,m}(t),-b_{n-1,m}(t)),(-p \sin p x \cos q y,-q \sin qy \cos p x)\big\rangle_{\Omega}, \\ 
    &\pi^2\frac{d \beta^{n,m}_{1,p,q}}{dt}(t)
    +D_e (p^2+q^2) \pi^2 \beta^{n,m}_{1,p,q}(t)  + \pi^2 n^2 \beta^{n,m}_{1,p,q}(t)=\\
    &\frac{\Pe}{2}\big\langle (1-(a_{0}^{m}(t))_+)_+(b_{n-1,m}(t),a_{n-1,m}(t)),(-p \sin p x \cos q y,-q \sin q y \cos p x)\big\rangle_{\Omega}. 
\end{aligned}
\end{equation*}
The ODEs for the coefficients $\{\alpha_{3,p,q}^{k,m},\beta_{3,p,q}^{k,m}\}_{k=0}^n$ are obtained similarly by testing the equations with $\varphi = \sin p x \cos q y$ for $p,q \in \{1,\dots,m\}$, and have a similar structure to those of $\{\alpha^{n,m}_{1,p,q},\beta^{n,m}_{1,p,q}\}_{k=0}^n$. The same is true for the ODEs for the coefficients $\{\alpha_{2,p,q}^{k,m},\beta_{2,p,q}^{k,m}\}_{k=0}^n$ (obtained by testing the equations with $\varphi = \cos p x \sin q y$ for $p,q\in\{1,\dots,m\}$) and those for the coefficients $\{\alpha_{4,p,q}^{k,m},\beta_{4,p,q}^{k,m}\}_{k=0}^n$ (obtained by testing the equations with $\varphi = \sin p x \sin q y$ for $p,q \in \{1,\dots,m\}$). 

As was already mentioned in Step 1 of \cref{subsec:existence of parabolic system further galerkin}, the previous system may be rewritten in the form of \cref{eq:ODE system Lambda}, \textit{i.e.}~$   \Lambda_m'(t) = F(t,\Lambda_m(t)) $ for $t \in [0,T]$, with initial condition $\Lambda_m(0) = \Lambda_{m,0}$, where $\Lambda_m = \{\alpha^{k,m}_{j,p,q}\}_{k,j,p,q}$ is the tensor of coefficients, with initial value $\Lambda_{m,0}$, and $F$ is locally Lipschitz since the function $x \mapsto x_+^2$ is locally Lipschitz. The initial data is given by the explicit formulas  \begin{equation}\label{eq:initial data alpha beta i}
    \begin{aligned}
    4\pi^2 \alpha^{0,m}_{1,0,0}(0) = \int_\Omega a^0_{0,m} \, \ud x \, \ud y, \qquad & 4\pi^2 \alpha_{1,0,0}^{k,m}(0) = \int_\Omega a^0_{k,m}(x,y) \, \ud x \ud y , 
    \end{aligned}
\end{equation}
where we recall the expression for $a^0_{k,m}$ in \cref{eq-eg a0km initial data for ode system}, and, for $p,q \in \{1,\dots,m\}$, 
\begin{equation}\label{eq:initial data alpha beta ii}
    \begin{aligned} 
    &\alpha_{1,p,0}^{k,m}(0) \! = \! \int_\Omega \frac{a^0_{k,m}(x,y)}{2\pi^2} \cos p x \, \ud x \, \ud y, &&\alpha_{3,p,0}^{k,m}(0) \!=\! \int_\Omega \frac{a^0_{k,m}(x,y)}{2\pi^2} \sin p x \, \ud x \, \ud y, \\ 
    &\alpha_{1,0,q}^{k,m}(0)\! =\! \int_\Omega \frac{a^0_{k,m}(x,y)}{2\pi^2} \cos q y \, \ud x \, \ud y, &&\alpha_{2,0,q}^{k,m}(0) \!=\! \int_\Omega \frac{a^0_{k,m}(x,y)}{2\pi^2} \sin q y \, \ud x \, \ud y, \\ 
    &\alpha_{1,p,q}^{k,m}(0) \!=\! \int_\Omega \! \frac{a^0_{k,m}(x,y)}{\pi^2} \! \cos p x \! \cos q y \ud x \ud y,\! &&\alpha_{3,p,q}^{k,m}(0) \!=\! \int_\Omega \! \frac{a^0_{k,m}(x,y)}{\pi^2} \! \sin p x \! \cos q y \ud x \ud y , \\ 
    &\alpha_{2,p,q}^{k,m}(0) \!=\! \int_\Omega \! \frac{a^0_{k,m}(x,y)}{\pi^2} \! \cos p x \! \sin q y \ud x \ud y,\! &&\alpha_{4,p,q}^{k,m}(0) \!=\! \int_\Omega \! \frac{a^0_{k,m}(x,y)}{\pi^2} \! \sin p x \! \sin q y  \ud x \ud y, 
    \end{aligned}
\end{equation}
and the formulas for the coefficients $\{\beta_{i,p,q}^{k,m}\}_{i,p,q,k,m}$ are identical, with the exception that $b^{0}_{k,m}$ replaces $a^0_{k,m}$ on the right-hand sides of the previous equations. Note that the above are well-defined real numbers due to the integrability of the initial data $a^0_{k,m},b^0_{k,m}$ which is inherited from the integrability of $f_0 \in L^2_{per}(\Upsilon)$.

\section{Properties of the periodic heat kernel}\label{sec:complementary-sec-regularity}
Below we provide proofs of \cref{lem:properties of periodic kernel} and \cref{lem:generate periodic solutions}, respectively.
\begin{proof}[Proof of \cref{lem:properties of periodic kernel}]
\textit{Step 1 (smoothness away from initial time): } Begin by noting that, by a standard argument using the Weierstra{\ss} M-test, the series representation \cref{eq:periodic heat kernel def} for $\Phi$ is well-defined, and moreover can repeatedly be differentiated term-by-term for positive times. As such, $\Phi \in C^\infty((0,\infty)\times \mathbb{R})$, and $\partial_t \Phi(t,x) = -\frac{1}{\pi}\sum_{n=1}^\infty n^2 e^{-n^2 t} \cos nx = \partial_{xx} \Phi(t,x)$ for all $(t,x) \in (0,\infty)\times \mathbb{R}$. 

\textit{Step 2 (nonnegativity and $L^1$-norm): } Recall (\textit{cf.}~\cite[Section 21.3]{watson}) the infinite product representation for the Jacobi theta-3 function. For every $(t,x) \in (0,\infty)\times\mathbb{R}$, 
\begin{equation*}
    \Phi(t,x) = \frac{1}{2\pi}\theta_3\left( \frac{x}{2}, e^{-t} \right) = \frac{1}{2\pi} \prod_{n=1}^\infty (1-e^{-2nt})(1+2 e^{-t(2n-1)}\cos x + e^{-t(4n-2)}). 
\end{equation*}
By the Young inequality, $1 + e^{-t(4n-2)} \geq 2e^{-t(2n-1)}$, and hence every term in the above product is nonnegative. As such, $\Phi$ itself is nonnegative in $(0,\infty)\times\mathbb{R}$. Meanwhile, integrating the series term by term, we obtain $\int_{0}^{2\pi}\Phi(t,x) \, \ud x = 1$ for every $t\in(0,\infty)$. It therefore follows that $\Vert \Phi(t,\cdot) \Vert_{L^1([0,2\pi])} = \int_0^{2\pi} \Phi(t,x) \, \ud x = 1$ for all $t \in (0,\infty)$. Also, the Plancherel Theorem yields $\Vert \Phi(t,\cdot) \Vert^2_{L^2([0,2\pi])} = 1 + \sum_{n=1}^\infty e^{-2n^2 t}$. Integrating in time (using also $\sum n^{-2}<+\infty$) yields \cref{eq:L2L2 for kernel}. 

\textit{Step 3 (constant $L^2$-norm): } Furthermore, given any $(t,x) \in (0,T]\times[0,2\pi]$, successive changes of variables yield 
\begin{equation*}
    \begin{aligned}
        \int_0^t \int_0^{2\pi} \Phi(t-s,x-y)^2 \, \ud y \, \ud s &= \int_0^t \int_0^{2\pi} \Phi(\tau,x-y)^2 \, \ud y \, \ud \tau \\ 
        &= \!\int_0^t \int_{0}^{x}\! \Phi(\tau,w)^2 \, \ud w \, \ud \tau \! + \! \int_0^t \int_{x-2\pi}^{0}\! \Phi(\tau,w\!+\!2\pi)^2 \, \ud w \, \ud \tau \\ 
        &= \int_0^t \int_{0}^{2\pi} \Phi(\tau,w)^2 \, \ud w \, \ud \tau, 
    \end{aligned}
\end{equation*}
as required, where we also used the periodicity of $\Phi$ to obtain the second equality. 

\textit{Step 4 (convergence to initial data): } It remains to verify the initial data. Fix any $2\pi$-periodic $C^2$ function $\psi$. Then, the elementary theory of Fourier series shows that $\psi$ admits the Fourier representation $\psi(x) = \frac{1}{2\pi} a_0 + \frac{1}{\pi}\sum_{n=1}^\infty a_n \cos n x + \frac{1}{\pi}\sum_{n=1}^\infty b_n \sin n x$ a.e.~$x \in \mathbb{R}$, where, for $n \in \mathbb{N}$, 
\begin{equation}\label{eq:exp psi fourier for reg}
    \begin{aligned}
        a_0 = \int_0^{2\pi} \psi(x) \, \ud x, \quad & a_n = \int_0^{2\pi} \psi(x) \cos n x \, \ud x, \quad b_n = \int_0^{2\pi} \psi(x) \sin n x \, \ud x. 
    \end{aligned}
\end{equation}
Note that, due to the $C^2$-regularity of $\psi$ and its periodicity, two consecutive integration by parts yields $a_n = -\frac{1}{n^2}\int_0^{2\pi} \psi''(x) \cos n x \, \ud x$ and $b_n = -\frac{1}{n^2}\int_0^{2\pi} \psi''(x) \sin n x \, \ud x$ so that there exists a positive constant $C_\psi$ depending on $\psi$ alone such that 
\begin{equation}\label{eq:decay fourier coeffs an bn}
|a_n|+|b_n| \leq C_\psi n^{-2} \qquad \forall n\in\mathbb{N}. 
\end{equation}
Hence, $| \psi(0) - \int_{0}^{2\pi} \Phi(t,x) \psi(x) \, \ud x | = \frac{1}{\pi}| \sum_{n=1}^\infty (1-e^{-n^2 t}) a_n | \leq \frac{C_\psi}{\pi}  \sum_{n=1}^\infty n^{-2}(1-e^{-n^2 t})$, and, by the Dominated Convergence Theorem applied to the atomic measure (using that $\sum n^{-2} < +\infty$), the right-hand side vanishes in the limit as $t \to 0^+$. 
\end{proof}

\begin{proof}[Proof of \cref{lem:generate periodic solutions}]
Recall from \cref{lem:properties of periodic kernel} that $\Phi \in C^\infty((0,\infty)\times\mathbb{R})$. Since $\psi \in L^1([0,2\pi])$, an application of the Dominated Convergence Theorem shows that $\Psi \in C^\infty((0,\infty)\times\mathbb{R})$, and that for any $m,n\in\mathbb{N}\cup\{0\}$, there holds $\partial_t^{(m)}\partial_x^{(n)}\Psi(t,x) = \int_0^{2\pi} \partial_t^{(m)}\partial_x^{(n)} \Phi(t,x-y) \psi(y) \, \ud y$ for all $(t,x) \in (0,\infty) \times\mathbb{R}$. In turn, we directly obtain that $\partial_t \Psi - \partial_{xx}\Psi = 0$ pointwise in $(0,\infty)\times\mathbb{R}$. 

It remains to verify the initial condition. Suppose for the time being that $\psi$ is smooth enough, say $\psi \in C^2_{per}([0,2\pi])$. From the Fourier series decomposition \cref{eq:exp psi fourier for reg} for $\psi$ we compute, using the Fubini--Tonelli Theorem (or directly using the convolution result for Fourier series), the Fourier coefficients of $\Psi$: 
\begin{equation}\label{eq:exp Psi fourier for reg}
    \begin{aligned}
        &\tilde{a}_0(t) = a_0, \quad \tilde{a}_n(t) = e^{-n^2 t}a_n, \quad \tilde{b}_n(t) = e^{-n^2 t}b_n \qquad \text{for } n \in \mathbb{N}. 
    \end{aligned}
\end{equation}
Using the Plancherel Theorem for Fourier series of square integrable functions we get (up to a multiplicative constant) 
\begin{equation}\label{eq:convergence-L^2}
    \Vert \Psi(t,\cdot) - \psi \Vert^2_{L^2([0,2\pi])} = \sum_{n=1}^\infty (1-e^{-n^2 t})^2 (a_n^2 + b_n^2) \to 0 \qquad \text{as } t \to 0^+, 
\end{equation}
where we used the decay estimate \cref{eq:decay fourier coeffs an bn} and the Dominated Convergence Theorem applied to the atomic measure.

When $\psi$ is only $L^1_{per}([0,2\pi])$, we consider $\{\psi_m\}_{m\in\mathbb{N}}$ a sequence in $C^2_{per}([0,2\pi])$ approximating $\psi$ in $L^1([0,2\pi])$. Then, given any $\varepsilon>0$, there exists $N(\varepsilon) \in \mathbb{N}$ such that $\Vert \psi - \psi_m \Vert_{L^1([0,2\pi])} < \varepsilon/3$ $\forall m \geq N(\varepsilon)$. We then have 
\begin{equation*}
    \begin{aligned}
        \Vert \Psi(t,\cdot) - \psi \Vert_{L^1([0,2\pi])} \leq \!\Vert \Psi(t,\!\cdot) - \Psi_m(t,\cdot) \Vert_{L^1([0,2\pi])} &+ \Vert \Psi_m(t,\cdot) - \psi_m \Vert_{L^1([0,2\pi])} \\ 
        &+ \Vert \psi_m - \psi \Vert_{L^1([0,2\pi])}, 
    \end{aligned}
\end{equation*}
where $\Psi_m = \int_0^{2\pi} \Phi(t,x-y) \psi_m(y) \, \ud y$. Note that the second term in the previous expression vanishes in the limit as $t \to 0^+$ by the argument in the previous paragraph. Next, using the Fubini--Tonelli Theorem to exchange the order of the integrals, $\Vert \Psi(t,\cdot) - \Psi_m(t,\cdot) \Vert_{L^1([0,2\pi])} \leq \Vert \Phi(t,\cdot) \Vert_{L^1([0,2\pi])} \Vert \psi - \psi_m \Vert_{L^1([0,2\pi])} = \Vert \psi - \psi_m \Vert_{L^1([0,2\pi])},$ where we used from \cref{lem:properties of periodic kernel} that $ \Vert \Phi(t,\cdot) \Vert_{L^1([0,2\pi])} = 1$ for every $t \in (0,\infty)$. It follows that $\Vert \Psi(t,\cdot) - \psi \Vert_{L^1([0,2\pi])} < 2\varepsilon/3 + \Vert \Psi_m - \psi_m \Vert_{L^1([0,2\pi])}$ for all $m \geq N(\varepsilon)$. Thus, having fixed $m$ large enough, by now letting $t \to 0^+$ and using the result from the previous paragraph, we get $\Vert \Psi(t,\cdot) - \psi \Vert_{L^1([0,2\pi])}$ vanishes in the limit as $t \to 0^+$. 
\end{proof}

\section{Approximation Lemma}\label{sec:appendix-approximation}

\begin{lemma}\label{lem:approximation-initial-data}
Let $f\in L^2_{per}(\Omega;(H^1_{per})'(0,2\pi))$ be nonnegative and such that $\rho(\x):=\langle f(\x,\cdot), 1 \rangle_{(H^1_{per})'\times H^1_{per}}\in [0,1]$ for a.e.~$\x\in\Omega$. There exists a sequence $(f_\varepsilon)_{\varepsilon>0}$ of nonnegative elements of $C^\infty_{per}(\bar{\Upsilon})$ such that 
\begin{equation}\label{eq:regularised initial data-app}
    \Vert f - f_\varepsilon \Vert_{L^2_{per}(\Omega;(H^{1}_{per})'(0,2\pi))} \to 0 \qquad \text{as } \varepsilon \to 0, 
\end{equation}
with $\rho_\varepsilon(\x) = \int_0^{2\pi} f_\varepsilon(\x,\theta) \, \ud \theta \in [0,1]$ for a.e.~$\x \in \Omega$. 
\end{lemma}

\begin{proof}

  \textit{Step 1 (defining the approximation): } Let $\beta\in C^\infty_c(\R^2)$ and $\gamma\in C^\infty_c(\R)$ be nonnegative smooth bump functions, chosen such that $\beta$ is compactly supported inside $\Omega$ and $\gamma$ is compactly supported inside the interval $[0,2\pi)$, and without loss of generality 
\begin{equation}\label{eq:product mollifier has mass one}
    \bigg(\int_\Omega \beta(\y)\, \ud\y \bigg) \bigg(\int_{0}^{2\pi} \gamma(\theta)  \, \ud \theta\bigg) = 1.
\end{equation}
Moreover, we impose that $\beta$ is radially symmetric about the point $(\pi,\pi)$ and that $\gamma$ is symmetric about the point $\pi$. Consider the associated sequences $(\beta_\varepsilon)_{\varepsilon>0}$ and $(\gamma_\varepsilon)_{\varepsilon>0}$ \[
\beta_\varepsilon(\x) := \varepsilon^{-2}\beta(\x/\varepsilon), \qquad  \gamma_\varepsilon(\theta):=\varepsilon^{-\alpha}\gamma(\theta/\varepsilon^\alpha),
\] 
for $\alpha > 2$. These functions are smooth and nonnegative, and $\beta_\varepsilon$ is compactly supported in the square $\varepsilon\Omega$ while $\gamma_\varepsilon$ is compactly supported in $[0,2\pi\varepsilon^\alpha]$. We define $\tilde{\gamma}_\varepsilon$ to be the $2\pi$-periodic extension of $\gamma_\varepsilon$ from $[0,2\pi\varepsilon^\alpha]$ to all of $\mathbb{R}$, so that $\tilde{\gamma}_\varepsilon \in H^1_{per}(0,2\pi)$. Similarly, we define $\tilde{\beta}_\varepsilon$ to be the $\Omega$-periodic extension of $\beta_\varepsilon$ to all of $\mathbb{R}^2$. 

In turn, we define the sequence of approximations $(f_\varepsilon)_{\varepsilon>0}$ by the explicit formula 
\begin{equation*}
    f_\varepsilon(\x,\theta) := \int_\Omega \tilde{\beta}_\varepsilon(\x-\y) \langle f(\y,\cdot) , \tilde{\gamma}_\varepsilon(\theta-\cdot) \rangle\, \ud \y \qquad \text{a.e.~}(\x,\theta) \in \Upsilon, 
\end{equation*}
where, throughout the rest of this proof, the duality product is understood to be in the angle variable only, and denotes the duality product of $H^1_{per}(0,2\pi)$. 

Observe firstly that, due to the assumption that $f$ is nonnegative, we immediately have $f_\varepsilon$ nonnegative a.e.~in $\Upsilon$. Moreover, one may readily check (for instance by means of difference quotients) that $f_\varepsilon$ is smooth. Similarly, it is straightforward to verify that $f_\varepsilon$ is itself $\Upsilon$-periodic. Thus, $f_\varepsilon \in C^\infty_{per}(\bar{\Upsilon})$ for every $\varepsilon>0$.

  \textit{Step 2 (boundedness the approximate density): } Note that, in view of the aforementioned nonnegativity of $f_\varepsilon$, there holds the inequality $   0 \leq \int_0^{2\pi} f_\varepsilon(\x,\theta) \, \ud \theta 
 = \rho_\varepsilon(\x)$. Additionally, we have $\rho_\varepsilon(\x) = \int_0^{2\pi} (\int_\Omega \tilde{\beta}_\varepsilon(\x-\y) \langle f(\y,\cdot), \tilde{\gamma}_\varepsilon(\theta - \cdot) \rangle \, \ud \y ) \, \ud \theta$, and a standard argument involving Riemann sums and the Arzel\`{a}--Ascoli Theorem (using also the Fubini--Tonelli Theorem and the Dominated Convergence Theorem) leads to being able to commute the duality bracket with the integral in the angle variable, \textit{i.e.}, 
  \begin{equation}\label{eq:rho eps formula after duality bracket commute}
     \rho_\varepsilon(\x) =  \int_\Omega \tilde{\beta}_\varepsilon(\x-\y) \big\langle f(\y,\cdot), \int_0^{2\pi}\tilde{\gamma}_\varepsilon(\theta - \cdot) \, \ud \theta \big\rangle \, \ud \y. 
 \end{equation}
 Observe now that the integral inside the duality bracket is a constant. Indeed, since $\tilde{\gamma}_\varepsilon$ is smooth, and therefore bounded over any compact interval, the Dominated Convergence Theorem yields $\frac{d}{ds}\int_0^{2\pi} \tilde{\gamma}_\varepsilon(\theta-s) \, \ud \theta = -\int_0^{2\pi} \tilde{\gamma}_\varepsilon'(\theta-s) \, \ud \theta = \tilde{\gamma}_\varepsilon(-s) - \tilde{\gamma}_\varepsilon(2\pi-s) = 0$, where the final equality follows from the periodicity of $\tilde{\gamma}_\varepsilon$; recall that it is the periodic extension to all of $\mathbb{R}$ of the original mollifier $\gamma_\varepsilon$. It follows that 
 \begin{equation}\label{eq:computing integral of gamma tilde}
     \begin{aligned}
        \int_0^{2\pi} \tilde{\gamma}_\varepsilon(\theta-s) \, \ud \theta = \int_0^{2\pi} \tilde{\gamma}_\varepsilon(\theta) \, \ud \theta = \int_0^{2\pi} \gamma_\varepsilon(\theta) \, \ud \theta = \int_0^{2\pi} \gamma(\theta) \, \ud \theta \qquad \forall s \in \mathbb{R}, 
     \end{aligned}
 \end{equation}
where we recall in passing that $\gamma_\varepsilon$ is supported in $[0,2\pi\varepsilon^\alpha]$, while $\gamma$ is supported in $[0,2\pi]$. Hence, we get $\rho_\varepsilon(\x) = (  \int_0^{2\pi} \gamma(\theta) \, \ud \theta )(\int_\Omega \tilde{\beta}_\varepsilon(\x-\y) \langle f(\y,\cdot),1 \rangle \, \ud \y)$. 

Also, applying the Dominated Convergence Theorem and the Fubini--Tonelli Theorem gives $\partial_{x_1} \int_\Omega \tilde{\beta}_\varepsilon(\x-\y) \, \ud \y = \int_0^{2\pi} \int_0^{2\pi} \partial_{1} \tilde{\beta}_\varepsilon (x_1-y_1,x_2-y_2) \, \ud y_1 \, \ud y_2 = -\int_0^{2\pi}  (\tilde{\beta}_\varepsilon (x_1-2\pi,x_2-y_2) - \tilde{\beta}_\varepsilon (x_1,x_2-y_2)) \, \ud y_2 = 0$, where the final equality follows from the periodicity of $\tilde{\beta}_\varepsilon$; recall that this latter function is the $\Omega$-periodic extension of $\beta_\varepsilon$. An identical argument shows that $\partial_{x_2}\int_\Omega \tilde{\beta}_\varepsilon(\x-\y) \, \ud \y = 0$. Arguing as per \cref{eq:computing integral of gamma tilde} and making use of the symmetry property of $\beta$ about the point $(\pi,\pi)$, we deduce 
\begin{equation}\label{eq:equality translation only for periodic}
    \int_\Omega \tilde{\beta}_\varepsilon(\x-\y) \, \ud \y = \int_\Omega \beta(\y) \, \ud \y \qquad \forall \x \in \mathbb{R}^2. 
\end{equation}
We emphasise that the above property is not true for the function $\beta_\varepsilon$; one must take its periodic extension in order to obtain this translation invariance, since the integral ranges only over the square $\Omega$, and not over all of $\mathbb{R}^2$. 

Returning to \cref{eq:rho eps formula after duality bracket commute} and using that $\langle f(\y,\cdot),1 \rangle = \rho(\y)$ which takes values in the interval $[0,1]$, it follows that $0 \leq \rho_\varepsilon(\x) \leq (  \int_0^{2\pi} \gamma(\theta) \, \ud \theta )(\int_\Omega \tilde{\beta}_\varepsilon(\x-\y)  \, \ud \y ) = (\int_{0}^{2\pi} \gamma(\theta)  \, \ud \theta)(\int_\Omega \beta(\y)\, \ud\y) = 1$, as required, where we used the nonnegativity of $\tilde{\beta}_\varepsilon$ in the second inequality and \cref{eq:product mollifier has mass one} to obtain the final equality.

  \textit{Step 3 (convergence of approximations): } We now study the convergence of the sequence $(f_\varepsilon)_{\varepsilon>0}$ in $L^2_{per}(\Omega;(H^{1}_{per})'(0,2\pi))$. Recall that 
\begin{align*}
    \Vert f(\x,\cdot) - f_\varepsilon(\x,\cdot) \Vert_{(H^1_{per})'(0,2\pi)} &= \sup_{\overset{\varphi \in C^\infty_{per}(0,2\pi)}{\Vert \varphi \Vert_{H^1(0,2\pi)} \leq 1}} |\langle f(\x,\cdot) - f_\varepsilon(\x,\cdot) , \varphi \rangle|,
\end{align*}
for a.e.~$\x \in \Omega$, where the duality product above is that of $(H^1_{per})'(0,2\pi)$.
Note that, due to the regularity of $f_\varepsilon$, one may write explicitly 
\begin{equation}\label{eq:laying it out regularisation}
    \langle f(\x,\cdot) - f_\varepsilon(\x,\cdot), \varphi \rangle = \langle f(\x,\cdot) , \varphi \rangle - \int_0^{2\pi} \bigg(\int_\Omega \tilde{\beta}_{\varepsilon}(\x-\y)\langle f(\y,\cdot) , \tilde{\gamma}_\varepsilon(\theta-\cdot) \rangle \, \ud \y \bigg) \varphi(\theta) \, \ud \theta. 
\end{equation}
Using also the Fubini--Tonelli Theorem and the Dominated Convergence Theorem, a standard argument involving Riemann sums and the Arzel\`{a}--Ascoli Theorem shows that this latter term is equal to $\int_\Omega \tilde{\beta}_\varepsilon(\x-\y) \big\langle f(\y,\cdot),\int_0^{2\pi} \tilde{\gamma}_\varepsilon(\theta-\cdot)\varphi(\theta) \, \ud \theta \big\rangle \, \ud \y$. Using the equalities in \cref{eq:product mollifier has mass one}, \cref{eq:computing integral of gamma tilde}, and \cref{eq:equality translation only for periodic}, the first term on the right-hand side of \cref{eq:laying it out regularisation} can be written as $\langle f(\x,\cdot) , \varphi\rangle =  \int_\Omega   \tilde{\beta}_\varepsilon(\x-\y) \big\langle f(\x,\cdot) , \varphi(\cdot) \int_0^{2\pi}\tilde{\gamma}_\varepsilon(\theta)\, \ud \theta\big\rangle  \, \ud \y$ a.e.~$\x \in \Omega$. It therefore follows from \cref{eq:laying it out regularisation} that 
\begin{equation}\label{eq:splitting main for regularisation}
    \begin{aligned}
        \langle f(\x,\cdot) - f_\varepsilon(\x,\cdot), \varphi \rangle &= \int_\Omega \tilde{\beta}_\varepsilon(\x-\y) \big\langle f(\x,\cdot) - f(\y,\cdot), \varphi(\cdot) \int_0^{2\pi}\tilde{\gamma}_\varepsilon(\theta)\, \ud \theta\big\rangle  \, \ud \y \\ 
        &\!+\! \int_\Omega \tilde{\beta}_\varepsilon(\x\!-\!\y) \big\langle \! f(\y,\cdot) \!, \! \varphi(\cdot) \int_0^{2\pi}\tilde{\gamma}_\varepsilon(\theta)\, \ud \theta \!-\! \int_0^{2\pi} \tilde{\gamma}_\varepsilon(\theta\!-\! \cdot)\varphi(\theta) \, \ud \theta\big\rangle  \, \ud \y \\ 
        &=:I_\varepsilon(\x) + J_\varepsilon(\x). 
    \end{aligned}
\end{equation}
The term $I_\varepsilon$ may be rewritten as $I_\varepsilon(\x) = c_\gamma \int_\Omega \tilde{\beta}_\varepsilon(\x-\y) \langle f(\x,\cdot) - f(\y,\cdot), \varphi \rangle \, \ud \y$, where $c_\gamma = \int_0^{2\pi} \gamma(\theta) \, \ud \theta$. Using that $\Vert \varphi \Vert_{H^1_{per}(0,2\pi)} \leq 1$, we get 
\begin{equation}\label{eq:I eps estimate rookie}
    |I_\varepsilon(\x)| \leq c_\gamma \int_\Omega \tilde{\beta}_\varepsilon(\x-\y) \Vert f(\x,\cdot) - f(\y,\cdot) \Vert_{(H^1_{per})'(0,2\pi)} \, \ud \y, 
\end{equation}
however, for the time being, we concentrate on the term $J_\varepsilon$. We have 
\begin{equation}\label{eq:Jeps estimate rookie}
    |J_\varepsilon(\x)|\! \leq\! \int_\Omega \!\tilde{\beta}_\varepsilon(\x-\y)\Vert f(\y\!,\!\cdot)\Vert_{(H^1_{per})'(0,2\pi)} \big\Vert\! \int_0^{2\pi}\! \big( \tilde{\gamma}_\varepsilon(\theta) \varphi(\cdot) - \tilde{\gamma}_\varepsilon(\theta-\cdot)\varphi(\theta) \big) \ud \theta \big\Vert_{H^1(0,2\pi)} \ud \y, 
\end{equation}
and, in view of \cref{eq:computing integral of gamma tilde}, we notice that, for every $s\in\mathbb{R}$, $\int_0^{2\pi} \big( \tilde{\gamma}_\varepsilon(\theta) \varphi(s) - \tilde{\gamma}_\varepsilon(\theta-s)\varphi(\theta)\big) \ud \theta = -\int^{2\pi-s}_{-s} \tilde{\gamma}_\varepsilon (w) \big( \varphi(w+s) - \varphi(s) \big) \ud w$. Thus, using also the Dominated Convergence Theorem to differentiate under the integral in the previous right-hand side, we get, for every $s\in\mathbb{R}$, $\frac{d}{ds}\int_0^{2\pi} \big( \tilde{\gamma}_\varepsilon(\theta) \varphi(s) - \tilde{\gamma}_\varepsilon(\theta-s)\varphi(\theta)\big) \, \ud \theta = -\int^{2\pi-s}_{-s} \tilde{\gamma}_\varepsilon (w) \big( \varphi'(w+s) - \varphi'(s) \big) \, \ud w$, where we remark that the derivative contribution due to the integral limits vanishes due to the fact that $\tilde{\gamma}_\varepsilon$ and $\varphi$ are $2\pi$-periodic. It follows that 
\begin{equation}\label{eq:splitting H1per bound into k1 and k2}
    \begin{aligned}
        \big\Vert \int_0^{2\pi} \tilde{\gamma}_\varepsilon&(\theta) \varphi(\cdot) - \tilde{\gamma}_\varepsilon(\theta-\cdot)\varphi(\theta) \, \ud \theta \big\Vert_{H^1_{per}(0,2\pi)}^2 = \\  &\int_0^{2\pi} \big| \int^{2\pi-s}_{-s} \tilde{\gamma}_\varepsilon (w) \big( \varphi(w+s) - \varphi(s) \big) \, \ud w \big|^2 \, \ud s \\ 
        &+ \int_0^{2\pi} \big| \int^{2\pi-s}_{-s} \tilde{\gamma}_\varepsilon (w) \big( \varphi'(w+s) - \varphi'(s) \big) \, \ud w \big|^2 \, \ud s =: K^1_\varepsilon + K^2_\varepsilon. 
    \end{aligned}
\end{equation}
Using the Fundamental Theorem of Calculus along with the fact that $\Vert \varphi^{(m)}\Vert_{L^\infty(\mathbb{R})}<+\infty$ for any $m \in \mathbb{N} \cup \{0\}$ due to the smoothness and periodicity of $\varphi$, we get $| \int^{2\pi-s}_{-s} \tilde{\gamma}_\varepsilon (w) \big( \varphi(w+s) - \varphi(s) \big) \, \ud w | \leq \Vert \varphi'\Vert_{L^\infty(\mathbb{R})} \int_{-s}^{2\pi-s} \tilde{\gamma}_\varepsilon(w) |w| \, \ud w$, and (using the change of variables $\theta = w+s$) 
\begin{equation*}
    K^1_\varepsilon \leq \Vert \varphi'\Vert_{L^\infty(\mathbb{R})}^2 \int_0^{2\pi} \big| \int_{0}^{2\pi} \tilde{\gamma}_\varepsilon(\theta-s) |\theta-s| \, \ud \theta \big|^2 \, \ud s. 
\end{equation*}Our strategy is to split the previous right-hand side into three regions in order to bound it. While $s$ is constrained to the interval $[0,2\pi(1-\varepsilon^\alpha)]$, since $\theta \in [0,2\pi]$, we have that $-2\pi(1-\varepsilon^\alpha) \leq \theta -s \leq 2\pi$, in which case $\tilde{\gamma}_\varepsilon(\theta-s)$ is nonzero only when $\theta -s \in [0,2\pi\varepsilon^\alpha]$; we call this \emph{Region 1}.     On the other hand, when $s$ is constrained to the interval $[2\pi(1-\varepsilon^\alpha),2\pi]$, since $\theta \in [0,2\pi]$, we have that $-2\pi \leq \theta-s \leq 2\pi \varepsilon^\alpha$, in which case $\tilde{\gamma}_\varepsilon(\theta-s)$ is nonzero when $\theta -s \in [-2\pi,-2\pi(1-\varepsilon^\alpha)] \cup [0,2\pi \varepsilon^\alpha]$. We therefore have two more regions; \emph{Region 2}, where $s \in [2\pi(1-\varepsilon^\alpha),2\pi]$ and $\theta-s \in [0,2\pi\varepsilon^\alpha]$, and \emph{Region 3}, where $s \in [2\pi(1-\varepsilon^\alpha),2\pi]$ and $\theta-s \in [-2\pi,-2\pi(1-\varepsilon^\alpha)]$. It follows from the triangle inequality and the fact that the function $x \mapsto x^2$ is increasing on $[0,\infty)$ that $K^1_\varepsilon$ is bounded above by the sum of the following three terms, each corresponding to Regions 1, 2, and 3, respectively: 
\begin{equation*}
   \begin{aligned}
    &L^1_\varepsilon := 4\pi^2 \varepsilon^{2\alpha} \Vert \varphi'\Vert_{L^\infty(\mathbb{R})}^2 \int_0^{2\pi(1-\varepsilon^\alpha)} \big( \int_{0}^{2\pi} \tilde{\gamma}_\varepsilon(\theta-s) \, \ud \theta \big)^2 \, \ud s, \\ 
    &L^2_\varepsilon := 4\pi^2 \varepsilon^{2\alpha} \Vert \varphi'\Vert_{L^\infty(\mathbb{R})}^2 \int_{2\pi(1-\varepsilon^\alpha)}^{2\pi} \big( \int_{2\pi(1-\varepsilon^\alpha)}^{2\pi} \tilde{\gamma}_\varepsilon(\theta-s) \, \ud \theta \big)^2 \, \ud s, \\ 
    &L^3_\varepsilon := \Vert \varphi'\Vert_{L^\infty(\mathbb{R})}^2 \int_{2\pi(1-\varepsilon^\alpha)}^{2\pi} \big( \int_{0}^{2\pi\varepsilon^\alpha} \tilde{\gamma}_\varepsilon(\theta-s) (s-\theta)\, \ud \theta \big)^2 \, \ud s, 
   \end{aligned}
\end{equation*}
and it is straightforward to verify, using the nonnegativity of $\tilde{\gamma}_\varepsilon$ and \cref{eq:computing integral of gamma tilde}, the estimates $|L^1_\varepsilon|+|L^2_\varepsilon| \leq C\varepsilon^{2\alpha}$ and $|L^3_\varepsilon| \leq C\varepsilon^\alpha$, for some positive constant $C$ depending only on $c_\gamma$ and $\Vert \varphi'\Vert_{L^\infty(\mathbb{R})}$. It follows that, for a possibly larger constant $C$ depending only on $c_\gamma$ and $\Vert \varphi'\Vert_{L^\infty(\mathbb{R})}$, there holds $
    K^1_\varepsilon \leq C\varepsilon^\alpha$ and, by the same argument, the same bound holds for $K^2_\varepsilon$ (with the only change being that the constant $C$ also depends on $\Vert \varphi''\Vert_{L^\infty(\mathbb{R})}$). We therefore obtain from \cref{eq:splitting H1per bound into k1 and k2}, for a possibly larger constant $C$ depending only on $c_\gamma$, $\Vert \varphi'\Vert_{L^\infty(\mathbb{R})}$, and $\Vert \varphi''\Vert_{L^\infty(\mathbb{R})}$, there holds 
\begin{equation*}
    \begin{aligned}
        \big\Vert \int_0^{2\pi} \tilde{\gamma}_\varepsilon(\theta) \varphi(\cdot) - \tilde{\gamma}_\varepsilon(\theta-\cdot)\varphi(\theta) \, \ud \theta \big\Vert_{H^1_{per}(0,2\pi)}\leq C \varepsilon^{\alpha/2}. 
    \end{aligned}
\end{equation*}
Thus, returning to the bound \cref{eq:Jeps estimate rookie} on $|J_\varepsilon|$, taking supremum over test functions and integrating, followed by applying the Jensen inequality and using the monotonicity of the functions $x \mapsto x^2$ and $x \mapsto x^{1/2}$ on $[0,\infty)$, we get 
\begin{equation}\label{eq:Jeps bound very close}
    \Vert \sup_{\overset{\varphi \in C^\infty_{per}(0,2\pi)}{\Vert \varphi \Vert_{H^1(0,2\pi)} \leq 1}} J_\varepsilon \Vert_{L^2_{per}(\Omega)} \leq C_\Omega \varepsilon^{\alpha/2} \bigg( \int_\Omega \int_\Omega \tilde{\beta}_\varepsilon(\x-\y)^2\Vert f(\y,\cdot)\Vert_{(H^1_{per})'(0,2\pi)}^2  \, \ud \y \, \ud \x \bigg)^{1/2}, 
\end{equation}
where $C_\Omega$ is a positive constant depending only on $|\Omega|$, $c_\gamma$, $\Vert \varphi'\Vert_{L^\infty(\mathbb{R})}$, and $\Vert \varphi''\Vert_{L^\infty(\mathbb{R})}$. Using the Fubini--Tonelli Theorem to evaluate the double integral exactly, using also the same procedure as the one used for \cref{eq:equality translation only for periodic} to get $\int_\Omega \tilde{\beta}_\varepsilon(\x-\y)^2 \, \ud \x  = \varepsilon^{-2}\int_\Omega \beta(\y)^2 \, \ud \y$, 
\begin{equation}\label{eq:main Jeps estimate}
    \Vert \sup_{\overset{\varphi \in C^\infty_{per}(0,2\pi)}{\Vert \varphi \Vert_{H^1(0,2\pi)} \leq 1}} J_\varepsilon \Vert_{L^2_{per}(\Omega)} \leq C_\Omega \varepsilon^{\frac{\alpha}{2} - 1} \Vert f \Vert_{L^2_{per}(\Omega;(H^1_{per})'(0,2\pi))}, 
\end{equation}
for some new positive constant $C_\Omega$ which again only depends on $|\Omega|$, $c_\gamma$, $\Vert \varphi'\Vert_{L^\infty(\mathbb{R})}$, and $\Vert \varphi''\Vert_{L^\infty(\mathbb{R})}$. Notice that, since $\alpha>2$, the above vanishes in the limit as $\varepsilon \to 0$.

Finally, we return to the term $I_\varepsilon$. By taking the supremum over test functions and integrating, we find that $\Vert \sup_{\overset{\varphi \in C^\infty_{per}(0,2\pi)}{\Vert \varphi \Vert_{H^1(0,2\pi)} \leq 1}} I_\varepsilon \Vert_{L^2_{per}(\Omega)}$ is bounded above by 
\begin{equation*}
    \begin{aligned}
        c_\gamma \bigg( \int_\Omega \bigg| \int_\Omega \tilde{\beta}_\varepsilon(\x-\y) \Vert f(\x,\cdot) - f(\y,\cdot) \Vert_{(H^1_{per})'(0,2\pi)} \, \ud \y \bigg|^2 \, \ud \x \bigg)^{1/2} =: c_\gamma L_\varepsilon. 
    \end{aligned}
\end{equation*}
We rewrite the innermost integral above using the change of variables $\y \mapsto \z = \x-\y$. Note that $\z \in [-2\pi,2\pi]\times [-2\pi,2\pi]$ and, moreover, $\tilde{\beta}_\varepsilon(\z)$ is only nonzero when $\z \in \bigcup_{j=1}^4 A^j_\varepsilon$, where $A^1_\varepsilon = [-2\pi,-2\pi(1-\varepsilon)]\times [0,2\pi\varepsilon]$, $A^2_\varepsilon = [0,2\pi\varepsilon]\times[0,2\pi\varepsilon]$, $A^3_\varepsilon = [-2\pi,-2\pi(1-\varepsilon)]\times [-2\pi,-2\pi(1-\varepsilon)]$, $A^4_\varepsilon = [0,2\pi\varepsilon]\times[-2\pi,-2\pi(1-\varepsilon)]$. It therefore follows that, using the triangle inequality for the norm of $L^2(\Omega)$ in the variable $\x$, there holds the inequality 
\begin{equation*}
    \begin{aligned}
        L_\varepsilon &\leq \sum_{j=1}^4 \bigg( \int_\Omega \bigg| \int_{A^j_\varepsilon}  \tilde{\beta}_\varepsilon(\z) \Vert f(\x,\cdot) - f(\x-\z,\cdot) \Vert_{(H^1_{per})'(0,2\pi)} \, \ud \z \bigg|^2 \, \ud \x \bigg)^{1/2} =: \sum_{j=1}^4 L^j_\varepsilon. 
    \end{aligned}
\end{equation*}
Using the Minkowski integral inequality for finite measure spaces, we get 
\begin{equation}\label{eq:final bit before translation convergence}
    \begin{aligned}
        L^j_\varepsilon &\leq \int_{A^j_\varepsilon} \tilde{\beta}_\varepsilon(\z) \bigg( \int_\Omega \Vert f(\x,\cdot) - f(\x-\z,\cdot) \Vert_{(H^1_{per})'(0,2\pi)}^2 \, \ud \x \bigg)^{1/2} \, \ud \z. 
    \end{aligned}
\end{equation}

We now have the following claim, the proof of which we omit as it is a standard exercise. The proof relies on observing that $\mathbf{f}(\x)=\mathbf{f}(\x-\lim_{j\to\infty}\z_j)$ in $(H^1_{per})'(0,2\pi)$ for a.e.~$\x\in\Omega$ and then integrating in $\x$. In order to pass to the limit in $j$, one needs to use the strong continuity property of the shift operator in $L^2_{per}(\Omega;X)$, \textit{i.e.}, $\lim_{\mathbf{w} \to 0}\Vert \tau_{\mathbf{w}}\mathbf{f} - \mathbf{f} \Vert_{L^2_{per}(\Omega;X)} = 0$, where $\tau_\mathbf{w}\mathbf{f}=\mathbf{f}(\cdot+\mathbf{w})$. The proof of this latter result is also straightforward (relying on identifying a convenient dense subset and applying the Dominated Convergence Theorem) and so we omit it. 

\begin{claim}\label{claim:final translation convergence}
Let $f \in L^2_{per}(\Omega;(H^1_{per})'(0,2\pi))$. Suppose that $(\z_j)_{j\in\mathbb{N}}$ is a sequence of points in $\mathbb{R}^2$ converging towards the point $(2\pi m,2\pi n) \in \mathbb{R}^2$ for some fixed $m,n\in\mathbb{Z}$. Then, 
$\lim_{j\to\infty} ( \int_\Omega \Vert f(\x,\cdot) - f(\x-\z_j,\cdot) \Vert_{(H^1_{per})'(0,2\pi)}^2 \, \ud \x)^{1/2} = 0$. 
\end{claim}

By rewriting $\tilde{\beta}_\varepsilon(\z)$ in each of the squares $\{A^j_\varepsilon\}_{j=1}^4$ in terms of the original bump function $\beta$, using the previous claim, and applying the Dominated Convergence Theorem to the right-hand side of \cref{eq:final bit before translation convergence}, we get $\lim_{\varepsilon \to 0} L^j_\varepsilon = 0$ for $j=1,\dots,4$. Hence $\Vert \sup_{\overset{\varphi \in C^\infty_{per}(0,2\pi)}{\Vert \varphi \Vert_{H^1(0,2\pi)} \leq 1}} I_\varepsilon \Vert_{L^2_{per}(\Omega)} \to 0$ as $\varepsilon \to 0$. This, \cref{eq:main Jeps estimate}, and \cref{eq:splitting main for regularisation}, proves \cref{eq:regularised initial data-app}. 
\end{proof}

\section*{Acknowledgments}
This work was carried out while AE and SMS were postdoctoral researchers at FAU Erlangen-N\"{u}rnberg and University of Cambridge, respectively.

\bibliographystyle{siamplain}
\bibliography{active.bib}
\end{document}